\documentclass[twoside, reqno, 10pt]{amsart}

\usepackage[left=3.1cm, right=3.1cm, top=3cm, bottom=2.5cm]{geometry}

\usepackage{algorithm}
\usepackage{algpseudocode}
\usepackage{placeins}
\usepackage[colorlinks=true, pdfstartview=FitV, linkcolor=blue,citecolor=blue, urlcolor=blue]{hyperref}

\usepackage{xcolor}

\definecolor{labelkey}{rgb}{0,0,1}
\definecolor{Red}{rgb}{0.7,0,0.1}
\definecolor{Green}{rgb}{0,0.7,0}

\usepackage{amsfonts, amssymb, amsmath, amsthm, mathrsfs, bbm, cjhebrew, gensymb, textcomp, mathtools, dsfont, calligra, esint}

\usepackage[normalem]{ulem}

\usepackage{commath, setspace, subcaption, parcolumns, multirow, multicol, accents, comment, marginnote, verbatim, empheq, enumerate, stackrel, enumitem, float, physics, soul}

\usepackage[capitalize,nameinlink,noabbrev]{cleveref}

\usepackage{graphicx, graphics, epsfig, psfrag, tikz, tikz-cd,svg}

\usepackage{todonotes}
\usepackage{amssymb}

\usepackage{multicol}
\usepackage[normalem]{ulem}

\numberwithin{equation}{section}

\newtheorem{Thm}{Theorem}[section]
\newtheorem{Lem}[Thm]{Lemma}
\newtheorem{Prop}[Thm]{Proposition}
\newtheorem{Cor}[Thm]{Corollary}

\newtheorem{Rmk}[Thm]{Remark}

\newtheorem*{Thm*}{Theorem}

\DeclareMathOperator{\argmin}{argmin}

\newcommand{\al}{\alpha}

\newcommand{\de}{\delta}
\newcommand{\De}{\Delta}
\newcommand{\gam}{\gamma}

\newcommand{\ph}{\varphi}
\newcommand{\lam}{\lambda}

\newcommand{\kap}{\kappa}
\newcommand{\si}{\sigma}
\newcommand{\Si}{\Sigma}
\newcommand{\tht}{\theta}

\newcommand{\om}{\omega}
\newcommand{\Om}{\Omega}

\newcommand{\hnu}{\hat{\nu}}

\newcommand{\bnu}{\bar{\nu}}
\newcommand{\unu}{\underline{\nu}}

\newcommand{\NN}{\mathbb{N}}
\newcommand{\ZZ}{\mathbb{Z}}

\newcommand{\RR}{\mathbb{R}}

\newcommand{\TT}{\mathbb{T}}

\newcommand{\sA}{\mathscr{A}}
\newcommand{\sO}{\mathscr{O}}

\newcommand{\sN}{\mathscr{N}}
\newcommand{\sC}{\mathscr{C}}
\newcommand{\sM}{\mathscr{M}}

\newcommand{\sV}{\mathscr{V}}

\newcommand{\cA}{\mathcal{A}}

\newcommand{\sfE}{\mathsf{E}}
\newcommand{\sfJ}{\mathsf{J}}
\newcommand{\sfK}{\mathsf{K}}
\newcommand{\sfR}{\mathsf{R}}
\newcommand{\sfc}{\mathsf{c}}
\newcommand{\sfW}{\mathsf{W}}
\newcommand{\sfP}{\mathsf{P}}
\newcommand{\sfG}{\mathsf{G}}
\newcommand{\sfS}{\mathsf{S}}

\newcommand{\hv}{\hat{v}}

\newcommand{\hz}{\hat{z}}

\newcommand{\hV}{\hat{V}}

\newcommand{\bdy}{\partial}
\newcommand{\lb}{\langle}
\newcommand{\rb}{\rangle}

\newcommand{\goesto}{\rightarrow}
\newcommand{\smod}{\setminus}
\newcommand{\imb}{\hookrightarrow}

\newcommand{\lpp}{((}
\newcommand{\rpp}{))}
\newcommand{\no}[2]{\|#2\|_{#1}}

\usepackage{todonotes}
\usepackage{amssymb}

\usepackage{multicol}
\usepackage[normalem]{ulem}

\title[Excitation and Identifiability in the 2D Navier-Stokes equations]{Excitation and Identifiability in the 2D Navier-Stokes equations}
 \author{Vincent R. Martinez, Sarah Strikwerda, Xiang Wan}


\begin{document}

\begin{abstract}

Convergence in parameter estimation classically requires ``persistency of excitation," which is a
non-degeneracy condition on a trajectory-dependent signal. To the best of our knowledge, this paper develops the first such excitation theory for a nonlinear
PDE. In the context of identifying the unknown viscosity from spectral observations in the two-dimensional
Navier--Stokes equations for incompressible fluids, our excitation condition is computable, verifiable a priori, and sharp with respect to scaling. Our approach employs a data assimilation methodology to account for an unknown initial state and select candidate viscosities by minimizing the loss between the observed low-modes of the fluid velocity and low-modes of a nudging-based filter assimilating these observations. The main novelty of our framework
is to distinguish a new functional, $\sfW$, representing the work done by the filter's associated sensitivity variable on the enstrophy, around which our entire analysis is centered. We show that $\sfW$ is explicitly comparable to the observability Gramian of the associated Gauss--Newton iteration, and
subsequently establish the following dichotomy at every critical point of the observational
loss: either $\sfW$ exceeds a minimal excitation threshold, in which case the
candidate viscosity obeys an explicit error estimate that depends inversely on $\sfW$ and the observational density $N$, but directly on the error between initial conditions, or else the observations are quantitatively insensitive to parameter updates in a way that is detectable to the user. Our approach is energy-based, and therefore expected to be adaptable to many other nonlinear dissipative systems.

\end{abstract}

\maketitle

\vspace{-10pt}

\setcounter{tocdepth}{2}
\tableofcontents

\section{Introduction}\label{sect:intro}

A basic problem in the prediction of physical processes is that although one often has access to a concrete model for the phenomenon of interest, the datum required for its practical implementation is not fully available. In the context of fluid systems, this problem manifests in the form of not having complete access to the present state of the flow, that is, to the full spectrum of length scales of motion in order to initialize its governing equation, or not having access to sufficiently accurate numerical values for the physical parameters of the system to ensure fidelity of its numerical simulation. The absence of this knowledge is often compensated for by modeling the unresolved scales, as for instance accomplished by closure models in the study of turbulence, where parameters such as \textit{eddy viscosity} or \textit{closure coefficients} are introduced \cite{Smagorinsky1963, Ladyzhenskaya1967, Bardina_Ferziger_Reynolds_1980, Pope_2000_bible, Lesieur_Metais_Comte_2005_LES_book, Graham_Holm_Mininni_Pouquet_2007}. Since these parameters are not derived from first principles, they must necessarily be inferred from observations collected on the flow itself. This calibration problem has been at the center of many modern developments in the modeling of turbulence and geophysical phenomena \cite{DuraisamyIaccarinoXiao2019,SchneiderLanStuartTeixeira2017}. Moreover, in operational settings, the availability of data is usually confined to observing a single trajectory of the system sparsely in space and time. 

Classically, the problem of determining a full state of the system by partial observations was posed by Charney, Halem, and Jastrow in context of numerical weather prediction \cite{CharneyHalemJastrow1969}. In that seminal work, the basic issues brought upon by the sparse availability of data and how to assimilate them into the equations of motion were already fully recognized. Their work has since proved to be a formative one in the subsequent development of data assimilation methodology \cite{GhilMalanotteRizzoli1991,Daley1991,KalnayBook, EvensenVossepoelVanLeeuwen2022}. Nevertheless, in spite of decades of advances, the mathematical theory of data assimilation remains largely underdeveloped beyond the setting of linear dynamics, Gaussian noise, and finite-dimensional systems \cite{KailathSayedHassibi2000, ReichCotter2015, LawStuartZygalakisBook}; see \cite{HoffmannParkStuart2026} for very recent developments in this direction. In this paper, we address the problem of determining the parameters of a dynamical system by a time-series of observations from a data assimilation perspective, where the dynamical system is defined by a nonlinear partial differential equation modeling fluid flow. 

Historically, the seminal work of Astr\"om and Bohlin \cite{AstromBohlin1965} introduced the concept of \textit{identifiability} to study when parameters can be reconstructed from observations of a dynamical system in the context of discrete-time dynamical systems, as well as the notion of \textit{persistent excitation} as a condition for ensuring identifiability. Their ideas were later extended to continuous-time systems in \cite{ShimkinFeuer1987}, where necessary and sufficient conditions for persistent excitation were established. In this paper, we are interested in developing a persistent excitation-type theory for identifiability in nonlinear partial differential equations within a continuous data assimilation framework. 

Specifically, we develop this theory for the two-dimensional (2D) Navier-Stokes equations (NSE) with unknown kinematic viscosity $\nu$ over a periodic domain $\TT^2=[0,2\pi]^2$, where the velocity field is mean-free. The functional-theoretic formulation of this equation is given by
    \begin{align}\label{eq:nse}
        \bdy_tu+\nu Au+B(u,u)=f,
    \end{align}
where $f$ is divergence-free, $A=P_\si(-\De)$ denotes the Stokes operator, $P_\si$ is the Leray projection onto divergence-free vector fields, and
    \begin{align}\label{def:B}
        B(u,v)=P_{\si}(u\cdotp\nabla)v.
    \end{align}
The viscosity identification problem can then be formulated as follows: Suppose that the observation time-series $\sO_{N,\tau}:=\{P_Nu(t)\}_{0\leq t\leq \tau}$ is given, where $P_N$ denotes projection onto Fourier wavenumbers $|k|\leq N$, $\tau>0$ denotes the size of the time-window of observations, and $u$ satisfies \eqref{eq:nse} for some initial condition $u_0$ and viscosity $\nu$, both of which are \textit{a priori unknown}. Find the optimal value of $\hnu$ which best approximates the unknown value of $\nu$.

Since the initial condition is not known, we will adopt a continuous data assimilation approach which will allow us to construct an approximation of the true state through assimilation of the observations into the underlying dynamical model \eqref{eq:nse}. In particular, we will adopt a simple nudging-based approach: given a proxy $\hnu>0$ for the viscosity and the observation time-series $\sO_{N,\tau}$, consider the following system:
    \begin{align}\label{eq:nse:nudge}
        \bdy_tv+\hnu Av+B(v,v)=f-\mu P_Nv+\mu P_Nu,
    \end{align}
equipped with initial condition
    \begin{align}\label{def:v0}
        v(0)=P_Nu(0)+Q_Nv_0,
    \end{align}
where $Q_N=I-P_N$, for some divergence-free $v_0$ such that $\nabla v_0$ is square-integrable over $\Om$. Note that although we do not know the initial condition $u_0$ of \eqref{eq:nse}, we observe the value $P_Nu(0)$ through $\sO_{N,\tau}$. The term $-\mu P_Nv+\mu P_N$ serves to drive the process $v$ towards $u$, but only on the observed modes, while the scalar $\mu>0$ is an algorithmic parameter that controls the rate at which this is enforced. Note, moreover, that both $\mu$ and $v_0$ of \eqref{eq:nse:nudge} and \eqref{def:v0} are \textit{prescribed by the user}. Ultimately, the system \eqref{eq:nse:nudge}--\eqref{def:v0} produces an approximation of the reference state $u$ given the observations $\sO_{N,\tau}$. Going forward, we will make use of the shorthand $\sO=\sO_{N,\tau}$.

To study the inverse problem of identifying the unknown viscosity, we consider an optimization approach in conjunction with the nudging approach above. Given the observations $\sO$, we define the \textit{observational loss} over $[0,\tau]$ by
    \begin{align}\label{def:loss}
        \sfJ(\hnu;\tau):=\frac{1}2\int_{0}^{\tau}\no{L^2}{P_Nv(t;\hnu)-P_Nu(t)}^2\,dt,
    \end{align}
where $v(t,\hat{\nu})$ satisfies \eqref{eq:nse:nudge}, \eqref{def:v0}. In practice, one attempts to minimize $\sfJ$ with respect to $\hnu$ in order to identify an ``optimal" approximation to $\nu$, namely, to solve
    \begin{align}\label{eq:minimize}
        \hnu_*\in\sM_\tau:=\argmin_{\hnu\in\sA}\sfJ(\hnu;\tau)
    \end{align}
where $\sA\subset\RR$ denotes an \textit{admissible set of viscosities} and $\sM_\tau$ denotes the set of minimizers of $\sfJ$ over $\sA$. We assume that $\sA$ is convex and thus, an interval $\sA:=[\underline{\nu},\overline{\nu}]$, for some $\underline{\nu}<\overline{\nu}$. It will be convenient to introduce $\rho:=\frac{\bnu-\unu}{\unu}$, which represents the relative error between the upper and lower bounds of $\sA$. Note that $\rho$ is a non-dimensional quantity and is available to the user as a tunable parameter. For the remainder of the manuscript, we set
    \begin{align}\label{def:sA}
        \sA:=[\unu,(1+\rho)\unu].
    \end{align}
    
Next, given any $\hnu\in\sA$, we define the \textit{normal cone to $\sA$ at $\hnu$} as
    \begin{align}\label{def:cone}
        \sN_{\sA}(\hnu):=\{d\in\RR:d(\nu'-\hnu)\leq0,\ \text{for all}\ \nu'\in\sA\}.
    \end{align}
Using this notation, we have the following first order necessary optimality condition: If $\hnu_*\in\sA$ is a minimum value of $\sfJ(\hnu;\tau)$, then
    \begin{align}\label{eq:FONOC}
        -\frac{d}{d\hnu}\bigg{|}_{\hnu=\hnu_*}\sfJ(\hnu;\tau)\in\sN_{\sA}(\hnu_*).
    \end{align}
Indeed, \eqref{eq:FONOC} simply encodes the familiar assertion that the global minimizer of $\sfJ$ over $\sA$ must occur at either the endpoints of $\sA$ or critical points of $\sfJ$ in $\sA$. We then consider the set of all \textit{candidate minima} by defining
    \begin{align}\label{def:candidates}
        \sC_\tau:=\left\{\hnu\in\sA:-\frac{d}{d\hnu}\sfJ(\hnu;\tau)\in\sN_{\sA}(\hnu)\right\}.
    \end{align}
Note that $\sM_\tau\subset\sC_\tau\subset\sA$ by definition.

Several practical methods for solving optimization problems such as \eqref{eq:minimize} are available. In this direction iteration schemes are useful and have been a subject of recent interest in the context of present work, both theoretically and practically \cite{CarlsonHudsonLarios2020, MartinezMurriWhitehead2025, Martinez2022, NeweyWhiteheadCarlson2025, PachevWhiteheadMcQuarrie2022}. In \cite{CarlsonHudsonLarios2020}, the authors developed a data assimilation methodology for solving the inverse problem of viscosity identification in \eqref{eq:nse} given finite-dimensional observations of the flow field. Specifically, their computational studies indicated a stabilization of state errors between the assimilated variable (generated by \eqref{eq:nse:nudge} with an incorrect value, $\hnu$, of the viscosity) and the reference variable (given by a solution of \eqref{eq:nse}). Remarkably, these observations inspired them to propose the following iteration scheme:
    \begin{align}\label{def:CHL}
        \hnu_{n+1}=\hnu_n+\frac{\mu\|P_Nv(t_{n+1};\hnu_n)-P_Nu(t)\|_{L^2}^2}{(P_Nv(t_{n+1};\hnu_n)-P_Nu(t_{n+1}),Av(t_{n+1};\hnu_n))_{L^2}}.
    \end{align}
It was then demonstrated through computational experiments in \cite{CarlsonHudsonLarios2020} that \eqref{def:CHL} robustly exhibited convergence to the true viscosity value across a wide range of dynamical regimes. This led to the rigorous convergence analysis of \eqref{def:CHL} in \cite{Martinez2022}, which subsequently developed an identifiability condition based on the non-vanishing of the denominator of \eqref{def:CHL}. The scheme \eqref{def:CHL} was eventually realized as a particular Newton-type iteration scheme applied to the loss function \eqref{def:loss} in \cite{MartinezMurriWhitehead2025, NeweyWhiteheadCarlson2025}. 
In the studies of \cite{MartinezMurriWhitehead2025, Martinez2022} on the convergence analysis of these Newton-type iteration schemes, the role of the nudging parameter $\mu$ in controlling the state and model errors was rigorously identified to  depend inversely on $\mu$ (see also the recent work of \cite{CibikFangLaytonSiddiqua2026}). 

The work \cite{NeweyWhiteheadCarlson2025} goes further in a finite-dimensional setting and considers formal asymptotic expansions in $\mu^{-1}$ of solutions to the corresponding nudged equation. In doing so, they provide a compelling justification of the Gauss-Newton approximation to the Hessian of the loss functional $\sfJ$. In the context of the present paper, the Gauss-Newton algorithm can then be written as
    \begin{align}\label{def:GN}
        \hnu_{n+1}=\hnu_n-\frac{\int_0^\tau(P_Nv(t;\hnu)-P_Nu(t),P_N\bdy_{\hnu}v(t;\hnu))_{L^2}\, dt}{\int_0^\tau\|P_N\bdy_{\hnu}v(t;\nu)\|_{L^2}^2\,dt}.
    \end{align}
The denominator of \eqref{def:GN} is readily identified as the \textit{observability Gramian} corresponding to \eqref{def:loss}:
    \begin{align}\label{def:G}
        \sfG(\hnu;\tau):=\int_0^\tau(P_N\bdy_{\hnu}v(t;\hnu),P_N\bdy_{\hnu}v(t;\hnu))_{L^2}\,ds=\int_0^\tau\|P_N\bdy_{\hnu}v(t;\nu)\|_{L^2}^2\,dt,
    \end{align}
which in this case is simply a scalar ($1\times1$ matrix) since the only parameter to identify is the scalar viscosity. Indeed, a formal calculation quickly yields
    \begin{align}\label{eq:Hessian}
        \frac{d^2}{d\hnu^2}\sfJ(\hnu;\tau)=\sfG(\hnu;\tau)+\sfS(\hnu;\tau),
    \end{align}
where $\sfS(\hnu;\tau)=\int_0^\tau(P_Nv(t;\hnu)-P_Nu(t),P_N\bdy_{\nu}^2v(t;\hnu))_{L^2}\,dt$ denotes the \textit{second-order correction}. It is then classical \cite{DennisSchnabel1996} that convergence of the Gauss-Newton algorithm \eqref{def:GN} to a minimizer can be guaranteed in a sufficiently small neighborhood of a minimizer provided that $\sfG$ is non-degenerate and the second-order term $\sfS$ is small relative to $\sfG$. Subsequently, identifiability criteria have typically been developed through $\sfG$ by \textit{assuming its non-degeneracy}. Indeed, such a program was recently developed in  \cite{AhmadBiswasHoffman2026} for the Lorenz 63 system. One of the main goals of this paper is to develop an energy-based approach in which these types of hypotheses \textit{can be verified a priori}.

\subsection{Main Results} The main contribution of this paper is to identify a new quantity around which to center one's analysis, namely:
    \begin{align}\label{def:W}
        \sfW(\hnu;\tau):=\int_0^\tau(\De v(t;\hnu),\bdy_{\hnu}v(t;\hnu))\,dt.
    \end{align}
This quantity is readily identified as the  sensitivity of the total enstrophy of the \textit{assimilated variable} $v$. Indeed, upon integrating by parts, one has
    \begin{align}\label{def:Eh}
   \sfW(\hnu;\tau)=-\int_0^\tau(\nabla v(t;\hnu),\nabla \bdy_{\hnu}v(t;\hnu))\,dt=-\frac{d}{d\hnu}\sfE(\hnu;\tau),
    \end{align}
where $\sfE$ denotes the total enstrophy of the assimilated variable over $[0,\tau]$:
    \begin{align}\label{def:E}
    \sfE(\hnu;\tau)=\int_0^\tau\frac{1}2|\nabla v(t;\hnu)|^2\,dt.
    \end{align}
We show in \cref{sect:relation} that $\sfW$ can be directly related to $\sfG$ via the sensitivity balance:
    \begin{align}\label{eq:sensitivity:balance}
        \frac{1}2|\bdy_{\hnu}v(\tau)|^2+\hnu\int_0^\tau\|\bdy_{\hnu}v(t)\|_{L^2}^2\,dt+\mu\sfG(\hnu;\tau)=\int_0^\tau (B(\bdy_{\hnu}v,v),\bdy_{\hnu}v)\,dt+\sfW(\hnu;\tau).
    \end{align}
    
Through \eqref{eq:sensitivity:balance}, we are then able to show that within a certain regime of parameters $\mu,\hnu,N$, which we refer to as the \textit{detectability regime}, i.e.,
    \begin{align}\label{eq:detectability}
        \mu\gtrsim \unu N^2,\quad N\gtrsim \frac{U}{\unu},
    \end{align}
where $U$ is a measure of the size of the reference flow, one has
    \begin{align}\label{eq:GW:equiv}
        \frac{\mu}2\sfG(\hnu;\tau)\leq \sfW(\hnu;\tau)\leq \mu \sfG(\hnu;\tau)+O\left(\frac{V^2\tau}{\mu}\right),
    \end{align}
where $V$ is a measure of the size of the assimilated flow (see Proposition \ref{lem:W}), which can be shown to be bounded independently of $\mu$ (see Corollary \ref{cor:v:simplify}). In particular, $\sfW$ and $\sfG$ are comparable within the detectability regime. In fact, since $\sfG\geq0$, one immediately deduces from \eqref{eq:GW:equiv} that $\sfW\geq0$, which means that the total enstrophy of $v$ over $[0,\tau]$ is \textit{monotonically decreasing} in this regime. This is to say that when the system is observable, non-degeneracy of the Gramian is characterized by the strength of monotonicity of the total enstrophy.

The main achievement of centering our analysis around $\sfW$ is the identification of a (non-dimensional) threshold, $\om_*>0$, within the detectability regime,  such that for any $\hnu\in\sC_\tau$, one has
    \begin{align}
        (\hnu-\nu)^2\sfW(\hnu;\tau)\leq O\left(\frac{\eta_0^2}{N}\right),\quad& \sfJ(\hnu;\tau)\leq O\left(\frac{\eta_0^2}{N^2\mu}\right),& \textrm{if}\quad \sfW(\hnu;\tau)&\geq\om_*\tau\label{eq:main:informal:a}
        \\
        \frac{1}{\tau}\int_0^\tau\|\bdy_{\hnu}v(t;\hnu)\|_{L^2}^2&\,dt\leq O\left(\frac{\om_*}{\mu}\right),&\textrm{if}\quad \sfW(\hnu;\tau)&<\om_*\tau,\label{eq:main:informal:b}
    \end{align}
where $\eta_0=\|\nabla (v_0-u_0)\|_{L^2}$, and $\om_*$ is a non-dimensional quantity that depends (in an explicit way) on $\eta_0$, $\rho$, and the size of the reference velocity field, but is \textit{independent of} the observational resolution, $N$, and the particular value $\hnu\in\sC_\tau$. 
In other words,  $\om_*$ determines a \textit{minimal excitation threshold} above which recovery of the true viscosity can be quantified precisely, and below which the dynamics are measurably insensitive to changes in viscosity on average. Moreover, this result is \textit{global} in the sense that it holds across all candidate viscosities as characterized by $\sC_\tau$. This result is stated precisely in Theorem \ref{thm:main} and proved in \cref{sect:proof}; we introduce the precise functional setting, well-posedness results of \eqref{eq:nse}, \eqref{eq:nse:nudge}, and the associated sensitivity equation (see \eqref{eq:sensitivity}), in addition to stating various inequalities and mathematical notations that will be used to rigorously state and prove our results in \cref{sect:prelim}; the a priori estimates needed for our analysis are relegated to \cref{sect:a priori}.

One immediately derives several interesting consequences: (1) as long as the excitation threshold is met, full recovery of the viscosity is ensured in the limit of full spatial observation resolution, i.e., $N\goesto\infty$; (2) in fact, if one defines the set of ``excited viscosities" by
    \begin{align}\label{def:excited}
        \sV_\tau:=\{\hnu\in\sA:\sfW(\hnu;\tau)\geq\om_*\tau\},
    \end{align}
then, provided that $\sV_\tau$ is nonempty, \eqref{eq:main:informal:a} implies that, in the limit $N\goesto\infty$, $\sV_\tau\cap\sC_\tau$ must eventually collapse onto the singleton set $\{\nu\}$; (3) moreover, $\nu\in\sM_\tau$. 
 Note that \textit{it is not at all clear} that $\nu\in\sM_\tau$ a priori since we do not know the initial condition. Indeed, it follows immediately from \eqref{eq:nse:nudge} that $P_Nv(\cdot;v_0,\nu)=P_Nu(\cdot;u_0,\nu)$ if $v_0=u_0$, but they need not be the same if $v_0\neq u_0$; (4) the fact that in the minimally excited case \eqref{eq:main:informal:a}, $\nu\in\sM_\tau$ necessarily holds in the limit $N\goesto\infty$ \textit{for any} $\tau>0$, is also a non-trivial consequence. Indeed, if $\hnu=\nu$, then it is known from \cite{AzouaniOlsonTiti2014} that the detectability regime implies that $u$ is asymptotically reconstructed by $v$ as $t\goesto\infty$. Our result then implies that $\nu$ is indeed a minimizer \textit{without having to pass to the infinite $\tau$ limit}, provided that we are in the detectability regime, i.e., \eqref{eq:detectability} holds, and the minimal excitation regime, i.e., $\hnu\in\sV_\tau$ holds; (5) another consequence of \eqref{eq:main:informal:a} can be derived by introducing the quantity
    \begin{align}\label{def:P}
        \sfP(\hnu;\tau)=\frac{\sfW(\hnu;\tau)}{\tau},
    \end{align}
which represents the sensitivity of the \textit{average enstrophy} of the assimilated variable over $[0,\tau]$. The excitation threshold is equivalent to $\sfP(\hnu;\tau)\geq\om_*$ and \eqref{eq:main:informal:a} implies that as long as this threshold is met as the optimization window increases ($\tau\goesto\infty$), then full recovery is guaranteed in the infinite-optimization window limit. Since failure of the minimal excitation condition leads to insensitivity to changes in viscosity, as captured by \eqref{eq:main:informal:b}, one may therefore consider 
    \begin{align}\label{eq:identifiability:practical}
        \sfW(\hnu;\tau)\geq\om_*\tau\quad\text{or}\quad \sfP(\hnu;\tau)\geq\om_*,
    \end{align}
as \textit{practical identifiability criteria}, while
    \begin{align}\label{eq:identifiability:intro}
\liminf_{N\goesto\infty}\inf_{\hnu\in\sC_\tau}\sfW(\hnu;\tau)\geq\om_*\tau\quad\text{or}\quad \liminf_{\tau\goesto\infty}\inf_{\hnu\in\sC_\tau}\sfP(\hnu;\tau)\geq\om_*,
    \end{align}
are bona fide \textit{identifiability criteria}; (6) incidentally, since our analysis is valid for all $\sC_\tau$, one may also derive the following condition with a meaningful consequence:
    \begin{align}\label{eq:locatability}
        \sup_{\hnu\in\sC_\tau}\sfW(\hnu;\tau)\geq\om_*\tau 
            \quad \text{or} \quad 
        \sup_{\hnu\in\sC_\tau}\sfP(\hnu;\tau)\geq\om_*,
    \end{align}
or equivalently, that $\sV_\tau\cap\sC_\tau\neq\varnothing$. The condition \eqref{eq:locatability} indicates that as long as one candidate satisfies \eqref{eq:identifiability:intro}, then that candidate automatically satisfies \eqref{eq:main:informal:a}. Thus, one may think of \eqref{eq:locatability} as a ``locatability criterion," that is, as a way to quantifiably locate that one is near the true viscosity.

In \cref{sect:sharp}, we then study the sharpness of the minimal threshold quantified in $\om_*$ by considering small amplitude external forcing (\cref{sect:Grashof}) and Kolmogorov flows (\cref{sect:Kolmogorov}). In the regime of small amplitude forcing, it is known (see, for instance, \cite{DascaliucFoiasJolly2005}) that \eqref{eq:nse} has a unique, globally attracting stationary solution (Proposition \ref{prop:Grashof}). Subsequently, we show that the assimilating system and its associated sensitivity equation also possess unique globally attracting stationary states (Proposition \ref{prop:Grashof:nudge} and Proposition \ref{prop:Grashof:sensitivity}). We then show that the power functional, $\sfP_*$, associated to the respective stationary states, i.e., \eqref{def:P} associated to $v_*, \hv_*$, obey a lower bound which saturate minimal threshold  $\om_*$ up to a constant (Corollary \ref{cor:power:lower} and Proposition \ref{prop:Grashof:sharp}). This lower bound is afforded by obtaining a new coercive-type estimate on the linearized Navier-Stokes operator about $v_*$, perturbed by the nudging operator (Lemma \ref{lem:power:lower}); to the best of our knowledge, this estimate appears to be new and may find useful application in further studies related to nudging-based data assimilation schemes. 

In \cref{sect:Kolmogorov}, we consider the special case of Kolmogorov flows, where the forcing is given explicitly by a shear flow. Here, one obtains an exact stationary solution of \eqref{eq:nse}. In this setting, we are able to show that the lower bounds obtained in the small forcing regime (Corollary \ref{cor:power:lower}) are sharp up to an explicit prefactor. Moreover, we show that Theorem \ref{thm:main} is sharp in the sense that it is possible to violate the excitation threshold with a smaller prefactor in a way that nevertheless \textit{scales consistently} with $\om_*$. Indeed, in this scenario, we show that the time-average of the $L^2$-norm of the associated sensitivity variable is strictly smaller than $\om_*\mu^{-1}$, thus indicating that one is in the alternative of insensitivity indicated by \eqref{eq:main:informal:b}. Our analysis in \cref{sect:Kolmogorov} therefore establishes a legitimate path to understanding identifiability at the \textit{a priori} level through special forcing structures, which we believe to be deserving of further study in its own right. The reader is referred to various remarks in which we discuss in greater depth the assumptions of Theorem \ref{thm:main} (Remark \ref{rmk:compatibility}, Remark \ref{rmk:nudging}), stability with respect to numerical simulation (Remark \ref{rmk:stability}), and further implications of our study of the sharpness of the excitation threshold (Remark \ref{rmk:obs:id} and Remark \ref{rmk:BH}). We conclude this section by emphasizing that although the identifiability framework introduced here is established only in the case of a single unknown parameter, we believe it is an important step to isolate the difficulty presented by nonlinearity and infinite-dimensionality in an effort to expose the most salient features of the approach. The extension to the multi-parameter case remains a crucial direction of development, which is currently ongoing and reserved for a future work.

For the remainder of the introduction, we discuss how our results relate to those found in three distinct, but overlapping communities of Data Assimilation, Control Theory, and Inverse Problems, as well as attempt to interpret our result from these perspectives. Our hope is that the ensuing discussion will help to frame the discussion and interpretation of our results to a broad audience of readers.

\subsection{Data Assimilation, Control Theory, and Inverse Problems}

\subsubsection{Data assimilation}\label{sect:intro:DA}
The mathematical study of the state reconstruction algorithm \eqref{eq:nse:nudge} in the context of nonlinear PDEs (specifically \eqref{eq:nse}) was carried out independently by Azouani, Olson, Titi \cite{AzouaniOlsonTiti2014} and Bl\"omker, Law, Stuart, Zygalakis \cite{BlomkerLawStuartZygalakis2013}. From one perpsective, the present work can be viewed as an expansion of the theory developed in \cite{AzouaniOlsonTiti2014, BlomkerLawStuartZygalakis2013} to accommodate state and parameter estimation simultaneously. 
Since the works of \cite{BlomkerLawStuartZygalakis2013, AzouaniOlsonTiti2014}, many works have expanded the theory by overcoming difficulties that arise from the presence of other forcing mechanisms in the system \cite{Farhat_Jolly_Titi_2014, FarhatLunasinTiti2016b, FarhatLunasinTiti2016c, BiswasHudsonLariosPei17, JollyMartinezTiti2017}, from the manner of observation \cite{Foias_Mondaini_Titi_2016, FarhatLunasinTiti2016a, FarhatLunasinTiti2016a, BiswasHudsonLariosPei17, Jolly_Martinez_Olson_Titi_2018_blurred_SQG, Farhat_Johnston_Jolly_Titi_2018, BiswasBradshawJolly2021, Franz_Larios_Victor_2021, BiswasBrownMartinez2022}, or from the presence of model errors \cite{FarhatGlattHoltzMartinezMcQuarrieWhitehead2020, CarlsonHudsonLarios2020, CarlsonHudsonLariosMartinezNgWhitehead2021, PachevWhiteheadMcQuarrie2022, Martinez2022, AlbanezBenvenutti2024, FarhatLariosMartinezWhitehead2024, Martinez2024, CibikFangLaytonSiddiqua2026, BrockerCarigiKunaMartinez2025, BessaihFerrarioLandoulsiZanella2026}.

The heart of the analysis carried out in \cite{BlomkerLawStuartZygalakis2013, AzouaniOlsonTiti2014} can be traced back to \cite{OlsonTiti2003}. In \cite{OlsonTiti2003}, the role played by the \textit{existence of finitely many determining modes} in \eqref{eq:nse} (proved decades before by Foias and Prodi in \cite{FoiasProdi1967}) in ensuring synchronization of an elementary CDA algorithm was first discovered. 
The insight of \cite{OlsonTiti2003} proved to be consequential in the subsequent development of the mathematics of DA in the context of nonlinear PDEs as it provided, for the first time, a \textit{detectability criterion} for the state reconstruction property associated to \eqref{eq:nse}. Since then, many results in relation to a Foias-Prodi-type detectability criterion have been established such as long-time stability and accuracy of filtering algorithms \cite{BrettLamLawMcCormickScottStuart2012, Sanz-AlonsoStuart2015, OljacaBrockerKuna2018, Biswas_Branicki_2024}, and the interrelation between the nudging filter and several other filtering algorithms in CDA in \cite{Biswas_Branicki_2024, CarlsonFarhatMartinezVictor2024a, CarlsonFarhatMartinezVictor2024b}.
While it has been known since \cite{BiswasHudsonLariosPei17} that the ability of the nudging algorithm to synchronize to the true reference solution is logically stronger than the determining modes property, the converse statement (effectively conjectured by Olson and Titi in \cite{OlsonTiti2003}), had not been known. In the recent work \cite{CarlsonFarhatMartinezVictor2024a},  the converse was finally established by demonstrating that a ``lifted version" of \eqref{eq:nse:nudge} obeys the determining modes property and that this is sufficient to imply the synchronization property of the nudging algorithm. The main lesson from these results is that one may thus consider the Foias-Prodi determining modes property as a bona fide \textit{detectability criterion} that \textit{equivalently manifests} through the ability of the nudging algorithm to synchronize with the underlying reference state being observed. Throughout the paper, this detectability criterion is expressed as the condition that both the nudging parameter, $\mu$, and observation resolution, $N$, be \textit{sufficiently large}.

The criterion \eqref{eq:identifiability:practical}, \eqref{eq:identifiability}, and \eqref{eq:locatability} are new in the context of the 2D Navier-Stokes equations. Previous results \cite[Theorem 5.3]{BiswasHudson2023} had established identifiability criterion in terms of the distance to the global attractor in a setting where the global attractor is not allowed to contain the zero state. However, at the moment, whether or not the global attractor contains the zero state for a given time-independent body force \textit{remains a longstanding open problem}. Since the framework in \cite{BiswasHudson2023} is developed on the global attractor, it requires one to work with observations given in the form of bi-infinite time series, i.e., trajectories $\{P_Nu(t)\}_{t\in\RR}$. In contrast, our identifiability criteria are formulated over finite-time windows and effectively only on the observations over that window through the assimilated variable and its associated sensitivity. Our results can be viewed as a practical counterpart to the results developed in \cite{BiswasHudson2023}. 

Our work can also be viewed as an extension of the recent work \cite{AhmadBiswasHoffman2026}, which studied the same type of optimization approach \eqref{eq:minimize} considered here in the finite-dimensional setting of the Lorenz 63 system, to the infinite-dimensional setting of a nonlinear PDE \eqref{eq:nse}. However, we emphasize that the insight we hope one gleans from our analysis is that the observability Gramian, $\sfG$, can effectively be replaced by the work functional, $\sfW$, for the purpose of both theoretical study and practical implementation.

\subsubsection{Control theory}\label{sect:control}
The auxiliary system \eqref{eq:nse:nudge} can be viewed as a generalization of the observer design introduced by Luenberger \cite{Luenberger1971, Luenberger1964}, first on finite-dimensional linear systems and later extended to infinite dimensional systems, including PDE problems with unbounded boundary controls/boundary observations of both parabolic and hyperbolic types \cite{C-L-T1999, HendricksonLasiecka1995, JiLasiecka1998, RT-Wan2023luenberger}, \cite[pp 495--504]{La-T.vol1}. 

The classical observer design problem of a linear ODE system $\frac{dy}{dt} =Ly$, along with an output of the system $z=Hy$, is to construct an \textit{observer}, $\hat{y}$, which is a feedback-controlled copy of the system, by designing an \textit{observer gain operator} $K$. 
The goal of this construction is to make the dynamics of simulation errors globally exponentially stable across, that, across is all initial states. This naturally manifests as a spectral condition on $L-KH$. The existence of such gain operators is equivalent to the \textit{detectability} of the plant and output matrix pair.
We refer the reader to the original paper of Luenberger \cite{Luenberger1964} for the linear case, \cite{Zabczyk2020} for the rigorous mathematical development in both finite- and infinite-dimensions, and \cite{wan2026mathematical} for some recent progress on interacting linear systems.

In the setup considered in the present paper, the correspondence of the classical setup above with \eqref{eq:nse:nudge} can be made as follows: the reference dynamics are \eqref{eq:nse}, $L$ denotes the \textit{nonlinear operator} defined by $L(y) = -\nu Ay - B(y,y) + f$, and the output and gain operators, respectively, are defined by $H:=P_N$ and $K:=\mu H^*$.
Thus, the output injection term in \eqref{eq:nse:nudge} is $K(z-Hv)=\mu P_N(u-v)$ (recall that $P_N$ is an orthogonal
projection), and the output $z=P_Nu$ induces the observation time-series $\sO_{N,\tau}$.
The nudging parameter $\mu$ is the observer gain. As in the classical setting, the reference initial datum $u_0$ is unknown while $v(0)$ is prescribed by the user. The seminal work of \cite{AzouaniOlsonTiti2014} established the analogue of detectability: for the observer dynamics with the correct viscosity, the state error $w=v-u$ obeys
    \begin{align}\label{eq:observer:error}
        \bdy_tw + \nu Aw + B(w,w) + DB(u)w = -\mu P_Nw,\qquad w(0) = w_0.
    \end{align}
Then $|w(t)|^2 \leq e^{-\mu t}|w_0|^2$ for arbitrary $w_0$, provided $\mu$ and $N$ satisfy \eqref{eq:detectability} (see \eqref{est:w:L2a:general} of Lemma \ref{lem:w:L2} with $\De\nu=0$). Strictly speaking, these conditions on $\mu$ and $N$ should be interpreted as \textit{detectability criterion}. 

Luenberger theory was also extended to simultaneous state and parameter estimations where $L$ is linear, but $H$ is unbounded.
For instance, when $H$ is given by the boundary trace operator, one typically needs to establish delicate trace regularity properties, which is where the root of the difficulties lie (see \cite{curtain2020introduction, La-T.vol1, FanDiCristoJiangNakamura2010, ImanuvilovYamamoto2015, BiswasBradshawJolly2021}). 
\textit{Nonlinear} Luenberger observer theory for state and parameter estimation remains an ongoing area of study (see in \cite{afri2016state, Andrieu2014, AndrieuPraly2006, KazantzisKravaris1998} for recent works). However, results are mostly to be confined to finite-dimensional systems that are \textit{linear} in the state (but nonlinear in the extended state-parameter system), i.e., $\dot{x}=A(\tht)x+B(\tht)u$, $\dot{\tht}=0$, where $A(\tht),B(\tht)\in\RR^{n\times n}$ (see \cite{afri2016state}) or nonlinear, but \textit{autonomous}, i.e., $\dot{x}=f(x)$ (see \cite{AndrieuPraly2006}). 
In contrast, \eqref{eq:nse} is nonlinear in the state, non-autonomous, with $H=P_N$ is bounded and finite-rank, which necessitates a 
careful derivation of the corresponding claimed excitation threshold.

The problem studied here is also connected to optimal control and variational DA \cite{Lions1971,Troltzsch2010,HinzePinnauUlbrichUlbrich2009,MarchukAgoshkovShutyaev1996}. Indeed, the viscosity identification problem could also be formulated as an optimal control problem or as strong-constraint 4DVAR \cite{KalnayBook,LawStuartZygalakisBook, Korn2021}, by extending the system to include the parameter as in \cite{EvensenVossepoelVanLeeuwen2022}:
Find $\hat{u}_0^*, \hnu^*\in \mathscr{U}\times \mathscr{A}$ such that 
    $\sfK(\hat{u}_0^*,\nu^*)\leq \sfK(\hat{u}_0, \hnu)$, for all $(\hat{u}_0, \hnu) \in \mathscr{U} \times \mathscr{A}$,
where 
    \begin{equation} \label{def:K}
        \sfK(\hat{u}_0, \hnu)=\frac{1}{2}\int_0^T\| P_N u(t;\hat{u}_0, \hnu)-P_N u(t;u_0,\nu)\|^2_{L^2}\,dt +\frac{\lambda}{2}\|\hat{u}_0-P_Nu(0)\|_{L^2}^2,
    \end{equation}
$\mathscr{U}\subset H$ is understood to be an admissible set of initial conditions, and $\mathcal{A}$ is defined in \eqref{def:sA}. The second term in $\sfK$ is referred to as a \textit{Tikhonov regularization} in optimal control or \textit{background} in DA. 
It enforces smallness of the high-modes and bias to prior information. Without it, the  minimization of $\sfK$ can behave poorly numerically.

To minimize $\sfK$, one would use some iterative scheme that requires $\frac{d}{d\hat{u}_0}\sfK$ and $\frac{d}{d\hnu}\sfK$. In contrast, the derivative of $\frac{d}{d\hnu}\sfJ$  can be calculated in a straightforward way by solving the sensitivity equation (see Theorem \ref{thm:CL}). However, the derivative of $\frac{d}{d\hat{u}_0}\sfK$ is an infinite-dimensional calculation that could be infeasibly high-dimensional after discretizing. One could avoid computing $\frac{d}{d\hat{u}_0}\sfK$  by passing to the adjoint system. However, the adjoint system is formulated as a \textit{final-value problem}, so that each gradient computation requires one forward solve (across the entire time-horizon) and one backward solve (see for instance \cite{MarchukAgoshkovShutyaev1996}). It is now clear what one gains from employing a nudging-based DA methodology: in considering the loss function, $\sfJ$ in \eqref{def:loss}, in terms of \eqref{eq:nse:nudge}, rather than considering the loss function $\sfK$ in \eqref{def:K}, we are able to forgo a potentially costly optimization with respect to initial conditions by instead relying on the stabilizing dynamics of the nudged system, which is \textit{agnostic} to initialization. From this point of view, our excitation condition can be viewed as a result analogous to those identifying an appropriate choice of $\lam$ to guarantee convergence. A key difference between our excitation framework and the classical theory which derives conditions on $\lam$ (see, for instance, \cite{Morozov1966, EnglHankeNeubauer1996}) is that our conditions are formulated in terms of what is observed and what is known, i.e., observation time-series, nudging parameter, and equations, instead of the level of noise or conditions on the unobserved component of the initial data.

\subsubsection{Inverse problems}\label{sect:inverse}

Through Theorem \ref{thm:main} and Corollary \ref{cor:main}, we effectively establish a conditional well-posedness theory for the inverse problem. Indeed, standard well-posedness of the forward problem \eqref{eq:nse} guarantees that the parameter-to-observation map is well-defined (see Theorem \ref{thm:nse} below), while its injectivity is guaranteed when the minimal excitation condition \eqref{eq:identifiability} holds. We effectively ensure stability of the inverse map through \eqref{est:limit} in Theorem \ref{thm:main}.

The mechanism by which uniqueness is obtained in this paper is structurally parallel to the role played by unique continuation properties (UCP) that appears naturally in the study of domain-localized or boundary-localized stabilization problems. In this context, UCP asserts that vanishing of the solution to the linearized system on a subdomain of the interior or of the boundary implies the solution vanishes everywhere on the domain (see for instance \cite{FabreLebeau1996, BarbuLasieckaTriggiani2006, triggiani2009unique, TriggianiWan2021}). These ideas have been extended to the problem of identifying viscosity in, for instance, \cite{FanDiCristoJiangNakamura2010, ImanuvilovYamamoto2015}. In our case, the corresponding unique continuation statement would assert that $P_N\hv\equiv0$ implies $\hv\equiv0$. 
In Section \ref{sect:Kolmogorov}, we show that such a statement is, in general, \textit{false} when the observation operators are given as spectral projections rather than local interior restrictions or boundary trace operators. Nevertheless, our excitation framework can be viewed as a type of \textit{quantitative} unique continuation theory for spectral observations. In stabilization problems, UCP is what allows one to establish a Kalman-type rank condition which characterizes when finite-dimensional feedback controls can stabilize the unstable directions of the system. The analog of these conditions for the state estimation problem is effectively our \textit{detectability criterion}, while our excitation criterion is effectively its identifiability analog for the parameter estimation problem.

Lastly, our result is directly related to the recent work of \cite{Vieira2025}, which studied a viscosity identifiability in the setting of the \textit{steady} Navier-Stokes system. Their main results establish well-posedness and stability of the associated inverse problem and convergence properties of the optimization scheme with respect to the choice of discretization. We emphasize that their results are developed in the context of \textit{small-amplitude forcing}, which, as we mentioned previously, is known to guarantee the existence of a globally attracting steady state for the reference dynamics. This regime, which implicitly serves as as identifiability regime, is the one in which we demonstrate sharpness of our minimal excitation threshold. The results of \cite{Vieira2025} can be viewed as a type of numerical companion of the ones obtained in our paper under the constraint of small-amplitude forcing. Since our analysis is carried out at the level of the PDE and our results are quantitative, the effect of numerical discretization can be accounted for in a relatively straightforward way; see Remark \ref{rmk:stability} for more details.

\section{Mathematical Preliminaries}\label{sect:prelim}

For $1\leq p\leq\infty$, we denote by  $L^p=L^p(\TT^2)$ the space of real-valued, Borel measureable, $p$-integrable functions over $\TT^2$, which are $2\pi$-periodic a.e. in each direction $x_1,x_2$, equipped with norm:
    \begin{align}\label{def:Lp}
        |\phi|_p:=\begin{cases}\left(\int_{\TT^2}|\phi(x)|^p\,dx\right)^{1/p},&p\in[1,\infty)\\
        \sup_{x\in\TT^2}|\phi(x)|,&p=\infty.
        \end{cases}
    \end{align}
When $p=2$, we denote the $L^2$-inner product by: 
    \begin{align}\label{def:L2}
        (\phi,\psi):=\int_{\TT^2}\phi(x)\psi(x)\,dx.  
    \end{align}
Next, we define the subspace $Z\subset L^1$ of mean-free, Borel measureable functions over $\TT^2$ by:
    \begin{align}\label{def:Z}
        Z=\{\phi\in L^1:\int_{\TT^2} \phi(x)\,dx=0\}.
    \end{align}
For functions $\phi\in L^2\cap Z$, one has a well-defined Fourier representation:
    \begin{align}\label{def:Fourier}
        \phi(x)=\sum_{k\in\ZZ^2\smod\{(0,0)\}}\phi_ke^{ik\cdot x},\quad \phi_k:=\frac{1}{(2\pi)^2}\int_{\TT^2}\phi(x)e^{-ik\cdot x},
    \end{align}
where $\phi_{-k}=-\phi_k^*$, for all $k\in\ZZ^2\smod\{(0,0)\}$. For $\ell\in\NN_0:=\NN\cup\{0\}$, we denote the Sobolev spaces $H^\ell=H^\ell(\TT^2)$ to be the space of $\ell$-weakly differentiable, square-integrable functions over $\TT^2$ equipped with the norm
    \begin{align}\label{def:Hk}
        \|\phi\|_\ell^2:=\sum_{0\leq|\rho|\leq\ell}\int_{\TT^2}|\bdy^\rho\phi(x)|^2\,dx,
    \end{align}
for multi-indices $\rho\in\NN_0^2$, and partial derivative operator $\bdy^\rho=\bdy_1^{\al_1}\bdy_{2}^{\al_2}$. 
In particular, for all $\ell\geq0$, $|\bdy^\rho\phi|\leq\|\phi\|_\ell
$, for all $|\rho|<\ell$. For the remainder of the manuscript, we will adopt the slight abuse of notation of letting $L^p, H^\ell$ also denote their natural extensions to vector fields.

Let $P_\si:L^2\goesto L^2$ denote the Leray projection onto solenoidal (divergence-free) vector fields. Recall that 
    \begin{align}\label{est:P}
        |P_\si u|_2\leq|u|_2.    
    \end{align}
We denote by $H$ the subspace of $L^2$ consisting of mean-free and divergence-free vector fields, i.e. $H:=P_\si(L^2\cap Z)$. We then let $V=H^1\cap H$ with topology induced by
    \begin{align}\label{def:norm:V}
        \|u\|:=|\nabla u|_2.
    \end{align}
Next, we define the Stokes operator, $A$, by
    \begin{align}\label{def:A}
        A:=P_\si(-\De).
    \end{align}
Then, $A$ has domain $D(A)$ densely defined in $H^2\cap V$ and is self-adjoint with $A^{-1}$ compact in $H$. Note that in our setting of periodic boundary conditions, $A=P_\si(-\De)=(-\De)P_\si=-\De$. In particular, powers of $A$ are well-defined and one moreover has
    \begin{align}\label{equiv:Sob}
        c_\ell^{-1}\|u\|_\ell\leq|A^{\ell/2}u|_2\leq c_\ell
        \|u\|_\ell,
    \end{align}
for some $c_\ell\geq1$, for all $\ell\geq0$. Indeed, by \eqref{def:Fourier} and Parseval's identity, one has
    \begin{align}\notag
       |A^{\ell/2}\phi|_2^2=(2\pi)^2\sum_{k\in\ZZ^2\smod\{(0,0)\}}|k|^{2\ell}|\phi_k|^2
    \end{align}
for all $\ell\in\ZZ$. Note that in the special case $\ell=1$, one has
    \begin{align}\notag
        |A^{1/2}u|_2=|\nabla u|_2=\|u\|.
    \end{align}
Moreover, one has the Poincar\'e inequality:
    \begin{align}\label{est:Poincare}
        |A^{\ell'/2}\phi|_{2}\leq |A^{\ell/2}\phi|_2,\quad \ell'\leq \ell.
    \end{align}
Lastly, we denote the dual of $V$ by $V^*$, whose topology is induced by
    \begin{align}\label{def:norm:Vstar}
        \|u\|_*=|A^{-1/2}u|_2,
    \end{align}
and denote the duality pairing between $V^*$ and $V$ by $\lb\cdot,\cdot\rb$. Note that one then has the following continuous and compact embeddings:
    \begin{align}\label{eq:imbed}
        V^*\imb H\imb V.
    \end{align} 
For the remainder of the manuscript, it will be convenient to adopt the following notation:
    \begin{align}\label{eq:norm:convention}
        |u|:=|u|_{2}=\|u\|_0,\quad \|u\|:=|\nabla\phi|=|A^{1/2}\phi|,\quad \|u\|_*=|A^{-1/2}u|.
    \end{align}
For more details on the above functional setting, we refer the reader to the classical treatments \cite{ConstantinFoiasBook, TemamBook2001} of the theory of Navier-Stokes equations.

\subsection{Well-posedness}
First, we recall the solution theory of \eqref{eq:nse}. It will be convenient at this stage to introduce the trilinear functional:  
    \begin{align}\label{def:trilinear}
        b(u_1,u_2,u_3)=(B(u_1,u_2),u_3).
    \end{align}
We detail its properties later.

\begin{Thm}\label{thm:nse}
Let $\nu>0$, $T>0$, and $f\in L^\infty(0,T;V^*)$. Then for all $u_0\in H$, there exists a unique $u\in L^\infty(0,T;H)\cap L^2(0,T;V)$ such that $\frac{du}{dt}\in L^2(0,T;V^*)$, and
    \begin{align}\label{def:weak}
        \lb\frac{du}{dt},\ph\rb+\nu( A^{1/2}u,A^{1/2}\ph)+b(u,u,\ph)=(f,\ph),\quad \lim_{t\goesto0^+}(u(t),\ph)=(u_0,\ph),
    \end{align}
for a.e. $[0,T]$, for all $\ph\in V$, and $u$ satisfies the energy inequality:
    \begin{align}\label{est:energy:inequality}
        |u(t)|^2+\nu\int_0^t\|u(s)\|^2\,ds\leq |u_0|^2+\int_0^t(f(s),u(s))\, ds,
    \end{align}
for a.e. $0\leq s\leq t\leq T$. Moreover, if $f\in L^\infty(0,T;H)$ and $u_0\in V$, then $u\in C([0,T];V)\cap L^2(0,T;D(A))$, $\frac{du}{dt}\in L^2(0,T;H)$, and 
    \begin{align}\label{def:strong}
        \frac{du}{dt}+\nu Au+B(u,u)=f,\quad u(0)=u_0,
    \end{align}
holds a.e. $[0,T]$, and \eqref{est:energy:inequality} holds with equality for all $t\in[0,T]$.
\end{Thm}

We refer to any $u$ that satisfies \eqref{def:weak}, \eqref{est:energy:inequality} as a \textit{Leray-Hopf weak solution}, and any $u$ that satisfies \eqref{def:strong}, as a \textit{strong solution}. Since the external force, $f$, will typically be fixed in our analysis, we will denote the unique solution $u$ corresponding to initial data $u_0$, force $f$, and viscosity $\nu$  by
    \begin{align}\label{notation:u}
        u(t;u_0):=u(t;u_0,f,\nu)=:u(t;\nu),
    \end{align}
and when the context is clear, simply by $u$. For $f\in L(0,\infty;H^\ell\cap H)$, for $\ell\geq0$ or $f\in L(0,\infty:V^*)$ it will be convenient to introduce the following quantities:
    \begin{align}\label{def:Fk}
        F_\ell:=\sup_{t\geq0}|A^{\ell/2}f(t)|,\ \text{for}\ \ell\geq0,\quad\text{and}\quad F_*:=\sup_{t\geq0}\|f(t)\|_*.
    \end{align}
With this in hand, we note that one additionally has absorbing ball estimates for strong solutions of \eqref{eq:nse}: there exist $t_0>0$ and $t_1>0$ such that
    \begin{align}\label{est:abs:ball}
        \sup_{t\geq t_0}|u(t;u_0)|^2\leq \frac{2F_*^2}{\nu^2},\quad \sup_{t\geq t_1}\|u(t;u_0)\|^2\leq \frac{2F_0^2}{\nu^2}.
    \end{align}
In particular,
    \begin{align}\label{def:U0U1}
        U_0:=\sup_{t\geq0}|u(t;u_0)|,\quad U_1:=\sup_{t\geq0}\|u(t;u_0)\|<\infty.
    \end{align}
These bounds rely on the following important identities for the trilinear form, which we will make use of throughout the paper:
    \begin{align}\label{eq:energy:identity}
        b(u_1,u_2,u_3)=-b(u_1,u_3,u_2),
    \end{align}
which implies 
\begin{align}\label{eq:zeroid}b(u,v,v)=0,\end{align} and
    \begin{align}\label{eq:enstrophy:identity}
        b(u_2,u_1,Au_1)+b(u_1,u_2,Au_1)+b(u_1,u_1,Au_2)=0.
    \end{align}
Making use of \eqref{eq:enstrophy:identity} and an elementary bootstrap, one may also show that 
    \begin{align}\label{def:U2}
        U_2:=\sup_{t\geq0}|Au(t;u_0)|<\infty.
    \end{align}
We note that the explicit bound $U_2\leq O(F_2^2\nu^{-3})$ can be found in \cite{DascaliucFoiasJolly2005} and \cite{Martinez2022}, where the former establishes such a bound on the global attractor and the latter establishes the bound for initial data in $u_0\in D(A)$. Since $u(t;u_0)\in L^2(0,T;D(A))$ for any $u_0\in V$, it will suffice for the remainder of the manuscript to assume for the sake of convenience that $u_0\in D(A)$. As before, the reader is referred to \cite{ConstantinFoiasBook, TemamBook2001} for a  comprehensive development of the above results.

In \cite{AzouaniOlsonTiti2014}, the well-posedness of the nudging equation \eqref{eq:nse:nudge} was established for general interpolant observable operators. However, we point out that in the special case of observations being given by the spectral projection $P_N$, well-posedness is unconditional in $\mu, N$ due to orthogonality of the projection.

\begin{Thm}\label{thm:AOT}
Let $T>0$, $\hnu>0$, $\mu>0$, $N>0$, $f\in L^\infty(0,T;V^*)$ and a Leray-Hopf weak solution $u$ of \eqref{eq:nse} corresponding to $f$ be given. Then for all $v_0\in H$, there exists a unique $v\in L^\infty(0,T;H)\cap L^2(0,T;V)$ such that $\frac{dv}{dt}\in L^2(0,T;V^*)$, $\lim_{t\goesto0^+}(v(t),\ph)=(v_0,\ph)$, and
    \begin{align}\label{def:weak:nudge}
        \lb\frac{dv}{dt},\ph\rb+\hnu( A^{1/2}v,A^{1/2}\ph)+b(v,v,\ph)=(f,\ph)-\mu(P_Nv,\ph),\quad
    \end{align}
holds for a.e. $[0,T]$, for all $\ph\in V$, and $v$ satisfies the energy inequality:
    \begin{align}\label{est:energy:inequality:nudge}
        |v(t)|^2+\nu\int_0^t\|v(s)\|^2\,ds+\mu\int_0^t|P_Nv(s)|^2\,ds\leq |v_0|^2+\int_0^t(f(s)+\mu P_Nu(s),v(s))\, ds,
    \end{align}
for a.e. $0\leq s\leq t\leq T$. Moreover, if $f\in L^\infty(0,T;H)$, $u$ is a strong solution of \eqref{eq:nse}, and $v_0\in V$, then $v\in C([0,T];V)\cap L^2(0,T;D(A))$, $\frac{du}{dt}\in L^2(0,T;H)$, $v(0)=v_0$, and 
    \begin{align}\label{def:strong:nudge}
        \frac{dv}{dt}+\nu Av+B(v,v)=f-\mu P_Nv+\mu P_Nu,
    \end{align}
holds for a.e. $t\in[0,T]$, and \eqref{est:energy:inequality} holds with equality for all $t\in[0,T]$.
\end{Thm}

Similar to \eqref{notation:u}, we adopt the convention that
    \begin{align}\label{notation:v}
        v(t;v_0):=v(t;v_0,f,\sO,\mu,\hnu)=:v(t;\hnu),
    \end{align}
which we use interchangeably depending on whether we want to emphasize the dependence on $v_0$ or $\hnu$. Lastly, in \cite{CarlsonLarios2021}, the well-posedness of the sensitivity equations for $u$ and $v$ were studied. In particular, it was shown in \cite{CarlsonLarios2021} that solutions $u$ of \eqref{eq:nse} and $v$ of \eqref{eq:nse:nudge} can be correspondingly differentiated with respect to $\nu,\hnu$, and moreover, that $\bdy_\nu u$, $\bdy_{\hnu}v$ satisfy the corresponding system obtained by formally differentiating \eqref{eq:nse}, \eqref{eq:nse:nudge}, in $\nu,\hnu$, respectively. 

By differentiating \eqref{eq:nse:nudge} in $\hnu$, one obtains the sensitivity equation:
    \begin{align}\label{eq:sensitivity:intro}
        \bdy_t\hv+\hnu A\hv+ DB(v)\hv=-\mu P_N\hv-Av,\quad \hv(0)=0
     \end{align}
where $DB(u)$ denotes the Fr\'echet derivative of $B(u,u)$ with respect to $u$:
    \begin{align}\label{def:DB}
        DB(u)v:=B(v,u)+B(u,v).
    \end{align}

\begin{Thm}\label{thm:CL}
Let $T>0$, $\nu>0$, $\hnu>0$, $\mu>0$, $N>0$, and strong solutions $u,v$ of \eqref{eq:nse}, \eqref{eq:nse:nudge} corresponding to some $f\in L^\infty(0,\infty;H)$, be given. Then there exists a unique $\hv\in C([0,T];V)\cap L^2(0,T;D(A))$ such that $\frac{d\hv}{dt}\in L^2(0,T;H)$, and \eqref{eq:sensitivity:intro} holds. Moreover,
    \begin{align}\label{notation:hv}
        \hv=\bdy_{\hnu}v(\cdot;\hnu).
    \end{align}
\end{Thm}

Note that under Theorem \ref{thm:CL}, one also has, in analogy to Theorem \ref{thm:nse} and Theorem \ref{thm:AOT}, the following energy balance relation:
    \begin{align}\label{eq:energy:inequality:sensitivity}
        |\hv(t)|^2+\hnu\int_0^t\|\hv(s)\|^2\,ds+\mu\int_0^t|P_N\hv(s)|^2\,ds=-\int_0^t(Av(s),\hv(s))\, ds-b(\hv,v,\hv).
    \end{align}
Although \eqref{eq:energy:inequality:sensitivity} is neither stated nor exploited in \cite{CarlsonLarios2021}, its proof follows that of Theorem \ref{thm:nse} and Theorem \ref{thm:AOT}. As suggested by \eqref{eq:sensitivity:balance}, the sensitivity energy balance \eqref{eq:energy:inequality:sensitivity} is central to our analysis.

With Theorem \ref{thm:nse}, Theorem \ref{thm:AOT}, Theorem \ref{thm:CL} in hand, our formal calculations from \cref{sect:intro} and below are fully justified within the functional setting we are working.

\subsection{Inequalities}
We will make use of several elementary or well-known functional inequalities. We use the Bernstein inequalities for estimating functions with frequency cut-offs:
    \begin{align}\label{est:Bernstein}
        \abs{Q_N\phi}\leq \frac{1}{N}\|Q_N\phi\|,\qquad \|P_N\phi\|\leq N|P_N\phi|.
    \end{align}
We also make use of the following interpolation inequalities of Ladyzhenskaya and Agmon:
    \begin{align}\label{est:interpolation}
        \abs{\phi}_4^2&\leq c_L\|\phi\||\phi|,\quad \abs{\phi}_\infty\leq c_A|A\phi|^{1/2}|\phi|^{1/2}.
    \end{align}
We will have the elementary interpolation inequality
    \begin{align}\label{est:interpolation:elementary}
        |A^{m/2}\phi|\leq |A^{n/2}\phi|^{\frac{m-\ell}{n-\ell}}|A^{\ell/2}\phi|^{\frac{n-m}{n-\ell}}, \quad \ell\leq m\leq n.
    \end{align}
We will also make use of the Br\'ezis-Gallouet inequality in 2D:
    \begin{align}\label{est:BrezisGallouet}
        \abs{\phi}_\infty&\leq c_{BG}\|\phi\|\left[1+\log\left(\frac{|A\phi|}{\|\phi\|}\right)\right]^{1/2},
    \end{align}
for some constant $c_{BG}>0$, for all $\phi \in D(A)$. In particular, if $\phi=P_N\phi$, then \eqref{est:Bernstein} implies
    \begin{align}\label{est:BrezisGallouet:localized}
    \abs{\phi}_\infty&\leq c_{BG}(1+\log N)^{1/2}\|\phi\|.
    \end{align}
With H\"older's inequality, \eqref{est:interpolation}, and \eqref{eq:energy:identity}, one has the following bound on the trilinear term:
    \begin{align}
        |b(u_1,u_2,u_3)|&\leq c_L\|u_1\|^{1/2}|u_1|^{1/2}\min\left\{\|u_2\|\|u_3\|^{1/2}|u_3|^{1/2},\|u_2\|^{1/2}|u_2|^{1/2}\|u_3\|\right\}\label{est:trilinear:a}
        \\
        |b(u_1,u_2,u_3)|&\leq c_A\|u_2\|\min\left\{|Au_1|^{1/2}|u_1|^{1/2}|u_3|,|u_1||Au_3|^{1/2}|u_3|^{1/2}\right\}.\label{est:trilinear:b}
    \end{align}

\section{Model Error Analysis}\label{sect:error}
In this section, we develop the framework for studying the model error that is centered around $\sfP$. To this end, let \begin{align}\label{def:state:model:error}
        w:=v-u,\quad \De\nu:=\hnu-\nu,\quad v=v(\cdot\ ;\hnu).
    \end{align}
We refer to $w$ as the \textit{state error} and $\De\nu$ as the \textit{model error}. In identifying the system that governs the evolution of the state error, let us first observe that the difference, $\hnu Av-\nu Au$, may be written as 
    \begin{align}\label{eq:diff:equiv}
        \hnu Av-\nu Au=\De\nu Av+\nu Aw=\hnu Aw+\De\nu Au
    \end{align}
We see that the state error is governed simultaneously by
    \begin{align}\label{eq:state:error}
        \bdy_tw+\nu Aw+B(w,w)+DB(u)w=-\mu P_Nw-(\De\nu) Av,\quad w(0)=w_0,
    \end{align}
and
    \begin{align}\label{eq:state:error:equiv}
        \bdy_tw+\hnu Aw+B(w,w)+DB(u)w=-\mu P_Nw-(\De\nu) Au,\quad w(0)=w_0.
    \end{align}
Ultimately, it will be useful to take advantage of both representations of the evolution of $w$.

Next, let us identify the sensitivity equation. In what follows, it will be expedient to introduce the following shorthand:
    \begin{align}\label{def:shorthand}
        \hv:=\bdy_{\hnu}v. 
    \end{align}
Then 
    \begin{align}\label{eq:sensitivity}
        \bdy_t\hv+\hnu A\hv+ DB(v)\hv=-\mu P_N\hv-Av,\quad \hv(0)=0.
    \end{align}
Upon taking the inner product of \eqref{eq:state:error} with $\hv$, we obtain
    \begin{align}
        (\bdy_tw+\nu Aw+B(w,w)+DB(u)w,\hv)=-(\De\nu)(Av,\hv)-\mu(P_Nw,\hv)\notag.
    \end{align}
Integrating over $[0,\tau]$, invoking \eqref{def:W}, and orthogonality of $P_N$, we see that
\begin{align}\label{eq:model:error:rep}
        (\De\nu)\sfW(\hnu;\tau)-\mu&\int_0^{\tau}(P_Nw,P_N\hv)\, dt\notag
        \\
        &=\int_{0}^{\tau}(\bdy_tw+\nu Aw+B(w,w)+DB(u)w,{\hv})\, dt.
    \end{align}

\begin{Rmk}\label{rmk:work:power}
Note that within our setting, which ultimately retains \textit{time} as the only physical dimensional, both  $\hv$ and  $\sfW$ are \textit{non-dimensional}, while $\sfP$, (see \eqref{def:P}) has the dimensions of $\textit{time}^{-1}$. Since $-Av$ appears on the right-hand side of \eqref{eq:sensitivity:intro} as a forcing term, $\sfW$ physically represents the \textit{total work} done by $-Av$ on $\hv$ over $[0,\tau]$, while $\sfP$ represents the \textit{average power} injected into $\hv$ by $-Av$ over $[0,\tau]$. Thus, both $\sfW$ and $\sfP$ form natural measures for tracking the impact of changes to $\hnu$ on the assimilated variable $v$. 
A power-type quantity was also identified in the convergence analysis of the iterative scheme \eqref{def:CHL} carried out in \cite{Martinez2022} through the derivative of an energy, which was crucial to obtaining the identifiability criterion there, though its relation to $\sfG$ is not clear. In kind, $\sfW$ and $\sfP$ hold a distinguished place in the present analysis, but have the added benefit of being directly related to $\sfG$.
\end{Rmk}

For the remainder of the manuscript, we will make use of the shorthand
    \begin{align}\label{def:short}
        \sfJ=\sfJ(\hnu;\tau),\quad \sfW=\sfW(\hnu;\tau),\quad\sfP=\sfP(\hnu;\tau).
    \end{align}
Next, we develop a more refined representation of the model error that will be useful for our analysis.

\subsection{Model Error: Refined Representation}
Since our ultimate goal is obtain a quantitative estimate on the model error $|\De\nu|$, we will develop \eqref{eq:model:error:rep} further to obtain a more useful form. To begin, we recall \eqref{eq:model:error:rep} and see that
    \begin{align}\label{eq:model:error:rep:expand}
        \left(\De\nu\right)&{\sfW}-\mu\int_0^\tau(P_Nw,P_N\hv)\, dt\notag
        \\
        &=\int_{0}^{\tau}(\bdy_tw,\hv)\, dt+\nu\int_{0}^{\tau} (Aw, \hv)\, dt+\int_{0}^{\tau}b(w,w,\hv)\, dt +\int_{0}^{\tau}(DB(u)w,{\hv})\, dt.
    \end{align}
Inspecting the first term on the right-hand side, we may integrate by parts in time and invoke \eqref{eq:sensitivity} and the fact that $\hv(0)=0$ to obtain
    \begin{align}\label{eq:model:error:rep:ibp}
        &\int_{0}^{\tau}(\bdy_tw,\hv)\, dt-\mu\int_0^\tau(P_Nw,P_N\hv)\, dt\notag
        \\
        &=(w,{\hv})\big{|}_{t=0}^{t=\tau}-\int_{0}^{\tau}({w},\bdy_t\hv)dt-\mu\int_0^\tau(P_Nw,P_N\hv)\, dt\notag
        \\
        &=(w(\tau),{\hv(\tau)})+\int_{0}^{\tau}({w},\hnu A\hv+ DB(v)\hv+Av)dt\notag
        \\
        &=(w(\tau),{\hv(\tau)})+\hnu\int_0^\tau(w,A\hv)\,dt+\int_0^\tau(w,DB(v)\hv)\,dt+\int_0^\tau(w,Av)\, dt.
    \end{align}
Note that we have made use of the fact that $P_N$ is an orthogonal projection. We now turn our attention to the last term on the right-hand side of \eqref{eq:model:error:rep:ibp}.
Observe that upon taking the $L^2$ inner product of \eqref{eq:state:error} with $w$, we obtain
    \begin{align}
        &(\De\nu)(Av,w)=-\frac{1}2\frac{d}{dt}|w|^2-\nu\|w\|^2-b(w,u,w)-\mu|P_Nw|^2.\notag
    \end{align}
Note that we have made use of the identities \eqref{eq:energy:identity} and \eqref{eq:zeroid}. Upon integrating over $[0,\tau]$, we then see that
    \begin{align}\label{eq:ibp:last}
        &(\De\nu)\int_0^\tau(w,Av)\,dt\notag
        \\
        &=-\frac{|w(\tau)|^2}{2}+\frac{|w(0)|^2}{2}-\nu\int_0^\tau\|w\|^2\, dt-\int_0^\tau b(w,u,w)\, dt-\mu\int_0^\tau|P_Nw|^2\, dt.
    \end{align}
Thus, upon multiplying \eqref{eq:model:error:rep:expand} and \eqref{eq:model:error:rep:ibp} by $\De\nu$, then adding the results and invoking \eqref{eq:ibp:last}, we arrive at
    \begin{align}\label{eq:model:error:rep:expand:final}
     (\De\nu)^2\sfW&+\frac{|w(\tau)|^2}{2}+\nu\int_0^\tau\|w\|^2\,dt+\mu\int_0^\tau|P_Nw|^2\, dt\notag
     \\
     &=(\De\nu)\sfR+\frac{|w(0)|^2}{2}-\int_0^\tau b(w,u,w)\,dt+2\mu(\De\nu)\int_0^\tau(P_Nw,P_N\hv)\, dt,
    \end{align}
where
    \begin{align}\label{def:residual}
        \sfR=\sfR_1+\sfR_2+\sfR_3+\sfR_4+\sfR_5+\sfR_6,
    \end{align}
and
    \begin{align}
        \sfR_1&=(w(\tau),{\hv(\tau)}),&\sfR_2&=\hnu\int_0^\tau(w,A\hv)\,dt\notag
        \\
        \sfR_3&=\nu\int_0^\tau(Aw,\hv)\,dt,&\sfR_4&=\int_0^\tau(DB(u)w,\hv)\,dt\notag
        \\
        \sfR_5&=\int_0^\tau b(w,w,\hv)\,dt,&\sfR_6&=\int_0^\tau(w,DB(v)\hv)\,dt\notag.
    \end{align}
We conclude this section with a lemma that will be central to our analysis. To state the lemma, it will be useful to introduce the following quantity:
    \begin{align}\label{def:U1}
        U_1(\tau):=\sup_{t\in[0,\tau]}\|u(t)\|.
    \end{align}
Note that since $u$ is given, we will often conflate $U_1(\tau)$ with $U_1$ from \eqref{def:U0U1}.
\begin{Lem}\label{lem:model:error:basic}
Suppose that 
    \begin{align}\label{cond:mu}
        \mu\geq \frac{c_L^2U_1^2}{\nu},\quad N\geq \frac{c_LU_1}{\nu}
    \end{align}
Then for all $\hnu\in\sC_\tau$, it holds that
    \begin{align}\label{est:model:error:basic}
        (\De\nu)^2\sfW+\mu\sfJ\leq(\De\nu)\sfR+\frac{|w(0)|^2}{2},
    \end{align}
\end{Lem}

\begin{proof}
Recall the notation \eqref{def:shorthand}. A direct calculation  yields
    \begin{align}\notag
        \frac{d}{d\hnu}\sfJ=\int_{0}^{\tau}(P_Nw,P_N\hv)\, dt.
    \end{align}
Thus, by orthogonality of $P_N$ and the fact that $\De\nu=\hnu-\nu$, for $\hnu\in\sC_\tau$, \eqref{eq:FONOC} is equivalent to
    \begin{align}\
          (\De\nu)\int_0^\tau(P_Nw, P_N\hv)\ dt\leq0,\quad \text{for all}\ \nu\in\sA.\notag
    \end{align}
In particular,
    \begin{align}\label{eq:nonnegative}
        2\mu(\De\nu)\int_0^\tau(P_Nw,P_N\hv)\,dt\leq0.
    \end{align}
Next, we estimate $\int_0^\tau b(w,u,w)\,dt$. By H\"older's inequality, \eqref{est:interpolation}, Young's inequality, and \eqref{est:Bernstein}, we obtain
    \begin{align}
        \int_0^\tau|b(w,u,w)|\,dt&\leq \int_0^\tau|w|_4^2\|u\|\,dt\leq c_L\int_0^\tau\|w\||w|\|u\|\,dt\notag
        \\
        &\leq c_LU_1\left(\int_0^\tau\|w\|^2\,dt\right)^{1/2}\left(\int_0^\tau|w|^2\,dt\right)^{1/2}\notag
        \\
        &\leq \frac{c_L^2U_1^2}{2\nu}\int_0^\tau|w|^2\,dt+\frac{\nu}2\int_0^\tau\|w\|^2\,dt\notag
        \\
        &= \frac{c_L^2U_1^2}{2\nu}\int_0^\tau|P_Nw|^2\,dt+\frac{c_L^2U_1^2}{2\nu}\int_0^\tau|Q_N
        w|^2\,dt+\frac{\nu}2\int_0^\tau\|w\|^2\,dt\notag
        \\
        &= \frac{c_L^2U_1^2}{2\nu}\int_0^\tau|P_Nw|^2\,dt+\left(\frac{c_L^2U_1^2}{2\nu N^2}+\frac{\nu}2\right)\int_0^\tau\|w\|^2\,dt.\notag
    \end{align}
It then follows from \eqref{cond:mu} that
    \begin{align}\label{est:model:error:trilinear}
         \int_0^\tau|b(w,u,w)|\,dt\leq \frac{\mu}2\int_0^\tau|P_Nw|^2\,dt+\nu\int_0^\tau\|w\|^2\,dt.
    \end{align}
We then combine \eqref{eq:model:error:rep:expand:final}, \eqref{eq:nonnegative}, \eqref{est:model:error:trilinear} with $|w(\tau)|^2\geq0$ and $\frac{\mu}2\int_0^\tau|P_Nw|^2\,dt=\mu\sfJ$ to conclude \eqref{est:model:error:basic}.
\end{proof}

\subsection{Relation between Work Functional and Observability Gramian}\label{sect:relation} We now show that $\sfW$ is well-defined, that under a certain parameter regime, $\sfW\geq0$, and establish its connection with the observability Gramian $\sfG$ (see \eqref{def:G}). First, we introduce the following quantity:
    \begin{align}\label{def:V1}
    V_1(\tau)=\sup_{t\in[0,\tau]} \|v(t)\|.
    \end{align}
In what follows, it will be useful to adopt the more general of view of $\sfW$  as a mapping $(v_1,v_2)\mapsto\sfW(v_1,v_2)$:
    \begin{align}\label{def:W:mapping}
        \sfW(v_1,v_2):=-\int_0^\tau (Av_1(t),v_2(t))\,dt
    \end{align}
In the context of \eqref{eq:minimize}, we will subsequently abuse notation and interchangeably write
    \[
    \sfW=\sfW(\hnu;\tau)=\sfW(v(\cdot;\hnu),\hv(\cdot;\hnu)).
    \]

\begin{Prop}\label{lem:W}
Given $\tau>0$, $\sfW$ is well-defined as a mapping $(v_1,v_2)\in L^2(0,\tau;D(A))\times L^2(0;\tau;H)\mapsto W(v_1,v_2)$, and satisfies
    \begin{align}\label{est:W:upper}
        |\sfW(v_1,v_2)| \leq  \left(\int_0^\tau|Av_1(t)|^2\,dt\right)^{1/2}\left(\int_0^\tau|v_2(t)|^2\,dt\right)^{1/2}.
    \end{align}
Given $\mu,\hnu>0$, suppose that $v, \hv$ are Leray-Hopf weak solutions of \eqref{eq:nse:nudge}, \eqref{eq:sensitivity}, respectively. If $\mu,\hnu, N$ satisfy
    \begin{align}\label{cond:mu:hnu:N:W}
           \max\left\{\frac{4c_L^2V_1^2}{\mu},\frac{2c_LV_1}{N}\right\}\leq\hnu
    \end{align}
then
    \begin{align}\label{est:W:G}
        \frac{1}{2}|\hv(\tau)|^2+\frac{\hnu}2\int_0^\tau\|\hv(t)\|^2\,dt+\frac{\mu}2\sfG\leq \sfW\leq\mu\sfG+\left(\frac{1}{\mu}+\frac{2}{\hnu N^2}\right)\int_0^\tau|Av(t)|^2\,dt.
    \end{align}
If we additionally assume that
    \begin{align}\label{cond:mu:hnu:N:W:plus}
        \mu\leq \frac{1}2\hnu N^2,
    \end{align}
then
    \begin{align}\label{est:W:plus}
        \frac{1}{2}|\hv(\tau)|^2+\frac{\hnu}4\int_0^\tau\|\hv(t)\|^2\, dt+\frac{\mu}2\int_0^\tau|\hv(t)|^2\,dt\leq\sfW\leq  \mu\sfG+\frac{2}{\mu}\int_0^\tau|Av(t)|^2\,dt.
    \end{align}
\end{Prop}

\begin{proof}
The upper bound \eqref{est:W:upper} follows from a direct application of the Cauchy-Schwarz inequality.

For the lower bound, we study the energy balance of \eqref{eq:sensitivity}:
\begin{align}
        \frac{1}2\frac{d}{dt}|\hv|^2+\hnu\|\hv\|^2+\mu|P_N\hv|^2=-b(\hv,v,\hv)-(Av,\hv)\notag.
    \end{align}
Observe that by Holder's inequality, \eqref{est:interpolation}, and Young's inequality 
we have
    \begin{align}
        |b(\hv,v,\hv)|&\leq c_L\|\hv\||\hv|\|v\|\leq \frac{\hnu}4\|\hv\|^2+\frac{2c_L^2}{\hnu}\|v\|^2|\hv|^2.\notag
    \end{align}
By orthogonality of $P_N$ and the Bernstein inequality in \eqref{est:Bernstein}, we furthermore have
    \begin{align}
        \frac{2c_L^2}{\hnu}\|v\|^2|\hv|^2&\leq \frac{2c_L^2}{\hnu}\|v\|^2|P_N\hv|^2+\frac{2c_L^2}{\hnu}\|v\|^2|Q_N\hv|^2\notag
        \\
        &\leq \mu\left(\frac{2c_L^2V_1^2}{\hnu\mu}\right)|P_N\hv|^2+\frac{2c_L^2}{\hnu N^2}V_1^2\|\hv\|^2.\notag
    \end{align}
Combining the above estimates and invoking \eqref{cond:mu:hnu:N:W}, it follows that
    \begin{align}\label{eq:hv:balance:mid}
        \frac{1}2\frac{d}{dt}|\hv|^2+\frac{\hnu}2\|\hv\|^2+\frac{\mu}2|P_N\hv|^2\leq-(Av,\hv).
    \end{align}
Thus, upon integrating over $[0,\tau]$ and using the fact that $\hv(0)=0$, we arrive at
    \begin{align}
        -\int_0^\tau(Av,\hv)dt\geq \frac{1}2|\hv(\tau)|^2+\frac{\hnu}2\int_0^\tau\|\hv\|^2dt+\frac{\mu}2\int_0^\tau|\hv|^2dt,\notag
    \end{align}
which implies the lower bound in \eqref{est:W:G}.

For the upper bound, we start from \eqref{est:W:upper}, invoke orthogonality of $P_N, Q_N$, then apply \eqref{est:Bernstein}, the lower bound in \eqref{est:W:G},  followed by Young's inequality to estimate
    \begin{align}
        \sfW&\leq \left(\int_0^\tau|Av(t)|^2\,dt\right)^{1/2}\left(\int_0^\tau|\hv(t)|^2\,dt\right)^{1/2}\notag
        \\
        &\leq\left(\int_0^\tau|Av(t)|^2\,dt\right)^{1/2}\left(\int_0^\tau|P_N\hv(t)|^2\,dt+\int_0^\tau|Q_N\hv(t)|^2\,dt\right)^{1/2}\notag
        \\
        &\leq\left(\int_0^\tau|Av(t)|^2\,dt\right)^{1/2}\left(\sfG+\frac{1}{N^2}\int_0^\tau\|\hv(t)\|^2\,dt\right)^{1/2}\notag
        \\
        &\leq\left(\int_0^\tau|Av(t)|^2\,dt\right)^{1/2}\left(\sfG+\frac{2}{N^2\hnu}\sfW\right)^{1/2}\notag
        \\
        &\leq \frac{\mu}2\sfG+\frac{1}{2}\sfW+\left(\frac{1}{2\mu}+\frac{1}{N^2\hnu}\right)\int_0^\tau|Av(t)|^2\,dt.\notag
    \end{align}
It then follows that
    \begin{align}\notag
        \frac{1}2\sfW\leq \frac{\mu}2\sfG+\left(\frac{1}{2\mu}+\frac{1}{N^2\hnu}\right)\int_0^\tau|Av(t)|^2\,dt,
    \end{align}
which implies the upper bound in \eqref{est:W:plus}.

Next, if we further assume \eqref{cond:mu:hnu:N:W:plus}, then
    \begin{align}
        \mu|P_N\hv|^2&=\mu|\hv|^2-\mu|Q_N\hv|^2\notag
        \\
        &\geq\mu|\hv|^2-\frac{\mu}{N^2}\|\hv\|^2\geq\mu|\hv|^2-\frac{\hnu}2\|\hv\|^2.\notag
    \end{align}
Applying this in \eqref{eq:hv:balance:mid} yields
    \begin{align}\notag
     \frac{1}2\frac{d}{dt}|\hv|^2+\frac{\hnu}4\|\hv\|^2+\frac{\mu}2|\hv|^2\leq-(Av,\hv).
    \end{align}
Integrating over $[0,\tau]$, applying the upper bound in \eqref{est:W:G}, then applying \eqref{cond:mu:hnu:N:W:plus} once again, yields \eqref{est:W:plus}.
\end{proof}

When $v, \hv$ are fixed Leray-Hopf weak solutions of \eqref{eq:nse:nudge}, \eqref{eq:sensitivity} (corresponding to $\mu, \hnu$), respectively, we may then  deduce an upper bound with explicit dependence on the nudging parameter $\mu$.

\begin{Cor}\label{cor:W}
Suppose that \eqref{cond:mu:hnu:N:W} and \eqref{cond:mu:hnu:N:W:plus} hold. Then
    \begin{align}\label{est:W:cor}
        \sfW\leq \frac{2}{\mu}\int_0^\tau|Av(t;\hnu)|^2\,dt.
    \end{align}
\end{Cor}
\begin{proof}
By \eqref{est:W:plus}, we have
    \begin{align}\notag
        \int_0^\tau
        |\hv|^2\, dt\leq \frac{2}{\mu}\sfW.
    \end{align}
By \eqref{est:W:upper}, it then follows that
    \begin{align}\notag
       \sfW\leq \frac{\sqrt{2}}{\sqrt{\mu}}\left(\int_0^\tau|Av|^2dt\right)^{1/2}\sfW^{1/2}.
    \end{align}
Therefore
    \begin{align}\notag
        \sfW^{1/2}\leq \frac{\sqrt{2}}{\sqrt{\mu}}\left(\int_0^\tau|Av|^2dt\right)^{1/2},
    \end{align}
which implies \eqref{est:W:cor}.
\end{proof}

\subsection{Residual Estimates}
We now estimate each of the residual terms \eqref{def:residual} that appear in \eqref{est:model:error:basic}. We make use of Lemma \ref{lem:W} to bound the sensitivity variable $\hv$ and a priori estimates that we develop later in \cref{sect:a priori}. 

For the remainder of the section, let us therefore assume that \eqref{cond:mu:hnu:N:W} and \eqref{cond:mu:hnu:N:W:plus} hold, so that by Lemma \ref{lem:W}, it holds that
     \begin{align}\label{est:hv:summary}
        |\hv(\tau)|^2\leq 2\sfW,
        \qquad
        \int_0^\tau\|\hv\|^2dt\leq\frac{2}{\hnu}\sfW,
        \qquad
        \int_0^\tau|\hv|^2dt\leq\frac{2}{\mu}\sfW.
    \end{align}
For the state errors, we introduce the following notation
    \begin{align}\label{est:state}
     (\hnu+\nu)\int_0^\tau\|w\|^2dt+\mu\int_0^\tau|w|^2dt\leq \Si.
    \end{align}
The results of \cref{sect:a priori} establish that $\Si<\infty$ and obtain an explicit form for $\Si$ whose dependence will be important later, but for now it will be useful to only make use of \eqref{est:state} in order to expose the dependence on the parameters $\hnu,\mu$ in the estimates. The form of $\Si$ and the assumptions required for \eqref{est:state} to hold will be revealed in \cref{sect:proof} below where we prove Theorem \ref{thm:main}.

\begin{Lem}\label{lem:R1}
Recall that $\sfR_1=(w(\tau),{\hv(\tau)})$. Then
    \begin{align}\label{est:R1}
        |\sfR_1|&\leq 2^{1/2}|w(\tau)|\sfW^{1/2}.
    \end{align}
\end{Lem}

\begin{proof}
 By the Cauchy-Schwarz inequality and \eqref{est:hv:summary}, we have
    \begin{align}
        |\sfR_1|&\leq|w(\tau)||\hv(\tau)|\leq 2^{1/2}|w(\tau)|\sfW^{1/2}\notag,
    \end{align}
as claimed.
\end{proof}

\begin{Lem}\label{lem:R2}
Recall that $\sfR_2=\hnu\int_{0}^{\tau}({w}, A\hv)\,dt$. We have
    \begin{align}\label{est:R2}
        |\sfR_2|
        \leq2^{1/2}\Si^{1/2}\sfW^{1/2}.
    \end{align}
\end{Lem}

\begin{proof}
By self-adjointness of $A$, the Cauchy-Schwarz inequality, and \eqref{est:hv:summary}, we have
    \begin{align}
        |\sfR_2|&\leq \left(\hnu\int_0^\tau\|w\|^2dt\right)^{1/2}\left(\hnu\int_0^\tau\|\hv\|^2dt\right)^{1/2}
        \leq 2^{1/2}\Si^{1/2}\sfW^{1/2}\notag.
    \end{align}
This is precisely \eqref{est:R2}. 
\end{proof}

\begin{Lem}\label{lem:R3}
Recall that $\sfR_3=\nu\int_0^\tau(Aw,\hv)\, dt$. We have
    \begin{align}\label{est:R3}
        |\sfR_3|&\leq 2^{1/2}\left(\frac{|\De\nu|}{\hnu}+1\right)\Si^{1/2}\sfW^{1/2}.
    \end{align}
\end{Lem}

\begin{proof}
First observe that
    \begin{align}\notag
        \sfR_3=-(\De\nu)\int_0^\tau(Aw,\hv)\,dt+\sfR_2.
    \end{align}
Then by the Cauchy-Schwarz inequality and \eqref{est:hv:summary}, we have
    \begin{align}
        |\De\nu|\int_0^\tau|(Aw,\hv)|\, dt&\leq \frac{|\De\nu|}{\hnu}\left(\hnu\int_0^\tau\|w\|^2\,dt\right)^{1/2}\left(\hnu\int_0^\tau\|\hv\|^2\,dt\right)^{1/2}\notag
        \\
        &\leq 2^{1/2}\frac{|\De\nu|}{\hnu}\Si^{1/2}\sfW^{1/2}\notag.
    \end{align}
Adding the result from Lemma \ref{lem:R2} yields \eqref{est:R3}.
\end{proof}

\begin{Lem}\label{lem:R4}
Recall that $\sfR_4=\int_0^\tau (DB(u)w,\hv)\, dt$. We have
    \begin{align}\label{est:R4}
         |\sfR_4|&\leq   2^{3/2}c_A\frac{U_2^{1/2}U_0^{1/2}}{\hnu^{1/2}\mu^{1/2}}\Si^{1/2}\sfW^{1/2}
    \end{align}
\end{Lem}

\begin{proof}
We first observe that \eqref{eq:energy:identity}, $(DB(u)w,\hv)=b(u,w,\hv)-b(w,\hv,u)$. Then by H\"older's inequality, \eqref{est:interpolation}, \eqref{est:hv:summary}, and \eqref{est:state} it follows that
    \begin{align}
        \int_0^\tau|b(u,w,\hv)|\, dt&\leq c_A\int_0^\tau|Au|^{1/2}|u|^{1/2}\|w\||\hv|\,ddt\notag
        \\
        &\leq c_AU_2^{1/2}U_0^{1/2}\left(\int_0^\tau\|w\|^2\,dt\right)^{1/2}\left(\int_0^\tau|\hv|^2\,dt\right)^{1/2}\notag
        \\
        &\leq 2^{1/2}c_A\frac{U_2^{1/2}U_0^{1/2}}{\hnu^{1/2}\mu^{1/2}}\Si^{1/2}\sfW^{1/2}.\notag
    \end{align}
On the other hand, we may similarly estimate to obtain
    \begin{align}
        \int_0^\tau|b(w,\hv,u)|\,dt& \leq c_A\int_0^\tau|w|\|\hv\||Au|^{1/2}|u|^{1/2}\,dt\notag\\
        &\leq c_AU_2^{1/2}U_0^{1/2}\left(\int_0^\tau|w|^2\,dt\right)^{1/2}\left(\int_0^\tau\|\hv\|^2\,dt\right)^{1/2}\notag
        \\
        &\leq  2^{1/2}c_A\frac{U_2^{1/2}U_0^{1/2}}{\hnu^{1/2}\mu^{1/2}}\Si^{1/2}\sfW^{1/2}.\notag
    \end{align}
Adding the results yields \eqref{est:R4}.
\end{proof}

\begin{Lem}\label{lem:R5}
Recall that $\sfR_5=-\int_0^\tau b(w,w,\hv)\,dt$. We have
    \begin{align}\label{est:R5}
        |\sfR_5|
        &\leq \frac{2^{1/2}c_L}{\hnu}\left(\sup_{t\in[0,\tau]}|w(t)|\right)\Si^{1/2}\sfW^{1/2}.
    \end{align}
\end{Lem}

\begin{proof}
By \eqref{eq:energy:identity} we have $b(w,w,\hv)=-b(w,\hv,w)$. By H\"older's inequality, \eqref{est:interpolation},  \eqref{est:hv:summary}, and \eqref{est:state} we have
    \begin{align}
        \int_0^\tau|b(w,\hv,w)|\,dt&\leq c_L\int_0^\tau\|w\||w|\|\hv\|\,dt\notag
        \\
        &\leq c_L\left(\sup_{t\in[0,\tau]}|w(t)|\right)\left(\int_0^\tau\|w\|^2\,dt\right)^{1/2}\left(\int_0^\tau\|\hv\|\,dt\right)^{1/2}\notag
        \\
        &\leq \frac{2^{1/2}c_L}{\hnu}\left(\sup_{t\in[0,\tau]}|w(t)|\right)\Si^{1/2}\sfW^{1/2},\notag
    \end{align}
as claimed.
\end{proof}

\begin{Lem}\label{lem:R6}
Recall that $\sfR_6=\int_0^\tau(w,DB(v)\hv)\,dt$. We have
    \begin{align}\label{est:R6}
        |\sfR_6|
        &\leq \left[\frac{2^{1/2}c_L}{\hnu}\left(\sup_{t\in[0,\tau]}|w(t)|\right)+2^{3/2}c_A\frac{U_2^{1/2}U_0^{1/2}}{\hnu^{1/2}\mu^{1/2}}\right]\Si^{1/2}\sfW^{1/2}
    \end{align}
\end{Lem}

\begin{proof}
First observe that by \eqref{eq:energy:identity} and the fact that $w=v-u$, we have
    \begin{align}
    (w,DB(v)\hv)&=b(v,\hv,w)+b(\hv,v,w)\notag
    \\
    &=b(w,\hv,w)+b(u,\hv,w)+b(\hv,u,w)=-b(w,w,\hv)+b(u,\hv,w)-b(\hv,w,u)\notag.
    \end{align}
Thus
    \begin{align}\notag
        \sfR_6=\sfR_5-\int_0^\tau\left(b(u,w,\hv)+b(\hv,w,u)\right)\,dt.
    \end{align}
Observe that the second grouping of terms can be estimated exactly as in Lemma \ref{lem:R4}. With Lemma \ref{lem:R5}, we therefore obtained the desired result.
\end{proof}

\subsection{Main Theorem}\label{sect:proof}

The main result of this paper is the following:

\begin{Thm}\label{thm:main}
There exists $\mu_*>0$ such that for all   
    \begin{align}\label{cond:main}
        \mu_*\leq\mu\leq \frac{1}4\unu N^2,
    \end{align}
there exists $\tau_*>0$ (depending only on $\mu$ and inversely) and $\om_*>0$ (depending only on $\rho$, the size of the reference flow $u$, and
inversely on $\mu$) such that for all $\tau\geq\tau_*$ and $\hnu\in\sC_\tau$, one has the following dichotomy: If $\sfW\geq\om_*\tau$, then
    \begin{align}\label{est:limit}
       (\De\nu)^2\sfW\leq O\left(\frac{\eta_0^2}{N^2}\right),
    \end{align}
and
    \begin{align}\label{est:loss}
        \sfJ\leq O\left(\frac{\eta_0^2}{N^2\mu}\right).
    \end{align}
whenever $\|v_0-u_0\|\leq\eta_0$. Otherwise, $\sfW<\om_*\tau$ and
    \begin{align}\label{est:insensitive}
        \frac{1}{\tau}\int_0^\tau|\bdy_{\hnu}v(t,\hnu)|^2\, dt\leq O\left(\frac{\om_*}{\mu}\right).
    \end{align}
\end{Thm}

We can identify $\mu_*$ and $\tau_*$ explicitly and we do so in the proof of Theorem \ref{thm:main} (see \eqref{cond:mu:summary} and \eqref{def:tau} below). Upon recalling the definition of the power functional, $\sfP=\frac{1}\tau\sfW$, (see \eqref{def:P}), one immediately deduces the following result from Theorem \ref{thm:main}:

\begin{Cor}\label{cor:main}
Within the regime  $\sfP\geq\om_*$ in Theorem \ref{thm:main}, given any sequence $\{\hnu_\tau\}\subset\sC_\tau$ such that
    \begin{align}\label{eq:identifiability}
        \liminf_{\tau\goesto\infty}\sfP\geq\om_*,
    \end{align}
it then follows that
    \begin{align}\notag
        \lim_{\tau\goesto\infty}|\hnu_\tau-\nu|=0.
    \end{align}
\end{Cor}

Before we proceed to the proof of Theorem \ref{thm:main}, we first recall that $\sA=[\unu,(1+\rho)\unu]$, for some $\rho\geq 0$. Since we assume that $\nu,\hnu\in\sA$, we see that
    \begin{align}\label{est:al}
        \frac{|\De\nu|}{\hnu}\leq \frac{\rho\unu}{\hnu}\leq \rho.
    \end{align}
Next, we recall that we assume $P_Nv(0)=P_Nu(0)$ (see \eqref{def:v0}). In particular, $P_Nw_0=0$. Thus, if we assume that
    \begin{align}\label{est:eta0}
        \|w_0\|\leq\eta_0,
    \end{align}
then by \eqref{est:Bernstein}, we always have   
    \begin{align}\label{est:w0}
        |w_0|^2&\leq\frac{1}{N^2}\|w_0\|^2\leq \frac{\eta_0^2}{N^2}.
    \end{align} 
In proving Theorem \ref{thm:main}, it will also be convenient to introduce the following quantities:
    \begin{align}\label{def:V0:V2}
        V_j:=\sup_{t\in[0,\tau]}\|v(t;v_0)\|_j,\quad j=0,1,2,\quad \lb V_2\rb=\left(\fint_0^\tau|Av(t;v_0)|^2\,dt\right)^{1/2}.
    \end{align}
These quantities are guaranteed to be finite provided that certain conditions on $\mu,\hnu,\nu, N$ are satisfied. We summarize these conditions, in conjunction with those identified in Lemma \ref{lem:model:error:basic} and Lemma \ref{lem:W}. Indeed, let us assume that
    \begin{align}\label{cond:mu:summary}
      \mu\geq  \max\left\{\frac{18c_L^2\max\{U_1^2, V_1^2\}}{\unu},\frac{F_0}{U_0},\frac{F_1}{U_1},\frac{F_2}{U_2}\right\}=:\mu_*,
    \end{align}
and     
    \begin{align}\label{cond:mu:N:summary}
       \mu \leq \frac{1}4\unu N^2
    \end{align}
Under these conditions Lemma \ref{lem:model:error:basic}, Lemma \ref{lem:W}, Lemma \ref{lem:v:H1}, Lemma \ref{lem:v:H2}, Corollary \ref{cor:v:simplify}, and Lemma \ref{lem:w:L2} all hold. In particular, $V_0, V_1, V_2, \lb V_2\rb$ are bounded independently of $\mu, N$. Moreover, observe that \eqref{cond:mu:summary}, \eqref{cond:mu:N:summary} imply
    \begin{align}\label{est:muhnu}
        \frac{1}{\mu\hnu}\leq \frac{1}{18c_L^2U_1^2},\quad \frac{1}{\hnu N^2}\leq \frac{1}{\mu}.
    \end{align}
Upon invoking Lemma \ref{lem:w:L2} and making use of \eqref{est:w0}, we have
    \begin{align}\label{est:w}
        |w(\tau)|\leq \left(e^{-\mu\tau}\frac{\eta_0^2}{N^2}+2\frac{U_2^2}{\mu^2}|\De\nu|^2\right)^{1/2},\quad\sup_{t\in[0,\tau]}|w(t)|\leq \left(\frac{\eta_0^2}{N^2}+2\frac{U_2^2}{\mu^2}|\De\nu|^2\right)^{1/2}.
    \end{align}
and
    \begin{align}\label{def:Si}
        \Si=2\frac{\eta_0^2}{N^4}+4\frac{U_2^2\tau}{\mu}|\De\nu|^2.
    \end{align}
We are finally ready to prove Theorem \ref{thm:main}.

\begin{proof}[Proof of Theorem \ref{thm:main}]
From Lemma \ref{lem:model:error:basic}, we apply Lemma \ref{lem:R1}, Lemma \ref{lem:R2}, Lemma \ref{lem:R3}, Lemma \ref{lem:R4}, Lemma \ref{lem:R5}, Lemma \ref{lem:R6}, and obtain
    \begin{align}\label{est:main:a}
&(\De\nu)^2\sfW+\mu\sfJ\leq \frac{|w_0|^2}{2}+2^{1/2}|w(\tau)||\De\nu|\sfW^{1/2}\notag
\\
&\qquad+2^{5/2}\left[1+\left(\frac{|\De\nu|}{\hnu}+1\right)+c_A\frac{U_2^{1/2}U_0^{1/2}}{\hnu^{1/2}\mu^{1/2}}+\frac{c_L}{\hnu}\left(\sup_{t\in[0,\tau]}|w(t)|\right)\right]\Si^{1/2}|\De\nu|^2\sfW^{1/2}
    \end{align}
Let us assume that
    \begin{align}\label{def:tau}
            \tau\geq\tau_*:=\frac{2}{\mu}
    \end{align}
By \eqref{est:w0}, \eqref{est:w}, and Young's inequality, we have
    \begin{align}
        \frac{|w_0|^2}{2}&\leq \frac{\eta_0^2}{2N^2}\label{est:1}
        \\
        2^{1/2}|w(\tau)||\De\nu|\sfW^{1/2}&\leq 2^{1/2}\left(e^{-\mu\tau}\frac{\eta_0^2}{N^2}+2\frac{U_2^2}{\mu^2}|\De\nu|^2\right)^{1/2}|\De\nu|\sfW^{1/2}\notag
        \\
        &\leq 2^{1/2}\frac{\eta_0}{N e^{\mu\tau/2}}|\De\nu|\sfW^{1/2}+\frac{2U_2}{\mu}|\De\nu|^2\sfW^{1/2}\notag
        \\
        &\leq \frac{4\eta_0^2}{N^2 e^{\mu\tau}}+\frac{1}{8}(\De\nu)^2\sfW+\frac{2U_2}{\mu}|\De\nu|^2\sfW^{1/2}.\label{est:2}
    \end{align}
We observe that the terms within the brackets in \eqref{est:main:a} can be estimated with \eqref{est:al}, \eqref{est:muhnu}, and \eqref{est:w0} to obtain
    \begin{align}
        &2^{5/2}\left[1+\left(\frac{|\De\nu|}{\hnu}+1\right)+c_A\frac{U_2^{1/2}U_0^{1/2}}{\hnu^{1/2}\mu^{1/2}}+\frac{c_L}{\hnu}\left(\sup_{t\in[0,\tau]}|w(t)|\right)\right]\notag
        \\
        &\leq 8\left[2+\rho+\frac{c_A}{6c_L}\left(\frac{U_2U_0}{U_1^2}\right)^{1/2}+\frac{c_L}{\hnu}\left(\frac{\eta_0^2}{N^2}+\frac{U_2^2}{\mu^2}|\De\nu|^2\right)^{1/2}\right]\notag
        \\
        &\leq 8\left[2+\rho+\frac{c_A}{6c_L}\left(\frac{U_2U_0}{U_1^2}\right)^{1/2}+\frac{c_L\left(\eta_0+U_2\rho\right)}{\mu}\right]=\gam+\frac{\kap}{\mu},\notag
    \end{align}
where
    \begin{align}\label{def:gam:kap}
        \gam=8(2+\rho)+\frac{4c_A}{3c_L}\left(\frac{U_2U_0}{U_1^2}\right)^{1/2},\quad \kap=8c_L(\eta_0+U_2\rho)
    \end{align}
Therefore, by rearranging and applying Young's inequality, we obtain
    \begin{align}
    &2^{5/2}\left[1+\left(\frac{|\De\nu|}{\hnu}+1\right)+c_A\frac{U_2^{1/2}U_0^{1/2}}{\hnu^{1/2}\mu^{1/2}}+\frac{c_L}{\hnu}\left(\sup_{t\in[0,\tau]}|w(t)|\right)\right]\Si^{1/2}|\De\nu|\sfW^{1/2}\notag
    \\
    &\leq\left(\gam+\frac{\kap}{\mu}\right)\left(2\frac{\eta_0^2}{N^2}+4\frac{U_2^2\tau}{\mu}|\De\nu|^2\right)^{1/2}|\De\nu|\sfW^{1/2}\notag
    \\
    &\leq \frac{2\eta_0}{N}\left(\gam+\frac{\kap}{\mu}\right)|\De\nu|\sfW^{1/2}+2\left(\gam+\frac{\kap}{\mu}\right)\frac{U_2\tau^{1/2}}{\mu^{1/2}}|\De\nu|^2\sfW^{1/2}\notag
    \\
    &\leq  \frac{8\eta_0^2}{N^2}\left(\gam+\frac{\kap}{\mu}\right)^2+\frac{1}{8}(\De\nu)^2\sfW+2\left(\gam+\frac{\kap}{\mu}\right)\frac{U_2\tau^{1/2}}{\mu^{1/2}}|\De\nu|^2\sfW^{1/2}\label{est:3}
    \end{align}
Upon combining \eqref{est:1}, \eqref{est:2}, \eqref{est:3}, we arrive at
    \begin{align}\label{est:combine}
\frac{3}4(\De\nu)^2\sfW+\mu\sfJ&\leq \frac{\eta_0^2}{N^2}\left(\frac{1}{2}+ \frac{4}{e^{\mu\tau}}+8\left(\gam+\frac{\kap}{\mu}\right)^2\right)\notag
\\
&\quad+\left[ \frac{2U_2}{\mu\tau}+2\left(\gam+\frac{\kap}{\mu}\right)\frac{U_2}{\mu^{1/2}\tau^{1/2}}\right]
\tau(\De\nu)^2\sfW^{1/2}.
    \end{align}
Finally, we let
    \begin{align}\label{def:om}
        \om_*:=64\left[ \frac{1}{(\mu\tau)^{1/2}}+\left(\gam+\frac{\kap}{\mu}\right)\right]^2\frac{U_2^2}{\mu}
    \end{align}
Suppose that $\sfW\geq \om_*\tau$. It follows that
    \begin{align}\label{est:4}
       \left[ \frac{2U_2}{\mu\tau}+2\left(\gam+\frac{\kap}{\mu}\right)\frac{U_2}{\mu^{1/2}\tau^{1/2}}\right]
\tau\leq \frac{1}4\sfW^{1/2},
    \end{align}
We therefore conclude from \eqref{est:combine} and \eqref{est:4} that
    \begin{align}
    \frac{1}2(\De\nu)^2\sfW+\mu\sfJ\leq 8\left[\frac{1}{16}+\frac{e^{-\mu\tau}}{2}+\left(\gam+\frac{\kap}{\mu}\right)^2\right]\frac{\eta_0^2}{N^2}\notag,
    \end{align}
from which we deduce \eqref{est:limit} and \eqref{est:loss}.

On the other hand, if $\sfW<\om_*\tau$, then from \eqref{est:W:plus} in Lemma \ref{lem:W}, we see that
    \begin{align}\notag
        \frac{1}{\tau}\int_0^\tau|\hv|^2\,dt\leq \frac{2\sfW}{\mu\tau}<\frac{2}{\mu}\om_*,
    \end{align}
which establishes \eqref{est:insensitive}. This completes the proof.
\end{proof}

Several remarks are in order.

\begin{Rmk}\label{rmk:admissiable}
From the conditions imposed on $\mu, \unu, N$ in \eqref{cond:mu:summary} and \eqref{cond:mu:N:summary}, one can extract a necessary condition on the admissible set, $\sA$. Indeed, \eqref{cond:mu:summary} implies
        $\max\left\{\frac{18c_L^2\max\{U_1^2,V_1^2\}}{\mu},\frac{\mu}{\frac{1}4N^2}\right\}\leq\unu$.
Thus, the framework of Theorem \ref{thm:main} automatically implies a lower bound for $\sA$. The lower bound fits with one's expectations: as the reference viscosity, $\nu$, is smaller, which typically indicates increasingly complex fluid behavior, the admissible set would require either a larger nudging parameter or more observations for the minimization problem to have hope of approximating $\nu$.
\end{Rmk}

\begin{Rmk}\label{rmk:compatibility}
We observe from \eqref{est:W:cor} and \eqref{def:om} that if $\sfW\geq\om_*\tau$, then
    \begin{align}\label{est:compatibility}
        \frac{U_2^2}{\mu}\tau\sim\om_*\tau\leq\frac{2}{\mu}\int_0^\tau|Av|^2\,dt.
    \end{align}
This is required in order for the minimal excitation threshold to be non-empty. From the enstrophy balance of the assimilated variable, we have
    \begin{align}\notag
        &\frac{1}2\|v(\tau)\|^2+\hnu\int_0^\tau|Av|^2\,dt+\mu\int_0^\tau\|P_Nv\|^2\,dt\notag=\frac{1}2\|v(0)\|^2+\int_0^\tau(f+\mu P_Nu,Av)\,dt\notag
        \\
        &\geq\frac{1}2\|v(0)\|^2+\int_0^\tau\left[(f,Av)+\mu (A^{1/2}P_Nu,A^{1/2}P_Nv)\right]\,dt\notag \geq\frac{1}2\|v(0)\|^2\\
        &~~~~~~~~~~-\tau^{1/2}F_0\left(\int_0^\tau|Av|^2\,dt\right)^{1/2}-\left(\mu\int_0^\tau\|P_Nu\|^2\,dt\right)^{1/2}\left(\mu\int_0^\tau\|P_Nv\|^2\,dt\right)^{1/2}\notag
        \\
        &\geq \frac{1}2\|v(0)\|^2-\tau\frac{F_0^2}{2\hnu}-\frac{\hnu}2\int_0^\tau|Av|^2\,dt-\frac{\mu}2\int_0^\tau\|P_Nu\|^2\,dt-\frac{\mu}2\int_0^\tau\|P_Nv\|^2\,dt,\notag \text{ and}\\
        &\frac{1}2\|v(\tau)\|^2+\hnu\int_0^\tau|Av|^2\,dt+\mu\int_0^\tau\|P_Nv(t)\|^2\,dt\notag
        \\
        &\leq \frac{1}2\left[e^{-\mu\tau}\left(\|v(0)\|^2-4U_1^2\right)+4U_1^2\right]+(\hnu+\mu)\int_0^\tau|Av|^2\,dt.\notag
    \end{align}
Upon recalling $\hnu\in[\unu,(1+\rho)\unu]$ and \eqref{cond:mu:summary}, we have
    \begin{align}\notag
        3[(1+\rho)\unu+\mu]\int_0^\tau|Av|^2\,dt\geq(1-e^{-\mu\tau})\left(\|v(0)\|^2-4U_1^2\right)-\left(\tfrac{F_0^2}{\hnu}+\mu U_1^2\right)\tau.
    \end{align}
In particular, \eqref{est:compatibility} can be guaranteed to hold provided that the assimilated system is initialized with sufficiently large enstrophy, namely
    \begin{align}\label{est:lower:v0}
        \|v(0)\|^2\gtrsim \max\left\{U_1^2,[(1+\rho)\unu+\mu]U_2^2\tau,\left(\tfrac{F_0^2}{\hnu}+\mu U_1^2\right)\tau\right\},
    \end{align}
where the suppressed constant is sufficiently large, and $\mu\tau\gtrsim1$. Lastly, we recall \eqref{cond:mu:summary} requires that $\mu\unu\gtrsim V_1^2$. Therefore, in order to guarantee \eqref{est:compatibility} consistently with our detectability conditions, we must also additionally impose
	\begin{align}\notag
		\max\left\{U_1^2,[(1+\rho)\unu+\mu]U_2^2\tau,\left(\tfrac{F_0^2}{\unu}+\mu U_1^2\right)\tau\right\}\lesssim \mu\unu.
	\end{align}
Since we assume that $\mu\gtrsim U_1^2\unu^{-1}$, the above condition is a constraint on $\tau$:
	\begin{align}\notag
        \tau\lesssim \mu\unu\min\left\{
            \left([(1+\rho)\unu+\mu]U_2^2 \right)^{-1}, \left(\tfrac{F_0^2}{\unu}+\mu U_1^2 \right)^{-1}
            \right\}.
	\end{align}
Since $\mu\gtrsim U_1^2\unu^{-1}$ and the only existing constraint on $\tau$ preceding this discussion is $\mu\tau\gtrsim1$, we see that an appropriate choice of $v(0)$ is possible when $\tau$ satisfies
		$\frac{1}{\mu}\lesssim \tau\lesssim \mu\unu\min\left\{\frac{1}{[(1+\rho)\unu+\mu]U_2^2},\frac{1}{\frac{F_0^2}{\unu}+\mu U_1^2}\right\}$,
which is non-empty provided that $\mu\gtrsim U_2^2\unu^{-1}$. 

Notice that as $\mu\goesto\infty$, the right-hand side approaches $(U_2^2/\unu)^{-1}$, and thus, in the $\mu\goesto\infty$ limit, the constraint on $\tau$ remains consistent with \eqref{cond:mu:summary}:
		$\frac{1}{\mu}\lesssim \tau\lesssim \left(\frac{U_2^2}{\unu}\right)^{-1}$.
On the other hand, we see from \eqref{est:lower:v0} that
	\begin{align}\label{eq:lower:eta0}
		\eta_0^2=\|v(0)-u(0)\|^2\gtrsim \|v(0)\|^2-O(U_1^2)\gtrsim [(1+\rho)\unu+\mu]U_2^2\tau.
	\end{align}
In the most optimistic scenario, one would have $\eta_0^2\sim [(1+\rho)\unu+\mu]U_2^2\tau$. In the excited regime, $\om_*\sim U_2^2/\mu$, this implies that
		$(\De\nu)^2\leq O\left(\frac{\eta_0^2}{N^2\om_*\tau}\right)\sim O\left(\frac{[(1+\rho)\unu+\mu]\mu}{N^2}\right)$.
Hence, one would like to choose $\mu$ within the \textit{detectability regime} \eqref{cond:mu:summary} to be as small as possible, which in our framework forces $\mu\sim N_*\unu$. Under these assumptions, we have
	\begin{align}\notag
	|\De\nu|\leq
		O\left(\tfrac{N_*}{N}\unu\right),
	\end{align}
where $N_*\sim \max\{U_1^2,V_1^2\}\unu^{-2}$. In other words, with proper initialization and tuning of the assimilated system, Theorem \ref{thm:main} guarantees an error bounds that does not require  smallness of the initial state error. We conclude this remark by emphasizing that the above analysis does not guarantee that minimum excitation bound is met. Later, we show that it is possible for $\sfW$ to be bounded away from zero even though one is not in a regime of identifiability (see \cref{sect:Kolmogorov} and Remark \ref{rmk:obs:id}).
\end{Rmk}

\begin{Rmk}\label{rmk:stability}
In practice, solving the optimization problem \eqref{eq:minimize} results in an approximation, $\hnu_*$, of $\hnu\in\sC_\tau$. Assuming such an approximation is obtained within $\de$-accuracy of $\hnu$, we see that $
        |\hnu_*-\nu|\leq |\hnu-\hnu_*|+|\hnu-\nu|\leq O(\de)+|\De\nu|\notag$.
Additionally, $\sfW(\hnu;\tau)$ can be shown to be Lipschitz in $\hnu$ (see Lemma \ref{lem:W:Lipschitz}). Denote the Lipschitz constant by $C_{\sfW}$. Hence, in the excited regime, $\sfW(\hnu;\tau)\geq\om_*\tau$, we see that
    \begin{align}\notag
        \sfW(\hnu_*;\tau)=\sfW(\hnu_*;\tau)-\sfW(\hnu;\tau)+\sfW(\hnu:\tau)\geq -C_{\sfW}|\hnu_*-\hnu|+\om_*\tau\geq O(\de)+\om_*\tau.
    \end{align}
Thus, the excitation can be triggered for a sufficiently good approximation of $\hnu$. In such a situation, \eqref{est:limit} would hold and 
         $|\hnu_*-\nu|\leq O(\de)+O\left(\frac{\eta_0^2}{N^2\sfW(\hnu;\tau)}\right)$. \notag
In this sense, our results Theorem \ref{thm:main} remain stable with respect to numerical approximation.
\end{Rmk}

\begin{Rmk}\label{rmk:nudging}
In \eqref{cond:main} $\mu$, is assumed to be upper-bounded by the resolution, $N$, of the observations, which suggests that $\mu$ cannot be taken larger in our results without improving our observations. This constraint is natural and present in the work of \cite{AzouaniOlsonTiti2014}. Operationally, the effect of choosing $\mu$ too large was already known by Hoke and Anthes, who introduced the nudging filter for state estimation \cite{HokeAnthes1976}. Choosing $\mu$ large has the effect of relaxing the assimilated variable towards the observations at a time-scale faster than the one dictated by the reference dynamics, thereby suppressing useful dynamics from emerging in the assimilated system. In a recent work \cite{CarlsonFarhatMartinezVictor2024b}, the direct-replacement filter, $\tilde{v}$, in which the low-mode observations are inserted directly into the system, was identified as the infinite nudging limit of the nudging filter, i.e., $\lim_{\mu\goesto\infty}v=\tilde{v}$. If one replaces $v$ with $\tilde{v}$, then  $P_N\tilde{v}=P_Nu$, resulting in $\sfJ\equiv0$. Additionally, since $\om_*\sim U_2\mu^{-1}$, the excitation threshold also degenerates in the infinite-$\mu$ limit. Nevertheless, in the detectability regime \eqref{cond:main}, Theorem \ref{thm:main} is consistent with a trivial regime of identifiability: since $N\goesto\infty$ as $\mu\goesto\infty$,  \cite{CarlsonFarhatMartinezVictor2024b} implies $\lim_{\mu\goesto\infty}P_Nv=u$. See also \cite{CibikFangLaytonSiddiqua2025} where the issue of how to choose $\mu$ is studied.
\end{Rmk}

\section{On the Sharpness of the Excitation Threshold}\label{sect:sharp}

In this section we consider the sharpness of the excitation threshold asserted in Theorem \ref{thm:main}. Throughout, we assume  $f$ is \textit{time-independent}. 

\subsection{Small forcing regime}\label{sect:Grashof}

In this section, we assume that $F_0$ is sufficiently small.  In this regime, it is known that \eqref{eq:nse} possesses a unique globally attracting stationary state; a proof of this result can be found in \cite[Theorem 4.1]{DascaliucFoiasJolly2005}, but we provide relevant details here that will be useful for our particular application. Recall that the corresponding stationary equation of \eqref{eq:nse} is given by:
	\begin{align}\label{eq:nse:steady}
		\nu Au_*+B(u_*,u_*)=f.	
	\end{align}
Note that if $f\in H$, then every solution of \eqref{eq:nse:steady} automatically belongs to $D(A)$ due to the identity \eqref{eq:enstrophy:identity}, and by the Cauchy-Schwarz inequality, satisfies
	\begin{align}\label{est:nse:steady:sharp}
		\|u_*\|\leq\frac{F_*}{\nu},\quad |Au|\leq \frac{F_0}{\nu}
	\end{align}
We will extend these ideas to the context of assimilating system \eqref{eq:nse:nudge} and corresponding sensitivity equation \eqref{eq:sensitivity}. Ultimately, we demonstrate the sharpness of the excitation threshold identified in Theorem \ref{thm:main}. 

First, we state and supply a proof \cite[Theorem 4.1]{DascaliucFoiasJolly2005} tailored for our purposes.

\begin{Prop}\label{prop:Grashof}
Suppose that
    \begin{align}\label{cond:Grashof}
        F_0\leq \frac{\nu^2}{2c_L}.
    \end{align}
For any $u_*\in D(A)$ satisfying \eqref{eq:nse:steady}, one has
    \begin{align}\label{est:nse:steady}
        |u(t;u_0)-u_*|^2\leq e^{-\nu t}|u_0-u_*|^2,
    \end{align}
for all $u_0\in V$. In particular, there exists a unique $u_*\in D(A)$ satisfying \eqref{eq:nse:steady}.
\end{Prop}

\begin{proof}
First, we recall the classical fact that stationary states $u_*\in D(A)$ of \eqref{eq:nse} exist for any $f\in H$ (see \cite{ConstantinFoiasBook}). Let $u_*\in D(A)$ be any such stationary solution and let $y=u-u_*$. Then 
    \begin{align}\notag
        \bdy_ty+\nu Ay+B(y,y)+DB(u_*)y=0.
    \end{align}
Upon taking the $L^2$--inner product with $y$ and invoking \eqref{eq:energy:identity}, we obtain
    \begin{align}\notag
        \frac{1}2\frac{d}{dt}|y|^2+\nu\|y\|^2&=-b(y,u_*,y).
    \end{align}
By \eqref{est:trilinear:a}, Young's inequality,  \eqref{est:nse:steady:sharp} and \eqref{cond:Grashof}, it  then follows that
    \begin{align}
        \frac{1}2\frac{d}{dt}|y|^2+\nu\|y\|^2&\leq \frac{c_L^2}{\nu}\|u_*\|^2|y|^2+\frac{\nu}4\|y\|^2
        \leq \frac{c_L^2F_0^2}{\nu^3}|y|^2+\frac{\nu}4\|y\|^2\leq \frac{\nu}2\|y\|^2\notag.
    \end{align}
From which an application of Gr\"onwall's inequality yields \eqref{est:nse:steady}. Uniqueness then follows by an application of the triangle inequality.
\end{proof}     

Next, we consider the assimilating system \eqref{eq:nse:nudge}. The corresponding stationary equation to \eqref{eq:nse:nudge} with $u=u_*$, where $u_*$ satisfies \eqref{eq:nse:steady} is given by:
	\begin{align}\label{eq:nse:nudge:steady}
		\hnu Av_*+B(v_*,v_*)+\mu P_Nv_*=f+\mu P_Nu_*,	
	\end{align}
As with \eqref{eq:nse}, whenever $f\in H$ and satisfies the smallness condition \eqref{cond:Grashof}, one can show existence of solutions to \eqref{eq:nse:nudge:steady}. Under the assumption of small forcing  \eqref{cond:Grashof} and the detectability regime \eqref{cond:obs}, we determine the long-time behavior of the solution $v=v(\cdotp;v_0,f,P_Nu,\hnu)$ of \eqref{eq:nse:nudge}, that is, when the observations are given by the corresponding trajectory $\{P_Nu(t)\}_{t\geq0}$. Note that, similar to \eqref{est:nse:steady:sharp}, all solutions of \eqref{eq:nse:nudge:steady} satisfy
	\begin{align}\label{est:nse:nudge:steady:sharp}
		\hnu|Av_*|^2+\mu\|P_Nv_*\|^2&\leq \left[1+\frac{\mu\hnu}{\nu^2}\left(\frac{F_*}{F_0}\right)^2\right]\frac{F_0^2}{\hnu}.
	\end{align}
In fact, under a regime of detectability, $\mu\leq \hnu N^2$, we may further deduce
    \begin{align}\label{est:vstarH1}
            \|v_*\|^2\leq \left[1+\frac{\mu\hnu}{\nu^2}\left(\frac{F_*}{F_0}\right)^2\right]\frac{F_0^2}{\mu\hnu},
    \end{align}
and
    \begin{align}\label{est:Avstar}
        |Av_*|^2&\leq \frac{c_0^2c_L^2}{2\hnu}\|v_*\|^2|Av_*|^2\leq\frac{1}2\left\{c_0^2c_L^2\left[1+\frac{\mu\hnu}{\nu^2}\left(\frac{F_*}{F_0}\right)^2\right]^2\left(\frac{F_0}{F_1}\right)^2+1\right\}\frac{F_1^2}{\hnu\mu}
    \end{align}
whenever $f\in V$. We will prove these bounds in \cref{sect:a priori} (see Lemma \ref{lem:vstar:bounds}), but we present them to the reader to show that $v_*\in D(A)$ and, moreover, can be bounded in $D(A)$ independently of $\mu$ in the detectability regime.

\begin{Prop}\label{prop:Grashof:nudge}
Given $f\in H$ such that \eqref{cond:Grashof} holds, let $u_*$ be the unique stationary state of \eqref{eq:nse}. 
Suppose that
    \begin{align}\label{cond:obs}
        \mu\geq \frac{2c_LV_1^2}{\unu},\quad N\geq\frac{2c_LV_1}{\unu}.
    \end{align}
Then for any $v_*\in D(A)$ satisfying \eqref{eq:nse:nudge:steady}, we have
    \begin{align}\label{est:nse:nudge:steady}
        |v(t;v_0,f,P_Nu,\hnu)-v_*(P_Nu_*)|^2\leq O\left(\unu^2e^{-\unu t}\right),
    \end{align}
for all $v_0\in V$, where the suppressed constant depends on $|v_0-v_*|\unu^{-1}$, $|u_0-u_*|\unu^{-1}$, $\mu\unu^{-1}$. In particular, there is a unique solution $v_*\in D(A)$ of \eqref{eq:nse:nudge:steady} that is forward-attracting in $H$, globally in $V$. Moreover, it is also globally forward-attracting in $V$:
	\begin{align}\label{est:nse:nudge:steady:H1}
		\|v(t;v_0,f,P_Nu,\hnu)-v_*(P_Nu_*)\|^2\leq O\left(\unu^2e^{-\unu t}\right)
	\end{align}
where the suppressed non-dimensional constant depends on $|Av_*|\unu^{-1}$, $\mu\unu^{-1}$, $|v-v_0|\unu^{-1}$, $\|v_0-v_*\|\unu^{-1}$, $|u-u_*|\unu^{-1}$
\end{Prop}

\begin{proof}
Let $z=v(\cdot;P_Nu,\hnu)-v_*$. Then
    \begin{align}\notag
        \bdy_tz+\hnu Az+B(z,z)+DB(v_*)z&=-\mu P_Nz+\mu P_Ny.
    \end{align}
Upon taking the $L^2$--inner product with $w$ and applying \eqref{eq:energy:identity}, we obtain
    \begin{align}\notag
        \frac{1}2\frac{d}{dt}|z|^2+\hnu\|z\|^2+\mu P_Nw=-b(z,v_*,z)+\mu(P_N y, P_Nz).
    \end{align}
We then invoke \eqref{est:trilinear:a}, Young's inequality, orthogonality of $P_N, Q_N$, \eqref{est:Bernstein}, and \eqref{cond:obs} to estimate
    \begin{align}
        |b(z,v_*,z)|&\leq \frac{c_L^2}{\hnu}\|v_*\|^2|z|^2+\frac{\hnu}4\|z\|^2\notag
        \\
        &\leq \frac{c_L^2}{\hnu}\|v_*\|^2|P_Nz|^2+\frac{c_L^2}{\hnu N^2}\|v_*\|^2\|Q_Nz\|^2+\frac{\hnu}4\|z\|^2\notag
        \\
        &\leq \frac{\mu}2|P_Nz|^2+\frac{\hnu}2\|z\|^2.\notag
    \end{align}
By the Cauchy-Schwarz inequality and Young's inequality, we estimate
	\begin{align}
		\mu|(P_Ny,P_Nz)|&\leq \mu|P_Ny||P_Nz|\leq \frac{\mu}2|P_Ny|^2+\frac{\mu}2|P_Nz|^2\notag.
	\end{align} 
It follows that
    \begin{align}
        \frac{d}{dt}|z|^2+\hnu\|z\|^2\leq\mu|P_Ny|^2.\notag
    \end{align}
Since \eqref{cond:Grashof} holds, Proposition \ref{prop:Grashof} ensures that $|y(t)|^2\leq O(\unu^2e^{-\unu t})$. An application of \eqref{est:Poincare} and Gr\"onwall's inequality then yields \eqref{est:nse:nudge:steady}, for any $v_*$ satisfying \eqref{eq:nse:nudge:steady}. Uniqueness then follows as in Proposition \ref{prop:Grashof}.

On the other hand, upon taking the $L^2$--inner product with $Az$ and invoking \eqref{eq:enstrophy:identity}, we obtain
	\begin{align}\notag
		\frac{1}2\frac{d}{dt}\|z\|^2+\hnu|Az|^2+\mu\|P_Nz\|^2=-b(z,z,Av_*)+\mu(P_Ny,Az).
	\end{align}
Then \eqref{est:trilinear:b} and Young's inequality imply
	\begin{align}
		|b(z,z,Av_*)|&\leq c_A|Az|^{1/2}|z|^{1/2}\|z\||Av_*|\notag
		\\
		&\leq \frac{3c_A^2}{4\hnu^{1/3}}|z|^{2/3}|Av_*|^{4/3}\|z\|^{4/3}+\frac{\hnu}{4}|Az|^2\notag
		\\
		&\leq \frac{64c_A^6}{3\hnu^3}|Av_*|^4|z|^2+\frac{\hnu}8\|z\|^2+\frac{\hnu}4|Az|^2\notag,
	\end{align}
while the Cauchy-Schwarz and Young's inequality imply
	\begin{align}
		\mu|(P_Ny,Az)|&\leq \mu|P_Ny||Az|\leq\frac{2\mu^2}{\hnu}|P_Ny|^2+\frac{\hnu}8|Az|^2.\notag
	\end{align}
It follows that
	\begin{align}
		\frac{d}{dt}\|z\|^2+\hnu|Az|^2+\mu\|P_Nz\|^2\leq \frac{128c_A^6}{3\hnu^3}|Av_*|^4|z|^2+\frac{4\mu^2}{\hnu}|P_Ny|^2.\ \notag
	\end{align}
An application of Gr\"onwall's inequality, Proposition \ref{prop:Grashof}, and \eqref{est:nse:nudge:steady} yields \eqref{est:nse:nudge:steady:H1}.
\end{proof}

Next, we study the long-time behavior of solutions of \eqref{eq:sensitivity}. Indeed, given the unique globally attracting state $v_*$ of \eqref{eq:nse:nudge} derived from Proposition \ref{prop:Grashof:nudge}, observe that \eqref{eq:sensitivity} with $v=v_*$ is an inhomogeneous linear system:
	\begin{align}\label{eq:sensitivity:star}
		\bdy_t\hv+\hnu A\hv+\mu P_N\hv+DB(v_*)\hv=-Av_*,\quad \hv=\hv(\cdot;v_*).
	\end{align} 
The stationary sensitivity equation is then defined by
	\begin{align}\label{eq:sensitivity:steady}
		\hnu A\hv_*+\mu P_N\hv_*+DB(v_*)\hv_*=-Av_*,\quad \hv_*=\hv_*(v_*).
	\end{align}
Existence and uniqueness of solutions to \eqref{eq:sensitivity:steady} follow from a straightforward application of the Lax-Milgram theorem; we leave these details to the appendix (see Lemma \ref{lem:Grashof:LM}). In the following, we show that the global strong solution $\hv=\hv(\cdot;v)$ of \eqref{eq:sensitivity} corresponding to $v=v(\cdot;v_0;f,P_Nu,\hnu)$ is globally attracted to the unique solution $\hv_*$ of \eqref{eq:sensitivity:steady}. The proof proceeds in two steps: compare $\hv(\cdot;v)$ with $\hv(\cdot;v_*)$, then compare $\hv(\cdot;v_*)$ with $\hv_*(v_*)$.

\begin{Prop}\label{prop:Grashof:sensitivity}
Given $f\in H$ such that \eqref{cond:Grashof} holds, let $u_*$ denote the unique stationary state of \eqref{eq:nse}. Let $v=v(\cdot;v_0,f,P_Nu,\hnu)$ denote the unique strong solution of \eqref{eq:nse:nudge} corresponding to initial data $v_0$ and observations $P_Nu$. Suppose that 
	\begin{align}\label{cond:obs:sensitivity}
		  \mu\geq\frac{20c_L^2V_1^2}{\unu},\quad N\geq \frac{4\sqrt{2}c_LV_1}{\unu}.
	\end{align}
Then
	\begin{align}\label{est:sensitivity:steady}
		|\hv(t;v)-\hv_*(v_*)|^2\leq O(e^{-\unu t}).
	\end{align}
\end{Prop}

\begin{proof}
First, observe that \eqref{cond:obs:sensitivity} implies \eqref{cond:obs}, so that the results of Proposition \ref{prop:Grashof:nudge} hold. Let $\hz=\hv(\cdot;v)-\hv(\cdot;v_*)$ and $z=v(\cdot;P_Nu,\hnu)-v_*(P_Nu_*)$. We use the shorthand $v_*=v_*(P_Nu_*)$. Then
	\begin{align}\notag
			\bdy_t\hz+\hnu A\hz+DB(v)\hv(v)-DB(v_*)\hv(v_*)=-\mu P_N\hz-Az,
	\end{align}
and
	\begin{align}
		&DB(v)\hv(v)-DB(v_*)\hv(v_*)\notag
		\\
		&=\left(B(v,\hv(v))-B(v_*,\hv(v_*))\right)+\left(B(\hv(v),v)-B(\hv(v_*),v_*)\right)\notag
		\\
		&=B(z,\hv(v))+B(v_*,\hz)+B(\hz,v)+B(\hv(v_*),z)\notag
		\\
		&=B(z,\hz)+B(z,\hv(v_*))+B(v_*,\hz)+B(\hz,z)+B(\hz,v_*)+B(\hv(v_*),z).\notag
	\end{align}
Upon taking the $L^2$--inner product with $\hz$ and invoking \eqref{eq:energy:identity}, we obtain
	\begin{align}
		&\frac{1}2\frac{d}{dt}|\hz|^2+\hnu \|\hz\|^2+\mu|P_N\hz|^2\notag
		\\
		&=-b(z,\hv(v_*),\hz)-b(\hz,z,\hz)-b(\hz,v_*,\hz)-b(\hv(v_*),z,\hz)-(Az,\hz).\notag
	\end{align}
By repeated application of \eqref{est:trilinear:a}, \eqref{est:trilinear:b}, Young's inequality, and \eqref{est:Bernstein}, we see that
	\begin{align}
		|b(z,\hv(v_*),\hz)|&\leq c_A|z||A\hv(v_*)|^{1/2}|\hv(v_*)|^{1/2}\|\hz\|\notag
		\\
		&\leq \frac{4c_A^2}{\hnu}|A\hv(v_*)||\hv(v_*)||z|^2+\frac{\hnu}{16}\|\hz\|^2\notag
		\\
		|b(\hv(v_*),z,\hz)|&\leq c_A|A\hv(v_*)|^{1/2}|\hv(v_*)|^{1/2}\|\hz\||z|\notag
		\\
		&\leq \frac{4c_A^2}{\hnu}|A\hv(v_*)||\hv(v_*)|z|^2+\frac{\hv}{16}\|\hz\|^2\notag
		\\
		|b(\hz,z,\hz)|&\leq c_L\|\hz\||\hz|\|z\|\notag
		\\
		&\leq \frac{4c_L^2}{\hnu}\|z\|^2|\hz|^2+\frac{\hnu}{16}\|\hz\|^2\notag
		\\
		&\leq \frac{4c_L^2}{\hnu}\|z\|^2|P_N\hz|^2+\frac{4c_L^2}{\hnu N^2}\|z\|^2\|Q_N\hz\|^2+\frac{\hnu}{16}\|\hz\|^2\notag
		\\
		|b(\hz,v_*,\hz)|&\leq c_L\|\hz\||\hz|\|v_*\| \notag
		\\
		&\leq \frac{4c_L^2}{\hnu}\|v_*\|^2|\hz|^2+\frac{\hnu}{16}\|\hz\|^2\notag
		\\
		&\leq \frac{4c_L^2}{\hnu}\|v_*\|^2|P_N\hz|^2+\frac{4c_L^2}{\hnu N^2}\|v_*\|^2\|Q_N\hz\|^2+\frac{\hnu}{16}\|\hz\|^2\notag
	\end{align}
By the Cauchy-Schwarz inequality and Young's inequality, we also have
	\begin{align}\notag
		|(Az,\hz)|&\leq \|z\|\|\hz\|\leq\frac{4}{\hnu}\|z\|^2+\frac{\hnu}{16}\|\hz\|^2.
	\end{align}
Upon combining the above bounds, we see that
	\begin{align}
		&\frac{1}2\frac{d}{dt}|\hz|^2+\left(\frac{9\hnu}{16}-\frac{4c_L^2}{\hnu N^2}\|z\|^2\right)\|\hz\|^2+\left(\mu-\frac{4c_L^2}{\hnu}\|z\|^2-\frac{4c_L^2}{\hnu}\|v_*\|^2\right)|P_N\hz|^2\notag
		\\
		&\leq \frac{8c_A^2}{\hnu}|A\hv(v_*)||\hv(v_*)|z|^2+\frac{4}{\hnu}\|z\|^2.\notag
	\end{align}
Observe that $\|z\|^2\leq4V_1^2$. Then by \eqref{cond:obs:sensitivity}, it follows that
	\begin{align}
		\frac{d}{dt}|\hz|^2+\hnu\|\hz\|^2 \leq\frac{16c_A^2}{\hnu}|A\hv(v_*)||\hv(v_*)|z|^2+\frac{8}{\hnu}\|z\|^2.\notag
	\end{align}
Since $v_*\in D(A)$ and, from Proposition \ref{prop:Grashof:nudge}, $ \|z(t)\|^2\leq O\left(\unu^2e^{-\unu t}\right)$, we may deduce 
	\begin{align}\label{est:sensitivity:steady:a}
		|\hz(t)|^2\leq O(\unu t e^{-\unu t}),
	\end{align}
where the suppressed constant additionally depends on $\limsup_t\fint_0^t|A\hv(v_*)||\hv(v_*)|\,dt$. Note that this quantity is finite by Corollary \ref{cor:hv:H1}.

Lastly, let $\hz_*=\hv(\cdot;v_*)-\hv_*(v_*)$, where $\hv_*=\hv_*(v_*)$ satisfies \eqref{eq:sensitivity:steady}. Then
	\begin{align}\notag
		\bdy_t\hz_*+\hnu A\hz_*+\mu P_N\hz_*+DB(v_*)\hz_*=0.
	\end{align}
The corresponding energy balance is
	\begin{align}\notag
		\frac{1}2\frac{d}{dt}|\hz_*|^2+\hnu\|\hz_*\|^2+\mu|P_N\hz_*|^2=-b(\hz_*,v_*,\hz_*),
	\end{align}
where we have applied \eqref{eq:energy:identity}. By \eqref{est:trilinear:a}, Young's inequality, and \eqref{est:Bernstein}, we have
	\begin{align}
		|b(\hz_*,v_*,\hz_*)|&\leq c_L\|\hz_*\||\hz_*|\|v_*\|\leq \frac{c_L^2}{2\hnu}\|v_*\|^2|\hz_*|^2+\frac{\hnu}2\|\hz_*\|^2\notag
		\\
		&\leq \frac{c_L^2}{2\hnu}\|v_*\|^2|P_N\hz_*|^2+\frac{c_L^2}{2\hnu}\|v_*\|^2|Q_N\hz_*|^2+\frac{\hnu}2\|\hz_*\|^2\notag
		\\
		&\leq  \frac{c_L^2}{2\hnu}\|v_*\|^2|P_N\hz_*|^2+\frac{c_L^2}{2\hnu N^2}\|v_*\|^2\|\hz_*\|^2+\frac{\hnu}2\|\hz_*\|^2\notag.
	\end{align}
Upon combining the above with \eqref{cond:obs:sensitivity}, we arrive at
	\begin{align}\notag
		\frac{d}{dt}|\hz_*|^2+\hnu\|\hz_*\|^2\leq0,
	\end{align}
from which we deduce $|\hz_*(t)|^2\leq O(e^{-\unu t})$ with an application of \eqref{est:Poincare} and Gr\"onwall's inequality.

Finally, by the triangle inequality, \eqref{est:sensitivity:steady:a}, and the fact that $|\hz_*(t)|^2\leq O(\unu te^{-\unu t})$, we conclude
	\begin{align}\notag
		|\hv(t;v)-\hv_*(v_*)|\leq| \hv(t;v)-\hv(t;v_*)|+|\hv(t;v_*)-\hv_*(v_*)|\leq O((\unu t+1)e^{-\unu t}),
	\end{align}
as desired.
\end{proof}

Next, we show that the unique states $v_*, \hv_*$ identified in Proposition \ref{prop:Grashof:nudge}, Proposition \ref{prop:Grashof:sensitivity} realize the minimal threshold asserted in Theorem \ref{thm:main} up to a constant. To do so, we will make use of three two lemmas, the first of which will allow one to compare stationary states of \eqref{eq:nse:nudge} with those of \eqref{eq:nse} within an detectability regime, and the second of which will establish a lower bound on the power functional corresponding to steady states. In particular, the following assertion holds independently of the smallness condition \eqref{cond:Grashof}.

\begin{Lem}\label{lem:w:star}
Let $u_*$ satisfy \eqref{eq:nse:steady} and $v_*$ satisfy \eqref{eq:nse:nudge:steady}. Suppose that  $\nu,\hnu\in\sA$ and 
	\begin{align}\label{cond:w:star}
		\mu\geq \frac{2c_A^2|Au_*|^2}{\unu},\quad N^2\geq \frac{2\sqrt{2}c_A|Au_*|}{\unu}.
	\end{align}
Then for all $\rho>0$
	\begin{align}\label{est:w:star}
		|Av_*-Au_*|\leq \sqrt{2}\rho|Au_*|.
	\end{align}
\end{Lem}

\begin{proof}
Let $w_*=v_*-u_*$. Then
	\begin{align}\notag
		\hnu Aw_*+\mu P_Nw_*+B(w_*,w_*)+DB(u_*)w_*=-(\De\nu)Au_*.
	\end{align}
Upon taking the $L^2$--inner product with $Aw_*$ and invoking \eqref{eq:enstrophy:identity}, we obtain
	\begin{align}\notag
		\hnu|Aw_*|+\mu\|P_Nw_*\|^2=-b(w_*,w_*,Au_*)-(\De\nu)(Au_*,Aw_*).
	\end{align}
Observe that \eqref{est:trilinear:b}, \eqref{est:Poincare}, \eqref{est:Bernstein}, and Young's inequality imply
	\begin{align}
		|b(w_*,w_*,Au_*)|&\leq c_A|Aw_*|^{1/2}|w_*|^{1/2}\|w_*\||Au_*|\leq c_A|Aw_*||w_*||Au_*|\notag
		\\
		&\leq \frac{2c_A^2|Au_*|^2}{\hnu}|w_*|^2+\frac{\hnu}8|Aw_*|^2\notag
		\\
		&\leq \frac{2c_A^2|Au_*|^2}{\hnu}|P_Nw_*|^2+\frac{2c_A^2|Au_*|^2}{\hnu}|Q_Nw_*|^2+\frac{\hnu}8|Aw_*|^2\notag
		\\
		&\leq \frac{2c_A^2|Au_*|^2}{\hnu\mu}\mu\|P_Nw_*\|^2+\frac{2c_A^2|Au_*|^2}{\hnu N^4}|Aw_*|^2+\frac{\hnu}8|Aw_*|^2\notag
		\\
		&\leq \mu\|P_Nw_*\|^2+\frac{\hnu}4|Aw_*|^2,\notag
	\end{align}
where we applied \eqref{cond:w:star} in obtaining the final inequality. By Cauchy-Schwarz and Young's inequality, we also have	
	\begin{align}\notag
		|\De\nu||(Au_*,Aw_*)|&\leq |\De\nu||Au_*||Aw_*|\leq \frac{|\De\nu|^2}{\hnu}|Au_*|^2+\frac{\hnu}4|Aw_*|^2.
	\end{align}
Upon combining these estimates and using the facts that $\hnu\geq\unu$ and $|\De\nu|\leq \rho\unu$, it follows that
	\begin{align}
		\frac{\hnu}2|Aw_*|^2\leq \frac{|\De\nu|^2}{\hnu}|Au_*|^2\leq \rho^2|Au_*|^2,\notag
	\end{align}
as desired.
\end{proof}

Next, we establish a crucial lemma through which we show how one can obtain a useful lower bound estimate on the power function $\sfP_*(\hnu):=-(Av_*,\hv_*(v_*))$ that depends explicitly on the steady state $v_*$ alone. It is based on the observation that the linear equation satisfied by $\hv_*$ can be decomposed as follows:
	\begin{align}\label{eq:vstar:rewrite}
		(L_\mu+DB(v_*))\hv_*=-Av_*,\quad L_\mu:=\hnu A+\mu P_N,
	\end{align}
which allows us to view $\hv_*$ as a function of $\hv_*$ (see Lemma \ref{lem:Grashof:LM} for details) . We then obtain the following estimate.

\begin{Lem}\label{lem:power:lower}
Given any $u_*$ satisfying \eqref{eq:nse:steady}, let $v_*(u_*)$ satisfy \eqref{eq:nse:nudge:steady} and $\hv_*=\hv_*(v_*)$ satisfy \eqref{eq:sensitivity:steady}. Suppose
    \begin{align}\label{cond:sharp}
        \begin{split}
        \mu&\geq \max\left\{\frac{32c_L^2\|v_*\|^2}{\hnu},\frac{16c_{BG}^2\|v_*\|^2}{\hnu}\left[1+\log\left(\frac{|Av_*|}{|v_*|}\right)\right]\right\},
        \\
        N&\geq \max\left\{\frac{8c_L\|v_*\|}{\hnu},\frac{4c_{BG}\|v_*\|}{\hnu}\left[1+\log\left(\frac{|Av_*|}{|v_*|}\right)\right]^{1/2}\right\}.
        \end{split}
    \end{align}
Then
    \begin{align}\label{est:power:lower:main}
		\sfP_*(\hnu)\geq \frac{1}3(Av_*,L_{\mu}^{-1}Av_*).
	\end{align}
\end{Lem}

\begin{proof}
Let $\hv_*=\hv_*(v_*)$. We may obtain an elementary lower bound on $\sfP_*$ upon taking the $L^2$--inner product of \eqref{eq:sensitivity:steady} with $\hv_*$, then applying \eqref{est:trilinear:a}, Young's inequality, and \eqref{est:Bernstein}, we obtain 
	\begin{align}
		\sfP_*(\hnu)&:=-(Av_*,\hv_*(v_*))\notag
		\\
		&= \hnu \|\hv_*\|^2+\mu|P_N\hv_*|^2+b(\hv_*,v_*,\hv_*)\notag
		\\
		&\geq \hnu \|\hv_*\|^2+\mu|P_N\hv_*|^2-c_L\|\hv_*\|^2|\hv_*|\|v_*\|\notag
		\\
		&\geq \frac{3\hnu}4 \|\hv_*\|^2+\mu|P_N\hv_*|^2-\frac{c_L^2}{\hnu}\|v_*\|^2|\hv_*|^2\notag
		\\
		&\geq\frac{3\hnu}4 \|\hv_*\|^2+\mu|P_N\hv_*|^2-\frac{c_L^2}{\hnu}\|v_*\|^2|P_N\hv_*|^2-\frac{c_L^2}{2\hnu N^2}\|v_*\|^2\|\hv_*\|^2\notag
		\\
		&\geq \frac{\hnu}2 \|\hv_*\|^2+\frac{\mu}2|P_N\hv_*|^2,\label{est:lower}
	\end{align}
where we have applied \eqref{cond:sharp} in obtaining the final inequality. To prove \eqref{est:power:lower:main}, it now suffices to obtain an upper bound on $(Av_*,L_{\mu}^{-1}Av_*)$ of the same form.

To this end, observe that $L_\mu$ is self-adjoint and invertible. Then by \eqref{eq:vstar:rewrite}, we see that
	\begin{align}
		&(Av_*,L_{\mu}^{-1}Av_*)\notag
		\\
		&=((L_\mu+DB(v_*))\hv_*,L_{\mu}^{-1}(L_\mu+DB(v_*))\hv_*)\notag
		\\
		&=((L_\mu+DB(v_*))\hv_*,(I+L_{\mu}^{-1}DB(v_*))\hv_*)\notag
		\\
		&=(L_\mu\hv_*,\hv_*)+(L_\mu\hv_*,L_{\mu}^{-1}DB(v_*)\hv_*)+(DB(v_*)\hv_*,\hv_*)+(DB(v_*)\hv_*,L_{\mu}^{-1}DB(v_*)\hv_*)\notag
		\\
		&=\hnu\|\hv_*\|^2+\mu|P_N\hv_*|^2+2b(\hv_*,v_*,\hv_*)+(DB(v_*)\hv_*,L_{\mu}^{-1}DB(v_*)\hv_*).\notag
	\end{align}
Applying \eqref{est:trilinear:a}, Young's inequality, orthogonality of $P_N, Q_N$, \eqref{est:Bernstein}, and \eqref{cond:sharp} yields
	\begin{align}
		2|b(\hv_*,v_*,\hv_*)|&\leq 2c_L\|\hv_*\||\hv_*|\|v_*\|\leq\frac{8c_L^2\|v_*\|^2}{\hnu}|\hv_*|^2+\frac{\hnu}{8}\|\hv_*\|^2\notag
		\\
		&\leq \frac{8c_L^2\|v_*\|^2}{\hnu}|P_N\hv_*|^2+\frac{8c_L^2\|v_*\|^2}{\hnu}|Q_N\hv_*|^2+\frac{\hnu}{8}\|\hv_*\|^2\notag
		\\
		&\leq \frac{8c_L^2\|v_*\|^2}{\hnu}|P_N\hv_*|^2+\frac{8c_L^2\|v_*\|^2}{\hnu N^2}\|\hv_*\|^2+\frac{\hnu}{8}\|\hv_*\|^2\notag
		\\
		&\leq \frac{\mu}4|P_N\hv_*|^2+\frac{\hnu}4\|\hv_*\|^2.\notag
	\end{align}
Also
	\begin{align}
		(DB(v_*)\hv_*,L_{\mu}^{-1}DB(v_*)\hv_*)&=(B(v_*,\hv_*)+B(\hv_*,v_*),L_{\mu}^{-1}(B(v_*,\hv_*)+B(\hv_*,v_*)))\notag
		\\
		&=(\bdy_j(v_*^j\hv_*^\ell+\hv_*^jv_*^\ell),L_{\mu}^{-1}(\bdy_i(v^i\hv_*^\ell+\hv_*^iv_*^\ell))\notag
		\\
		&=-(v_*^j\hv_*^\ell+\hv_*^jv_*^\ell,\bdy_i\bdy_jL_{\mu}^{-1}(v^i\hv_*^\ell+\hv_*^iv_*^\ell))\notag.
	\end{align}
Observe that
	\begin{align}
		|\bdy_i\bdy_jL_{\mu}^{-1}\phi|^2&=\sum_{|k|\leq N}\frac{|k_i|^2|k_j|^2}{(\hnu|k|^2+\mu)^2}|\phi_k|^2+\sum_{|k|> N}\frac{|k_i||k_j|}{\hnu|k|^2}|\phi_k|^2\notag
		\\
		&\leq\sum_{|k|\leq N}\frac{|k|^4}{(\hnu|k|^2+\mu)^2}|\phi_k|^2+\hnu^{-2}\sum_{|k|> N}|\phi_k|^2\leq \hnu^{-2}|\phi|^2\notag.
	\end{align}
Upon applying the Cauchy-Schwarz inequality, \eqref{est:BrezisGallouet}, Young's inequality, and \eqref{cond:sharp}, we deduce
	\begin{align}
		&|(DB(v_*)\hv_*,L_{\mu}^{-1}DB(v_*)\hv_*)|\notag
		\\
		&\leq \left(\sum_{j,\ell}|v_*^j\hv_*^\ell+\hv_*^jv_*^\ell|^2\right)^{1/2}\sum_i\left(\sum_{j,\ell}|\bdy_i\bdy_j L_{\mu}^{-1}(v_*^i\hv_*^\ell+\hv_*^iv_*^\ell)|^2\right)^{1/2}\notag
		\\
		&\leq\hnu^{-1} \left(\sum_{j,\ell}(|v_*^j|_\infty|\hv_*^\ell|+|\hv_*^j||v_*^\ell|_\infty)^2\right)^{1/2}\sum_i\left(\sum_{j,\ell}(|v_*^j|_\infty|\hv_*^\ell|+|\hv_*^j||v_*^\ell|_\infty)^2\right)^{1/2}\notag
		\\
		&\leq 4\hnu^{-1}\sum_{j,\ell}(|v_*^j|_\infty^2|\hv_*^\ell|^2+|\hv_*^j|^2|v_*^\ell|_\infty^2)\notag
		\\
		&\leq 8\frac{|v_*|_\infty^2}{\hnu}|\hv_*|^2\leq 8\frac{|v_*|_\infty^2|}{\hnu}|P_N\hv_*|^2+8\frac{|v_*|_\infty^2|}{\hnu}|Q_N\hv_*|^2\notag
		\\
		&\leq 8\frac{|v_*|_\infty^2|}{\hnu}|P_N\hv_*|^2+\frac{8|v_*|_\infty^2}{\hnu N^2}\|\hv_*\|^2\notag
		\\
		&\leq \frac{4c_{BG}^2\|v_*\|^2}{\hnu}\left[1+\log\left(\frac{|Av_*|}{|v_*|}\right)\right]|P_N\hv_*|^2+\frac{4c_{BG}^2\|v_*\|^2}{\hnu N^2}\left[1+\log\left(\frac{|Av_*|}{|v_*|}\right)\right]\|\hv_*\|^2\notag
		\\
		&\leq \frac{\mu}4|P_N\hv_*|^2+\frac{\hnu}4\|\hv_*\|^2\notag.
	\end{align}
Therefore,
	\begin{align}\label{est:upper}
		(Av_*,L_{\mu}^{-1}Av_*)\leq \frac{3\hnu}2\|\hv_*\|^2+\frac{3\mu}2|P_N\hv_*|^2.
	\end{align}
We deduce \eqref{est:power:lower:main} upon combining \eqref{est:lower} and \eqref{est:upper}.
\end{proof}

We then easily obtain the desired lower bound on $\sfP_*$ as a corollary.

\begin{Cor}\label{cor:power:lower}
Under the assumption of Lemma \ref{lem:power:lower}, we obtain
    \begin{align}\label{est:power:lower:final}
		\sfP_*(\hnu)\geq\frac{1}3\sup_{m\leq N}\frac{|AP_mv_*|^2}{\hnu m^2+\mu},
	\end{align}
\end{Cor}

\begin{proof}
By Parseval's theorem, we see that
\begin{align}
		(Av_*,L_{\mu}^{-1}Av_*)&=\sum_{|k|\leq N}\frac{1}{\hnu|k|^2+\mu}|(Av_*)_k|^2+\sum_{|k|>N}\frac{1}{\hnu|k|^2}|(Av_*)_k|^2\notag
		\\
		&\geq \sum_{|k|\leq N}\frac{1}{\hnu|k|^2+\mu}|(Av_*)_k|^2\geq \sup_{m\leq N}\frac{|AP_mv_*|^2}{\hnu m^2+\mu},\notag
	\end{align}
as desired.
\end{proof}

Finally, we are ready to show that the lower bound in Theorem \ref{thm:main} can be realized.

\begin{Prop}\label{prop:Grashof:sharp}
Suppose \eqref{cond:Grashof} holds. Let $u_*$ be the unique steady state of \eqref{eq:nse} guaranteed by Proposition \ref{prop:Grashof}. Suppose \eqref{cond:obs}, \eqref{cond:obs:sensitivity}, \eqref{cond:w:star}, and \eqref{cond:sharp} also hold. Then there exists $\rho_*>0$ and $N_f>0$ such that for all $\rho\leq \rho_*$ and $N\geq N_f$, one has
    \begin{align}\label{est:power:lower:sharp}
        \sfP_*(\hnu)\gtrsim\frac{|Au_*|^2}{\mu},
    \end{align}
whenever $\mu\sim \hnu N^2$.
\end{Prop}

\begin{proof}
Define $N_f$ such that
	\begin{align}\label{def:active}
		N_f:=\min\{m\geq0:|P_mu_*|\geq\frac{1}2|Au_*|\}.
	\end{align}
By Lemma \ref{lem:w:star}, for $\rho_*=\sqrt{2}/4$, we have
	\begin{align}
		|P_mAv_*|\geq	|P_mAu_*|-|P_mAw_*|\geq \frac{1}2|Au_*|-|Aw_*|\geq \frac{1}4|Au_*|.\notag
	\end{align}
Thus, for any $N\geq N_f$, we may apply Corollary \ref{cor:power:lower} to obtain
	\begin{align}
			\sfP_*(\hnu)\geq \frac{1}3\sup_{m\leq N}\frac{|P_mAv_*|^2}{\hnu m^2+\mu} \geq \frac{1}{12}\frac{|Au_*|^2}{\hnu N^2+\mu}.\notag
	\end{align}
Thus, for $\mu\sim \hnu N^2$, we have
    \begin{align}\notag
            \sfP_*(\hnu)\gtrsim \frac{|Au_*|^2}{\mu}.
    \end{align}
as claimed.
\end{proof}

We also have the following quantitative error estimate between $\sfP(\hnu;\tau)$ and $\sfP_*(\hnu)$.

\begin{Lem}\label{lem:Grashof:error}
Under the hypotheses of Proposition \ref{prop:Grashof:nudge} and Proposition \ref{prop:Grashof:sensitivity} alone, we have
    \begin{align}\label{est:power:converge}
        \sup_{\hnu\in\sA}\left|\sfP(\hnu;\tau)-\sfP_*(\hnu)\right|\leq O\left(\frac{1}{\tau}\right).
    \end{align}
\end{Lem}

\begin{proof}
Let $z=v-v_*$ and $\hz=\hv-\hv_*$. Then   
    \begin{align}\notag 
    \sfP(\hnu;\tau)-\sfP_*(\hnu)&=-\frac{1}{\tau}\int_0^\tau(Av,\hv)\,dt+(Av_*,\hv_*)\notag
    \\
    &=\frac{1}{\tau}\int_0^\tau (A^{1/2}z,A^{1/2}\hz)+(z,A\hv_*)+(Av_*,\hz)\,dt\notag.
    \end{align}
Observe that by Lemma \ref{lem:W} and Corollary \ref{cor:W}, we have
    \begin{align}\notag
        \int_0^\tau\|\hv(t)\|^2\,dt\leq \frac{2}{\hnu}\sfW\leq \frac{4}{\hnu\mu}\int_0^\tau|Av|^2\,dt.
    \end{align}
Then, by the Cauchy-Schwarz inequality Proposition \ref{prop:Grashof:nudge}, and Proposition \ref{prop:Grashof:sensitivity} we see that
    \begin{align}
        &\left|\sfP(\hnu;\tau)-\sfP_*(\hnu)\right|\notag
        \\
        &\leq \left(\fint_0^\tau\|z\|^2\,dt\right)^{1/2}\left(\fint_0^\tau\|\hz\|^2\,dt\right)^{1/2}+\left(\int_0^\tau|z|\,dt\right)\frac{|A\hv_*|}{\tau}+\frac{|Av_*|}{\tau}\left(\int_0^\tau|\hz|\,dt\right)\notag
        \\
        &\leq O\left(\frac{1}{\tau}\right)+O\left(\frac{|A\hv_*|}{\tau}\right)+O\left(\frac{|Av_*|}{\tau\unu}\right)\notag
    \end{align}
as desired.
\end{proof}

We recall from \eqref{def:om} that $\om_*$ scales like $\om_*\sim U_2^2/\mu$. If one then denotes the suppressed constant $\om_*\sim U_2^2/\mu$ by $c_*$ and in \eqref{est:power:lower:sharp} by $\sfc_*$, then we require $\sfc_*\geq c_*$ in order to be in the excited regime of Theorem \ref{thm:main}, i.e., $\frac{|Au_*|^2}{\mu}$. In other words, this isolates the verification of the minimal excitation threshold to the comparison of absolute constants. At this point, we emphasize that none of the constants appearing in the proof of Theorem \ref{thm:main} have been optimized and that refinement of the threshold is possible with a more careful analysis. Either way, a comprehensive numerical study is naturally called for to develop a better sense for the excitation regime. The analysis above serves to indicate that the small forcing regime is explicitly consistent with the excitation threshold in Theorem \ref{thm:main} from a dimensional analysis perspective and to provide a path to how the excitation threshold can be checked \textit{a priori} in practice.

We nevertheless remark that should $\sfc_*\geq c_*$ hold, then one may deduce that the minimal threshold is met through $v_*$. On the other hand, Lemma \ref{lem:Grashof:error} implies the long-time stability of the power functional across all $\hnu\in\sA$. One may then deduce the following from Corollary \ref{cor:main}:

\begin{Cor}\label{cor:Grashof:sharp:identify}
Assume that the hypotheses of Proposition \ref{prop:Grashof:sharp} and Lemma \ref{lem:Grashof:error} hold. If $\sfc_*\geq c_*$, then 
    \begin{align}\notag
        \limsup_{\tau\goesto\infty}\left|\sfP(\hnu_\tau;\tau)-\sfP_*(\hnu_\tau)\right|=0
    \end{align}
for all $\hnu_\tau\in\sA$, and
    \begin{align}\notag
        \lim_{\tau\goesto\infty}|\hnu_\tau-\nu|=0.
    \end{align}
\end{Cor}

In the following section, we demonstrate the sharpness of the lower bounds obtained in Corollary \ref{cor:power:lower} and Proposition \ref{prop:Grashof:sharp} in the setting of Kolmogorov flows. In doing so, we once again demonstrate the consistency of the excitation threshold of Theorem \ref{thm:main}, as well as its conclusions, but in a setting where all quantities are exactly computable, and therefore sharp.

\subsection{Kolmogorov flows}\label{sect:Kolmogorov}

The excitation threshold can also be met when the smallness condition \eqref{cond:Grashof} does not hold. We do so by considering Kolmogorov flows, which are special solutions of \eqref{eq:nse} when the external force is prescribed by an eigenfunction of the Stokes operator, $A$. Consequently, we will show that the excitation threshold characterized by $\om_*=c_*U_2^2\mu$ cannot in general be improved beyond optimizing for the constant $c_*$.

Define
    \begin{align}\label{def:kol:force}
        f:=F\sin(k_fx_2)e_1,
    \end{align}
where $F>0$ and $k_f\in\NN$. Then 
    \begin{align}\label{eq:kol:flow}
        u_f:=\frac{F}{\nu k_f^2}\sin(k_fx_2)e_1,
    \end{align}
is an exact solution of \eqref{eq:nse}. Indeed, by direct verification $B(u_f,u_f)=0$, so that \eqref{eq:nse} reduces simply to to
    \begin{align}\label{eq:kol:flow:rep}
        u_f=\nu^{-1}A^{-1}f.    
    \end{align}
We refer the reader to, for instance, \cite{Marchioro1986}, where a detailed study of these solutions was carried out. We will only require elementary properties of these solutions, all of which are readily calculated. 

Now let $\sO=\{P_Nu_f\}$ and consider the corresponding stationary equation for the assimilated variable $v$:
    \begin{align}\label{eq:nse:nudge:steady:kol}
        \hnu Av+B(v,v)=f-\mu P_Nv+\mu P_Nu_f.
    \end{align}
We may then consider two cases: $k_f\leq N$, i.e., the forcing is observed, and $k_f>N$, the forcing is not observed.

\subsubsection*{Case: $k_f\leq N$.} This case is a regime where the state is observed and viscosity can be recovered. Then
    \begin{align}\label{def:v:obs}
        v_{f,\mu}=(\hnu A+\mu P_N)^{-1}(\nu A+\mu P_N)u_f=\frac{\nu k_f^2+\mu}{\hnu k_f^2+\mu}u_f,
    \end{align}
satisfies \eqref{eq:nse:nudge:steady}. Turning to the stationary equation for the sensitivity variable $\hv$:
    \begin{align}\label{eq:sensitivity:steady:kol}
        \hnu A\hv_{f,\mu} +DB(v_{f,\mu})\hv_{f,\mu}=-\mu P_N\hv_{f,\mu}-Av_{f,\mu},
    \end{align}
we see that 
    \begin{align}\label{def:hv:obs}
        \hv_{f,\mu}:=-(\hnu A+\mu P_N)^{-1}Av_{f,\mu}=-k_f^2\frac{\nu k_f^2+\mu}{(\hnu k_f^2+\mu)^2}u_f
    \end{align}
satisfies \eqref{eq:sensitivity:steady:kol}. It follows that
    \begin{align}\notag
        \sfW(\hnu)&=-\tau(Av_{f,\mu},\hv_{f,\mu})=\tau\frac{(\nu k_f^2+\mu)^2}{(\hnu k_f^2+\mu)^3}|Au_f|^2,\notag
        \\
        \quad \sfG(\hnu)&=\tau|P_N\hv_{f,\mu}|^2=\tau\frac{(\nu k_f^2+\mu)^2}{(\hnu k_f^2+\mu)^4}|Au_f|^2.\notag
    \end{align}
Thus if $\hnu=\nu$, we have
    \begin{align}\label{eq:WG:obs}
        \sfW(\nu)=\frac{\tau}{\nu k_f^2+\mu}|Au_f|^2=\frac{U_2^2}{\nu k_f^2+\mu}\tau,\quad \sfG(\nu)=\tau\frac{k_f^2}{(\nu k_f^2+\mu)^2}|Au_f|^2.
    \end{align}
We recall $\om_*$ from Theorem \ref{thm:main}, which, up to a constant, satisfies $\om_*\sim\frac{U_2^2}{\mu}$. If $\mu\sim \nu k_f^2$, then 
    \begin{align}\notag
        \sfW(\nu)\sim \frac{U_2^2}{\mu}\tau\sim\om_*\tau,\quad \sfG(\nu)\sim\frac{U_2^2}{\nu\mu}\tau\sim\frac{1}{\mu}\om_*\tau.
    \end{align}
We further observe that
    \begin{align}
        \sfP(\nu)=|Av_{f,\mu}|^2/(\hnu k_f^2+\mu)\geq\frac{1}3|Av_{f,\mu}|^2/(\hnu k_f^2+\mu)\notag.
    \end{align}
Thus \eqref{est:power:lower:final} from Corollary \ref{cor:power:lower} is sharp up to a factor of $1/3$.

\subsubsection*{Case: $N<k_f$.} This case represents the regime where the state is not observed at any finite time. Then $P_Nu_f=0$ and 
    \begin{align}
        v_{f,\mu}=\hnu^{-1}A^{-1}f=\frac{\nu}{\hnu}u_f,\quad \hv_{f,\mu}=- \frac{1}{\hnu^{2}}A^{-1}f=-\frac{\nu}{\hnu^2}u_f.\notag
    \end{align}
It follows that $P_Nv=P_N\hv=0$. Thus   
    \begin{align}
        \sfW(\hnu)=\frac{1}{k_f^2\hnu^3}F_0^2\tau
        \sim\frac{\nu^2}{\hnu^3k_f^2}U_2^2\tau,\quad \sfG(\hnu)=\sfJ(\hnu)=0.\notag
    \end{align}  
In particular, for $\mu\sim N^2\hnu$, then
    \begin{align}
        \sfW(\hnu)\sim \left(\frac{\nu}{\hnu}\right)^2\left(\frac{N}{k_f}\right)^2\frac{U_2^2}{\mu}\tau.\notag
    \end{align}
Since 
    \[
    \left(\frac{\nu}{\hnu}\right)^2\left(\frac{N}{k_f}\right)^2\leq (1+\rho)^2\left(\frac{N}{k_f}\right)^2\ll1,
    \]
whenever $N\ll k_f$ and $\rho\ll1$, we see that it is not possible to lower $\om_* =c_0 U_2^2/\mu$ with $c_0\ll1$. This implies that any excitation condition must be formulated with $c_0\sim O(1)$, which is precisely what our analysis obtains. Indeed, when $\hnu=\nu$, we see that 
    \begin{align}\notag
    \frac{1}{\tau}\int_0^\tau|\hv_{f,\mu}|^2\,dt=\frac{\nu^2}{\hnu^4k_f^4}|Au_f|^2=\frac{U_2^2}{\nu^2k_f^4}.
    \end{align}
By \eqref{cond:main}, we therefore have
    \begin{align}\notag
        \frac{\frac{1}{\tau}\int_0^\tau|\hv_{f,\mu}|^2\,dt}{\om_*/\mu}\sim\left(\frac{\mu}{\nu k_f^2}\right)^2\leq\left(\frac{N^2}{4k_f^2}\right)^2<1,\qquad N<k_f,
    \end{align}
which is consistent with the assertion \eqref{est:insensitive} of Theorem \ref{thm:main} indicating the regime of insensitivity.

\begin{Rmk}\label{rmk:smallness}
We emphasize that no smallness assumption on $F$ is required here. The smallness condition \eqref{cond:Grashof} only enters when demonstrating the sharpness of our excitation condition at finite $\tau$.
\end{Rmk}

\begin{Rmk}\label{rmk:obs:id}
We note that the ``non-detectability regime" $N<k_f$ indicates an important distinction. When $\hnu=\nu$, it represents a regime where 1) true state is not observed at any finite time, 2) the assimilated dynamics asymptotically converge to the true references state, i.e., the detectable regime, but 3) the parameter is not identifiable. Thus, filter stability and state detectability can hold, but not identifiability. In our framework, the distinction between these concepts can be detected in the gap between positivity of $\sfW$ and the degeneracy of $\sfG$. This gap features in Lemma \ref{lem:W} through the additional residual term appearing in the right hand side of \eqref{est:W:G}.
\end{Rmk}

\begin{Rmk}\label{rmk:BH}
The example developed in \cref{sect:Kolmogorov} gives clarity to a proposed obstruction to identifiability discovered in \cite{BiswasHudson2023}. In \cite{BiswasHudson2023}, identifiability of $\nu$ was tied to the zero state being excluded from the the global attractor $\cA$. However, in the regime, $k_f>N$,  we see $\cA=\{u^*\}$ where $u^*\neq0$. Thus, the hypothesis, $0\notin\cA$, of \cite[Theorem 5.3]{BiswasHudson2023} holds, but the hypothesis of detectability, $N\geq k_f$, does not. In this case, the viscosity is not identifiable since the nudging filter only recovers $u_f$ up to the multiplicative error $\hnu/\nu$. Nevertheless, this is consistent with \cite[Theorem 5.2]{BiswasHudson2023}, which imposes $N\geq k_f$ (see also \cite[Example 5.4]{BiswasHudson2023}).

In this regime, it is therefore not vanishing on the attractor, but rather vanishing of the \emph{observed palinstrophy}, $|P_NAu^*|$, that leads to the breakdown of identifiability. In other words, the obstruction to identifiability is whether or not the global attractor is \emph{observably nonzero}. One can see from Corollary \ref{cor:power:lower} that $\sfW$ serves as the practical vehicle for deciding this since it is computable. Furthermore, as discussed in Remark \ref{rmk:obs:id}, our excitation condition is capable of deciding between identifiability and non-identifiability in the sharpest possible way: above the threshold, identifiability holds, while below the threshold, one has insensitivity with respect to changes in the viscosity, even if the state is recovered asymptotically.
\end{Rmk}

{\flushleft\textbf{Acknowledgments.}} V.R.M. was in part supported by the National Science Foundation through DMS 2213363, DMS 2206491, the Simons Foundation through MP-TSM 00014320, and the Dolciani Halloran Foundation. X.W. was in part supported by the National Science Foundation through DMS 2418701. S.S was in part supported by the National Science Foundation through DMS-2037851.

{\flushleft\textbf{AI Disclosure.}} The authors acknowledge the use of Anthropic's Claude, Fable 5.1, to assist in finding references and probing the relation of our results to existing literature. All references found in this way were checked by the authors. The manuscript, in its entirety, is written by the authors who take full responsibility for its content.

\appendix

\section{a priori estimates}\label{sect:a priori}

\subsection{a priori estimates for nudged Navier-Stokes solutions}

\begin{Lem}\label{lem:v:L2}
Suppose that $\mu, N, \hnu$ satisfy
    \begin{align}\label{cond:mu:N:hnu:L2}
        \mu\leq \frac{1}2\hnu N^2.
    \end{align}
Then
     \begin{align}\label{est:v:L2a}
        |v(t)|^2+\hnu\int_{0}^te^{-\mu(t-s)}\|v(s)\|^2\leq e^{-\mu t}|v(0)|^2+2\left(\frac{F_0^2}{\mu^2}+U_0^2\right)(1-e^{-\mu t})
    \end{align}
and
     \begin{align}\label{est:v:L2b}
        |v(t)|^2&+\hnu\int_{0}^{t}\|v(s)\|^2ds+\mu\int_{0}^{t}|v|^2ds\leq |v(0)|^2+2\mu\left(\frac{F_0^2}{\mu^2}+U_0^2\right)t,
    \end{align}
for all $t\geq0$. 
\end{Lem}

\begin{proof}
Upon taking the inner product of \eqref{eq:nse:nudge} with $v$ and invoking the identity \eqref{eq:energy:identity}, we obtain
    \begin{align}\notag
        \frac{1}2\frac{d}{dt}|v|^2+\hnu\|v\|^2+\mu|v|^2=\mu|Q_Nv|^2+(f,v)+\mu(P_Nu,v).
    \end{align}
Observe that the Cauchy-Schwarz inequality and Bernstein's inequality imply
    \begin{align}
        |(f,v)|&\leq \frac{F_0^2}{\mu}+\frac{\mu}4\|v\|^2\notag
        \\
        \mu|Q_Nv|^2&\leq\frac{\mu}{N^2}\|v\|^2\notag\\
        \mu|(P_Nu,v)|&\leq\mu|u|^2+\frac{\mu}4|v|^2.\notag
    \end{align}
Upon invoking \eqref{cond:mu:N:hnu:L2}, it follows that
    \begin{align}\notag
        \frac{d}{dt}|v|^2+\hnu\|v\|^2+\mu|v|^2\leq 2\frac{F_0^2}{\mu}+2\mu U_0^2.
    \end{align}
The inequalities \eqref{est:v:L2a}, \eqref{est:v:L2b} then follow.
\end{proof}

\begin{Lem}\label{lem:v:H1}
Suppose that $\mu, N, \hnu$ satisfy
    \begin{align}\label{cond:mu:N:hnu:H1}
        \mu\leq \frac{1}2\hnu N^2.
    \end{align}
Then
     \begin{align}\label{est:v:H1a}
        \|v(t)\|^2+\hnu\int_{0}^te^{-\mu(t-s)}|Av(s)|^2\leq e^{-\mu t}\|v(0)\|^2+2\left(\frac{F_1^2}{\mu^2}+U_1^2\right)(1-e^{-\mu t})
    \end{align}
and
     \begin{align}\label{est:v:H1b}
        \|v(t)\|^2&+\hnu\int_{0}^{t}|Av(s)|^2ds+\mu\int_{0}^{t}\|v(s)\|^2ds\leq \|v(0)\|^2+2\mu\left(\frac{F_1^2}{\mu^2}+U_1^2\right)t,
    \end{align}
for all $t\geq0$. 
\end{Lem}

\begin{proof}
Upon taking the inner product of \eqref{eq:nse:nudge} with $Av$ and invoking the identity \eqref{eq:enstrophy:identity}, we obtain
    \begin{align}\notag
        \frac{1}2\frac{d}{dt}\|v\|^2+\hnu|Av|^2+\mu\|v\|^2=\mu\|Q_Nv\|^2+(f,Av)+\mu(P_Nu,Av)
    \end{align}
Then the Cauchy-Schwarz inequality, Bernstein's inequality, and Young's inequality imply
    \begin{align}
        |(f,Av)|&\leq \frac{\|f\|^2}{\mu}+\frac{\mu}4\|v\|^2\notag
        \\
        \mu\|Q_Nv\|^2&\leq\frac{\mu}{N^2}|Av|^2\notag\\
        \mu|(P_Nu,Av)|&\leq \mu\|u\|^2+\frac{\mu}4\|v\|^2.\notag
    \end{align}
Upon invoking \eqref{cond:mu:N:hnu:H1}, it follows that
    \begin{align}\notag
        \frac{d}{dt}\|v\|^2+\hnu|Av|^2+\mu\|v\|^2\leq 2\frac{F_1^2}{\mu}+\mu U_1^2.
    \end{align}
The inequalities \eqref{est:v:H1a}, \eqref{est:v:H1b} then follow.
\end{proof}

\begin{Lem}\label{lem:v:H2}
Suppose that $\mu, N, \hnu$ satisfy
    \begin{align}\label{cond:mu:hnu:H2}
         \frac{9c_L^2}{\hnu}\left(\frac{F_1^2}{\mu^2}+ U_1^2\right)\leq \mu\leq \frac{1}4\hnu N^2.
    \end{align}
Then
    \begin{align}\label{est:v:H2}
        |Av(t)|^2&\leq \exp\left(\frac{9c_L^2\|v(0)\|^2}{\mu\hnu}\right)\left[e^{-\mu t}|Av(0)|^2+2\left(\frac{F_2^2}{\mu^2}+U_2^2\right)(1-e^{-\mu t})\right],
    \end{align}
for all $t\geq0$.
\end{Lem}

\begin{proof}
Upon taking the inner product of \eqref{eq:nse:nudge} with $A^2v$, we have
    \begin{align}\notag
        \frac{1}2\frac{d}{dt}|Av|^2+\hnu|A^{3/2}v|^2+\mu|Av|^2+b(v,v,A^2v)=\mu|Q_NAv|^2+(f,A^2v)+\mu(P_Nu,A^2v).
    \end{align}
We observe that
    \begin{align}\notag
        b(v,v,A^2v)&=b(\nabla v, v, \nabla Av)+b(v,\nabla v, A\nabla v)\notag
        \\
        &=b(Av,v,Av)-2b(\nabla v,\nabla v, Av).\notag
    \end{align}
Then by interpolation and Young's inequality, we obtain
    \begin{align}\notag
        |b(v,v,A^2v)|&\leq c_L\|Av\||Av|\|v\|+2c_L\|v\|\|\nabla\nabla v\|^{1/2}\|\nabla v\|^{1/2}\|Av\|^{1/2}|Av|^{1/2}\notag
        \\
        &\leq 3c_L|A^{3/2}v||Av|\|v\|\notag
        \\
        &\leq \hnu|A^{3/2}v|^2+\frac{9}4c_L^2\frac{\|v\|^2}{\hnu}|Av|^2.\notag
    \end{align}
On the other hand, the Cauchy-Schwarz inequality, Bernstein's inequality, and Young's inequality imply
    \begin{align}\notag
        \mu|Q_NAv|^2&\leq \frac{\mu}{N^2}|A^{3/2}v|^2\notag
        \\
        |(f,A^2v)|&\leq |Af||Av|\leq \frac{2}{\mu}|Af|^2+\frac{\mu}8|Av|^2\notag
        \\
        \mu|(P_Nu,A^2v)|&\leq \mu|P_NAu||Av|\leq 2\mu|Au|^2+\frac{\mu}8|Av|^2.\notag
    \end{align}
Therefore
    \begin{align}\notag
        \frac{d}{dt}|Av|^2+\mu\left(\frac{3}2-\frac{9}2c_L^2\frac{\|v\|^2}{\mu\hnu}\right)|Av|^2\leq 4\mu\left(\frac{F_2^2}{\mu^2}+U_2^2\right).
    \end{align}
By Gr\"onwall's inequality, it follows that
    \begin{align}\notag
        |Av(t)|^2&\leq\exp\left(-\frac{3}2\mu t+\frac{9c_L^2}{2\hnu}\int_0^t\|v(s)\|^2ds\right)|Av(0)|^2\notag
        \\
        &\quad+2\mu\left(\frac{F_2^2}{\mu^2}+U_2^2\right)\int_0^t\exp\left(-\frac{3}2\mu(t-s)+\frac{9c_L^2}{2\hnu}\int_s^t\|v(r)\|^2dr\right)ds.\notag
    \end{align}
By Lemma \ref{lem:v:H1}, it follows that
    \begin{align}\label{est:intermediate:H2}
        \frac{9c_L^2}{2\hnu}\int_{s}^{t}\|v(r)\|^2dr&\leq \frac{9c_L^2\|v(0)\|^2}{2\mu\hnu}+\frac{9c_L^2}{\hnu}\left(\frac{F_1^2}{\mu^2}+U_1^2\right)(t-s),
    \end{align}
for all $0\leq s\leq t$. Therefore, \eqref{est:v:H2} follows upon applying  \eqref{cond:mu:hnu:H2} and \eqref{est:intermediate:H2}.
\end{proof}

\begin{Cor}\label{cor:v:simplify}
Under the hypotheses of Lemma \ref{lem:v:L2}, if additionally
    \begin{align}\label{cond:mu:hnu:simplify:L2}
        \mu\geq \frac{F_0}{U_0},
    \end{align}
then    
    \begin{align}
        |v(t)|^2&\leq e^{-\mu t}|v(0)|^2+4U_0^2(1-e^{-\mu t})\label{est:v:L2a:simplify}
        \\
        \fint_0^{\tau}|v(t)|^2\,dt&\leq \frac{|v(0)|^2}{\mu\tau}+4U_0^2.\label{est:v:L2b:simplify}
    \end{align}
Under the hypotheses Lemma \ref{lem:v:H1}, if additionally
    \begin{align}\label{cond:mu:hnu:simplify:v:H1}
        \mu\geq\frac{F_1}{U_1},\quad 
    \end{align}
holds, then
    \begin{align}
         \|v(t)\|^2&\leq e^{-\mu t}\|v(0)\|^2+4U_1^2(1-e^{-\mu t})\label{est:v:H1a:simplify}
         \\
         \fint_{0}^{\tau}\|v(t)\|^2\, dt&\leq \frac{\|v(0)\|^2}{\mu \tau}+4 U_1^2.\label{est:v:H1b:simplify}
    \end{align}
Under the hypotheses of Lemma \ref{lem:v:H2}, if additionally \eqref{cond:mu:hnu:simplify:v:H1} and
    \begin{align}\label{cond:mu:hnu:simplify:v:H2}
        \mu\geq \max\left\{9c_L^2\frac{\|v(0)\|^2}{\hnu},\frac{F_2}{U_2}\right\},
    \end{align}
hold, then
    \begin{align}
        |Av(t)|^2&\leq e^{-\mu t+1}|Av(0)|^2+4eU_2^2(1-e^{-\mu t}),\label{est:v:H2a:simplify}
        \\
        \fint_0^\tau|Av(t)|^2\,dt&\leq e\frac{|Av(0)|^2}{\mu\tau}+4eU_2^2.\label{est:v:H2b:simplify}
    \end{align}
\end{Cor}

\subsection{a priori estimates for assimilated steady states}

We will now prove \eqref{est:vstarH1} and \eqref{est:Avstar}.

\begin{Lem}\label{lem:vstar:bounds}
Suppose that $u_*$ satisfies \eqref{eq:nse:steady} and that $v_*$ satisfies \eqref{eq:nse:nudge:steady}. Suppose that 
    \begin{align}\label{cond:obs:vstar}
        \mu\leq \hnu N^2.
    \end{align}
Then \eqref{est:vstarH1} holds. If $f\in V$, then \eqref{est:Avstar} also holds.
\end{Lem}

\begin{proof}
To prove \eqref{est:vstarH1}, we simply add $\mu\|Q_Nv_*\|^2$ to both sides, then apply \eqref{est:Bernstein} and \eqref{cond:obs:vstar}.

Let us then assume that $f\in V$. Upon taking the $L^2$--inner product of \eqref{eq:nse:nudge:steady} with $A^2v_*$ and applying \eqref{eq:energy:identity}, we obtain
  \begin{align}
        \nu|A^{3/2}v_*|^2+\mu|A\hv_*|^2&=b(v_*,v_*,\bdy_j^2\bdy_k^2v_*)+(A^{1/2}f,A^{3/2}v)\notag
        \\
        &=-b(\bdy_jv_*,v_*,\bdy_j^2\bdy_k^2v_*)-b(v_*,\bdy_jv_*,\bdy_j\bdy_k^2v_*)+(A^{1/2}f,A^{3/2}v)\notag
        \\
        &=b(Av_*,v_*,Av_*)-2b(\nabla v_*,\nabla v_*,Av_*)+(A^{1/2}f,A^{3/2}v).\notag
    \end{align}
It follows that
    \begin{align}
        |b(Av_*,v_*,Av_*)|+2|b(\nabla v_*,\nabla v_*,Av_*)|&\leq c_0c_L\|Av_*\||Av_*|\|v_*\|\leq c_0c_L|A^{3/2}v_*||Av_*|\|v_*\|\notag
        \\
        &\leq  \frac{c_0^2c_L^2}{2\hnu}\|v_*\|^2|Av_*|^2+\frac{\hnu}2|A^{3/2}v_*|^2\notag
        \\
        |(A^{1/2}f,A^{3/2}v_*)|&\leq \|f\|A^{3/2}v_*|\leq \frac{\|f\|^2}{2\hnu}+\frac{\hnu}2|A^{3/2}v_*|^2\notag.
    \end{align}
Thus, upon applying \eqref{est:vstarH1}, we obtain
    \begin{align}
        |Av_*|^2&\leq \frac{c_0^2c_L^2}{2\hnu}\|v_*\|^2|Av_*|^2\leq\frac{1}2\left\{c_0^2c_L^2\left[1+\frac{\mu\hnu}{\nu^2}\left(\frac{F_*}{F_0}\right)^2\right]^2\left(\frac{F_0}{F_1}\right)^2+1\right\}\frac{F_1^2}{\hnu\mu}.\notag
    \end{align}
as desired.
\end{proof}

\subsection{a priori estimates for state error}

\begin{Lem}\label{lem:w:L2}
Suppose that $\mu, N$ satisfy
    \begin{align}\label{cond:mu:N:L2}
        4c_L^2\frac{U_1^2}{\min\{\nu,\hnu\}}\leq \mu\leq\frac{1}4\min\{\nu,\hnu\} N^2.
    \end{align}
Then 
    \begin{align}\label{est:w:L2a:general}
        \begin{split}
         |w(t)|^2&\leq e^{-\mu t}|w(0)|^2+2\frac{|\De\nu|^2}{\mu}\int_{0}^te^{-\mu(t-s)}\min\left\{|Av(s)|^2,|Au(s)|^2\right\}ds
         \end{split}
    \end{align}
and
    \begin{align}\label{est:w:L2b:general}
        \begin{split}
        |w(t)|^2+\nu\int_{0}^t\|w(s)\|^2ds+\mu\int_{0}^t|w(s)|^2ds&\leq |w(0)|^2+2\frac{|\De\nu|^2}{\mu}\int_{0}^t|Av(s)|^2ds,
        \\
        |w(t)|^2+\hnu\int_{0}^t\|w(s)\|^2ds+\mu\int_{0}^t|w(s)|^2ds&\leq |w(0)|^2+2\frac{|\De\nu|^2}{\mu}\int_{0}^t|Au(s)|^2ds,
        \end{split}
    \end{align}
for all $t\geq0$.
\end{Lem}

\begin{proof}
Upon taking the inner product of \eqref{eq:state:error} with $w$, we obtain
    \begin{align}\notag
        \frac{1}2\frac{d}{dt}|w|^2+\nu\|w\|^2+\mu|w|^2=-b(w,u,w)+\mu|Q_Nw|^2-(\De\nu)\lb Av,w\rb.
    \end{align}
Observe that
    \begin{align}\notag
        |b(w,u,w)|&\leq c_L\|w\||w|\|u\|\leq c_L^2\frac{\|u\|^2}{\nu}|w|^2+\frac{\nu}4\|w\|^2
    \end{align}
Also
    \begin{align}\notag
        |\De\nu||\lb Av,w\rb|&\leq |\De\nu||Av||w|\leq \frac{|\De\nu|^2}{\mu}|Av|^2+\frac{\mu}4|w|^2.
    \end{align}
Upon invoking the second inequality in \eqref{cond:mu:N:L2}, we have
    \begin{align}\notag
        \mu|Q_Nw|^2&\leq\frac{\mu}{N^2}\|w\|^2\leq\frac{\nu}4\|w\|^2.
    \end{align}
Upon combining the above inequalities, we obtain
    \begin{align}\notag
        \frac{d}{dt}|w|^2+\nu\|w\|^2+\left(\frac{3}2\mu-2c_L^2\frac{\|u\|^2}{\nu}\right)|w|^2&\leq 2\frac{|\De\nu|^2}{\mu}|Av|^2.
    \end{align}
Lastly, upon invoking the first inequality of \eqref{cond:mu:N:L2}, we arrive at
    \begin{align}\notag
        \frac{d}{dt}|w|^2+\nu\|w\|^2+\mu|w|^2\leq 2\frac{|\De\nu|^2}{\mu}|Av|^2.
    \end{align}
We now repeat the same argument, but making use of \eqref{eq:state:error:equiv} instead. Then
    \begin{align}\notag
        \frac{d}{dt}|w|^2+\hnu\|w\|^2+\mu|w|^2\leq 2\frac{|\De\nu|^2}{\mu}|Au|^2,
    \end{align}
also holds. We combine both to deduce \eqref{est:w:L2a:general} and \eqref{est:w:L2b:general}.
\end{proof}

\subsection{a priori estimates for sensitivity variable}

\begin{Lem}\label{lem:hv:L2}
Under the hypotheses of Lemma \ref{lem:v:H2}, if additionally \eqref{cond:mu:hnu:simplify:v:H1}, \eqref{cond:mu:hnu:simplify:v:H2}, and
    \begin{align}\label{cond:mu:hnu:N:hv}
            \frac{8c_L^2}{\hnu}V_1^2\leq \mu\leq \frac{1}4N^2\hnu
    \end{align}
hold, then
     \begin{align}\label{est:hv:L2}
        |\hv(t)|^2+\hnu\int_0^t\|\hv(s)\|^2ds+\mu\int_0^t|\hv(s)|^2ds\leq \frac{2e}{\mu}\left(\frac{|Av(0)|^2}{\mu}+4U_2^2t\right),
    \end{align}
In particular
    \begin{align}\label{est:hv:L2:avg}
        \fint_0^\tau|\hv(t)|^2\,dt\leq \frac{2e}{\mu^2}\left(\frac{|Av(0)|^2}{\mu\tau}+4U_2^2\right).
    \end{align}
\end{Lem}

\begin{proof}
The energy balance of \eqref{eq:sensitivity} is given by
\begin{align}
        \frac{1}2\frac{d}{dt}|\hv|^2+\hnu\|\hv\|^2+\mu|P_N\hv|^2=-b(\hv,v,\hv)-(Av,\hv)\notag.
    \end{align}
Observe that by interpolation and the Cauchy-Schwarz inequality, we have
    \begin{align}
        |b(\hv,v,\hv)|&\leq c_L\|\hv\||\hv|\|v\|\leq \frac{\hnu}4\|\hv\|^2+\frac{2c_L^2}{\hnu}\|v\|^2|\hv|^2.\notag
    \end{align}
Also, by the Cauchy-Schwarz inequality, Bernstein inequality, and the upper bound in \eqref{cond:mu:hnu:N:hv} we have
    \begin{align}
        \mu|P_N\hv|^2=\mu|\hv|^2-\mu|Q_N\hv|^2\geq \mu|\hv|^2-\frac{\mu}{N^2}\|\hv\|^2\geq \mu|\hv|^2-\frac{\hnu}4\|\hv\|^2.\notag
    \end{align}
Lastly, by the Cauchy-Schwarz inequality and Young's inequality, we have
    \begin{align}
        |(Av,\hv)|&\leq |Av||\hv|\leq \frac{1}{\mu}|Av|^2+\frac{\mu}4|\hv|^2\notag.
    \end{align}
Combining the above estimates and invoking the lower bound in \eqref{cond:mu:hnu:N:hv}, it follows that
    \begin{align}
        \frac{d}{dt}|\hv|^2+{\hnu}\|\hv\|^2+{\mu}|\hv|^2\leq \frac{2}{\mu}|Av|^2.\notag
    \end{align}
Thus, upon integrating over $[0,\tau]$ and using the fact that $\hv(0)=0$, we arrive at
    \begin{align}
        |\hv(t)|^2+\hnu\int_0^t\|\hv(s)\|^2ds+\mu\int_0^t|\hv(s)|^2ds\leq \frac{2}{\mu}\int_0^t|Av(s)|^2ds.\notag
    \end{align}
An application of Corollary \ref{cor:v:simplify} then yields \eqref{est:hv:L2}, as desired.
\end{proof}

\begin{proof}
The energy balance of \eqref{eq:sensitivity} is given by
\begin{align}
        \frac{1}2\frac{d}{dt}|\hv|^2+\hnu\|\hv\|^2+b(\hv,v,\hv)+\mu|P_N\hv|^2=-(Av,\hv)\notag.
    \end{align}
Observe that
    \begin{align}
        b(\hv,v,\hv)=b(P_N\hv,v,P_N\hv)+b(Q_N\hv,v,P_N\hv)+b(Q_N\hv,v,P_N\hv)+b(Q_N\hv,v,Q_N\hv).\notag    
    \end{align}
Then by interpolation and the Cauchy-Schwarz inequality, we have
    \begin{align}
        |b(P_N\hv,v,P_N\hv)|&\leq c_L\|P_N\hv\||P_N\hv|\|v\|\leq \frac{\hnu}4\|\hv\|^2+\frac{c_L^2}{\hnu}\|v\|^2|P_N\hv|^2\notag
        \\
        |b(Q_N\hv,v,P_N\hv)|&\leq c_{BG}(1+\log N)^{1/2}|Q_N\hv|\|\hv\|\|P_N\hv\|\leq c_{BG}\frac{(1+\log N)^{1/2}}{N}\|v\|\|\hv\|^2\notag
        \\
        |b(P_N\hv,v,Q_N\hv)|&\leq c_{BG}(1+\log N)^{1/2}|Q_N\hv|\|\hv\|\|P_Nv\|\leq c_{BG}\frac{(1+\log N)^{1/2}}{N}\|v\|\|\hv\|^2\notag
        \\
        |b(Q_N\hv,v,Q_N\hv)|&\leq c_L\|Q_N\hv\||Q_N\hv|\|v\|\leq \frac{c_L}{N}\|v\|\|\hv\|^2\notag
    \end{align}
Upon combining the above estimates, we obtain
    \begin{align}
         \frac{1}2\frac{d}{dt}|\hv|^2+\frac{3\hnu}2\|\hv\|^2+\frac{3\mu}4|\hv|^2\leq -(Av,\hv)\notag.
    \end{align}
On the other hand, observe that
    \begin{align}
        (Av,\hv)=(Av,P_N\hv)+(Av,Q_N\hv).\notag    
    \end{align}
Thus, by the Cauchy-Schwarz inequality and Young's inequality, we have
    \begin{align}
        |(Av,P_N\hv)|&\leq |Av||P_N\hv|\leq \frac{1}{\mu}|Av|^2+\frac{\mu}4|P_N\hv|^2\notag\\
        |(Av,Q_N\hv)|&\leq \frac{1}N|Av|\|Q_N\hv\|\leq \frac{1}{\hnu N^2}|Av|^2+\frac{\hnu}4\|\hv\|^2\notag.
    \end{align}
It then follows that
    \begin{align}
        \frac{d}{dt}|\hv|^2+\hnu\|\hv\|^2+\mu|\hv|^2\leq 2\left(\frac{1}{\mu}+\frac{1}{\hnu N^2}\right)|Av|^2.\notag
    \end{align}
Integrating over $[0,t]$ then yields
    \begin{align}
            |\hv(t)|^2+\hnu\int_0^t\|\hv(s)\|^2ds+\mu\int_0^t|\hv(s)|^2ds\leq2\left(\frac{1}{\mu}+\frac{1}{\hnu N^2}\right)\int_0^t|Av(s)|^2ds,\notag
    \end{align}
as desired.
\end{proof}

\begin{Lem}\label{lem:hv:H1}
Suppose that each of the assumptions in Corollary \ref{cor:v:simplify} and Lemma \ref{lem:hv:L2} hold. Then
        \begin{align}\label{est:H1:hv}
 \|\hv(t)\|^2+\hnu\int_0^t|A\hv(s)|^2\, ds+&\mu\int_0^t\|P_N\hv(s)\|^2\, ds\leq\frac{2}{\hnu}\left(\frac{c_A^2V_2^2}{\mu^2}+1\right)V_2^2t
    \end{align}
\end{Lem}

\begin{proof}
Upon taking the $L^2$--inner product of \eqref{eq:sensitivity} with $A\hv$ and applying \eqref{eq:enstrophy:identity}, we obtain
    \begin{align}\notag
        \frac{1}2\frac{d}{dt}\|\hv\|^2+\hnu|A\hv|^2+\mu\|P_N\hv\|^2=b(\hv,\hv,Av)-(Av,A\hv).
    \end{align}
By \eqref{est:BrezisGallouet}, \eqref{est:interpolation:elementary}, and Young's inequality we have
    \begin{align}
       |b(\hv,\hv,Av)|&\leq c_A|A\hv|^{1/2}|\hv|^{1/2}\|\hv\||Av|\leq  c_A|A\hv||\hv||Av|\notag
       \\
       &\leq \frac{c_A^2}{\hnu}|Av|^2|\hv|^2+\frac{\hnu}4|\hv|^2.\notag
    \end{align}
By the Cauchy-Schwarz inequality and Young's inequality
    \begin{align}\notag
        |(Av,A\hv)|\leq|Av||A\hv|\leq \frac{1}{\hnu}|Av|^2+\frac{\hnu}4|A\hv|^2
    \end{align}
It follows that
    \begin{align}\notag
        \frac{d}{dt}\|\hv\|^2+\hnu|A\hv|^2+2\mu\|P_N\hv\|^2\leq \frac{2c_A^2}{\hnu}|Av|^2|\hv|^2+\frac{2}{\hnu}|Av|^2.
    \end{align}
Integrating over $[0,t]$ yields
    \begin{align}\notag
        \|\hv(t)\|^2+\hnu\int_0^t|A\hv(s)|^2\, ds+&\mu\int_0^t\|P_N\hv(s)\|^2\, ds\notag
        \\
        &\leq\frac{2c_A^2}{\hnu}\int_0^t|Av(s)|^2|\hv(s)|^2\,ds+\frac{2}{\hnu}\int_0^t|Av(s)|^2\,ds.\notag
    \end{align}
From Corollary \ref{cor:v:simplify}, Lemma \ref{lem:hv:L2}, we then have
    \begin{align}
 \|\hv(t)\|^2+\hnu\int_0^t|A\hv(s)|^2\, ds+&\mu\int_0^t\|P_N\hv(s)\|^2\, ds\leq\frac{2}{\hnu}\left(\frac{c_A^2V_2^2}{\mu^2}+1\right)V_2^2t\notag,
    \end{align}
as desired.

\end{proof}

We obtain the following corollary of Lemma \ref{lem:hv:L2} and Lemma \ref{lem:hv:H1} with an application of the Cauchy-Schwarz inequality. It is invoked in the proof of Proposition \ref{prop:Grashof:sensitivity}.

\begin{Cor}\label{cor:hv:H1}
Under the assumptions of Lemma \ref{lem:hv:H1}, we have
    \begin{align}\label{est:H1:hv:cor}
    \limsup_t\fint_0^t|A\hv(v_*)||\hv(v_*)|\,dt\leq O\left(\left(\frac{V_2(v_*)^2}{\mu^2}+1\right)^{1/2}\frac{V_2(v_*)^2}{\hnu\mu}\right),
    \end{align}
where $V_2(v_*):=|Av_*|$.
\end{Cor}

\subsection{Lipschitz dependence on viscosity}

Finally, we establish a Lipschitz estimate for the assimilated and sensitivity variable relative to $\hnu$. It is invoked to establish Lipschitz dependence of $\sfW$ on $\hnu$ (see Remark \ref{rmk:stability}).

\begin{Lem}\label{lem:stability:nudge}
Suppose that
    \begin{align}\label{cond:stability:nudge}
             \mu\geq\frac{c_L^2}{\unu}V_1^2,\quad N\geq\frac{2c_LV_1}{\unu}
    \end{align}
Then
    \begin{align}\label{est:stability:nudge}
        |z(t)|^2+\frac{\unu}2\int_0^t\|z\|^2\leq \frac{2V_1^2}{\unu}|\De\hnu|^2t
    \end{align}
and
    \begin{align}\label{est:stability:nudge:L2}
         |z(t)|^2\leq\frac{4V_1^2}{\unu^2}|\De\hnu|^2
    \end{align}
\end{Lem}

\begin{proof}
 Let $v_j=v(\cdot;\hnu_j)$, $z=v_1-v_2$, and $\De\hnu=\hnu_1,\hnu_2$. Then
    \begin{align}
        \bdy_tz+\hnu_1Az+\mu P_Nz+B(z,z)+DB(v_2)z=-(\De\hnu)Av_2.\notag
    \end{align}
Upon taking the $L^2$-inner product with $z$, we obtain
    \begin{align}\notag
        \frac{1}2\frac{d}{dt}|z|^2+\hnu_1\|z\|^2+\mu|P_Nz|^2=-b(z,v_2,z)-(\De\nu)(Av_2,z).
    \end{align}
Observe that
    \begin{align}
        |b(z,v_2,z)|&\leq c_L\|z\||z|\|v_2\|\leq \frac{c_L^2}{\unu}\|v_2\|^2|z|^2+\frac{\unu}4\|z\|^2\notag
        \\
        &\leq \frac{c_L^2}{\unu}\|v_2\|^2|P_Nz|^2+\frac{c_L^2}{\unu}\|v_2\|^2|Q_Nz|^2+\frac{\unu}4\|z\|^2\notag
        \\
        &\leq \frac{c_L^2}{\unu}\|v_2\|^2|P_Nz|^2+\frac{c_L^2}{\unu N^2}\|v_2\|^2\|z\|^2+\frac{\unu}4\|z\|^2\notag
        \\
        |\De\hnu||(Av_2,z)|&\leq \frac{\|v_2\|^2}{\unu}|\De\hnu|^2+\frac{\unu}4\|z\|^2\notag
    \end{align}
Thus
    \begin{align}\notag
        \frac{d}{dt}|z|^2+\left(\unu-\frac{2c_L^2}{\unu N^2}\|v_2\|^2\right)\|z\|^2+2\left(\mu-\frac{c_L^2}{\unu}\|v_2\|^2\right)|P_Nz|^2&\leq \frac{2\|v_2\|^2}{\unu}|\De\hnu|^2
    \end{align}
Then \eqref{est:stability:nudge} follows upon applying \eqref{cond:stability:nudge}, then integrating over $[0,t]$, while \eqref{est:stability:nudge:L2} follows from Gr\"onwall's inequality.
\end{proof}

\begin{Lem}\label{lem:stability:sensitivity}
Suppose that
    \begin{align}\label{cond:stability:sensitivity}
        N^2\geq \frac{8\sqrt{2}c_LV_1}{\unu},\quad \mu\geq\frac{32c_L^2V_1^2}{\unu}.
    \end{align}
Then
    \begin{align}\label{est:stability:sensitivity}
        |\hz(t)|^2\leq \frac{4}{\unu}\left(\frac{8c_L^2\hV_1\hV_0V_1^2}{\unu^2}+\frac{V_1^2}{\unu^2}+\hV_1^2\right)|\De\hnu|^2t.
    \end{align}
\end{Lem}

\begin{proof}
Let $v_j=v(\cdot;\hnu_j)$ and $\hv_j=\hv(v_j)$. Let $\hz=\hv_1-\hv_2$ and $z=v_1-v_2$. Then
    \begin{align}\notag
        \bdy_t\hz+\hnu_1A\hv_1-\hnu_2A\hv_2+\mu P_N\hz+DB(v_1)\hv_1-DB(v_2)\hv_2=-Az
    \end{align}
Let $\De\hnu=\hnu_1-\hnu_2$ and observe that
    \begin{align}\notag
            \hnu_1A\hv_1-\hnu_2A\hv_2=\hnu_1 A\hz+(\De\hnu)A\hv_2.
    \end{align}
Also,
    \begin{align}\notag
        DB(v_1)\hv_1-DB(v_2)\hv_2&=DB(z)\hz+DB(v_2)\hz+DB(z)\hv_2.
    \end{align}
Upon taking the $L^2$-inner product with $\hz$, we obtain
    \begin{align}
        \frac{1}2\frac{d}{dt}|\hz|^2+\hnu_1\|\hz\|^2+\mu|P_N\hz|^2&=-b(\hz,z,\hz)-b(\hz,v_2,\hz)-b(z,\hv_2,\hz)-b(\hv_2,z,\hz)\notag
        \\
        &\quad -(\De\hnu)(A\hv_2,\hz)-(Az,\hz)\notag.
    \end{align}
It follows that
    \begin{align}
        |b(\hz,z,\hz)|&\leq c_L\|\hz\||\hz|\|z\|\leq \frac{2c_L^2\|z\|^2}{\unu}|\hz|^2+\frac{\hnu_1}8\|\hz\|^2\notag
        \\
        &\leq \frac{2c_L^2\|z\|^2}{\unu}|P_N\hz|^2+ \frac{2c_L^2\|z\|^2}{\unu}|Q_N\hz|^2+\frac{\unu}8\|\hz\|^2\notag
        \\
        &\leq \frac{2c_L^2\|z\|^2}{\unu}|P_N\hz|^2+ \frac{2c_L^2\|z\|^2}{\unu N^2}\|\hz\|^2+\frac{\unu}8\|\hz\|^2.\notag
    \end{align}
Similarly, we have
    \begin{align}\notag
        |b(\hz,v_2,\hz)|&\leq \frac{2c_L^2\|v_2\|^2}{\unu}|P_N\hz|^2+ \frac{2c_L^2\|z\|^2}{\unu N^2}\|\hz\|^2+\frac{\unu}8\|\hz\|^2.
    \end{align}
On the other hand, we see that $b(z,\hv_2,\hz)=-b(z,\hz,\hv_2)$. Thus,
    \begin{align}
        |b(z,\hv_2,\hz)|&\leq c_L\|z\|^{1/2}|z|^{1/2}\|\hz\|\|\hv_2\|^{1/2}|\hv_2|^{1/2}\notag
        \\
        &\leq \frac{2c_L^2\|\hv_2\||\hv_2|}{\unu}\|z\||z|+\frac{\unu}8\|\hz\|^2\notag.
    \end{align}
Similarly, since $b(\hv_2,z,\hz)=-b(\hv_2,\hz,z)$, we have
    \begin{align}
        |b(\hv_2,z,\hz)|&\leq\frac{2c_L^2\|\hv_2\||\hv_2|}{\unu}\|z\||z|+\frac{\unu}8\|\hz\|^2.\notag
    \end{align}

We also have
    \begin{align}
        |\De\nu||(A\hv_2,\hz)|&\leq |\De\nu|\|\hv_2\|\|\hz\|\leq \frac{2|\De\nu|^2}{\unu}\|\hv_2\|^2+\frac{\unu}8\|\hz\|^2\notag
        \\
        |(Az,\hz)|&\leq \|z\|\|\hz\|\leq \frac{2}{\unu}\|z\|^2+\frac{\unu}8\|\hz\|^2\notag.
    \end{align}
Thus
    \begin{align}
        \frac{d}{dt}|\hz|^2+&\left(\frac{\unu}2-\frac{8c_L^2}{\unu N^2}\|z\|^2\right)\|\hz\|^2+2\left(\mu-\frac{8c_L^2}{\unu}\|z\|^2\right)|P_N\hz|^2\notag
        \\
        &\leq \frac{8c_L^2\|\hv_2\||\hv_2|}{\unu}\|z\||z|+\frac{4}{\unu}\|z\|^2+\frac{4\|\hv_2\|^2}{\unu}|\De\nu|^2.\notag
    \end{align}
Since $\|z\|^2\leq 2(\|v_1\|^2+\|v_2\|^2)\leq 4V_1^2$, we apply \eqref{cond:stability:sensitivity} to deduce
    \begin{align}\notag
         \frac{d}{dt}|\hz|^2+\frac{\unu}4\|\hz\|^2\leq \frac{8c_L^2\hV_1\hV_0}{\unu}\|z\||z|+\frac{4}{\unu}\|z\|^2+\frac{4\hV_1^2}{\unu}|\De\nu|^2.
    \end{align}
 Integrating over $[0,t]$, applying the Cauchy-Schwarz inequality, then Lemma \ref{lem:stability:nudge} implies
    \begin{align}
        &|\hz(t)|^2+\frac{\unu}4\int_0^t\|\hz(s)\|^2\,ds\notag
        \\
        &\leq  \frac{8c_L^2\hV_1\hV_0}{\unu}\left(\int_0^t\|z\|^2\,ds\right)^{1/2}\left(\int_0^t|z|^2\,ds\right)^{1/2}+\frac{4}{\unu}\int_0^t\|z\|^2\,ds+\frac{4\hV_1^2}{\unu}|\De\nu|^2t\notag
        \\
        &\leq \frac{32c_L^2\hV_1\hV_0V_1^2}{\unu^3}|\De\hnu|^2t+\frac{16V_1^2}{\unu^2}|\De\hnu|t+\frac{4\hV_1^2}{\unu}|\De\hnu|^2t,\notag
    \end{align}
as desired.
\end{proof}

\begin{Lem}\label{lem:W:Lipschitz}
Let $u$ be a strong solution of \eqref{eq:nse}. Given $\hnu_1,\hnu_2\in\sA$, denote $v_j=v_j(\cdot;v_0,P_Nu)$ the unique strong solution of \eqref{eq:nse:nudge} corresponding to initial data $v_0$ given by \eqref{def:v0} and $\hv_j=\hv_j(v_j)$ denote the unique strong solution of \eqref{eq:sensitivity} corresponding to $v_j$. Then there exists a constant $C_\sfW$ such that
    \begin{align}\label{est:Lipschitz}
        |\sfW(\hnu_1;\tau)-\sfW(\hnu_2;\tau)|\leq C_{\sfW}|\hnu_1-\hnu_2|.
    \end{align}
\end{Lem}

\begin{proof}
Let $z=v_1-v_2$ and $\hz=\hv_1-\hv_2$. Then
    \begin{align}\notag
        \sfW(\hnu_1)-\sfW(\hnu_2)&=\int_0^\tau((Av_1,\hv_1)-(Av_2,\hv_2))\,dt=\int_0^\tau(Az,\hv_1)\,dt+\int_0^\tau(Av_2,\hz)\,dt\notag.
    \end{align}
By the Cauchy-Schwarz inequality, we see that
    \begin{align}&\notag
        | \sfW(\hnu_1)-\sfW(\hnu_2)|\leq \int_0^\tau\|z\|\|\hv_1\|\,dt+\int_0^\tau|Av_2||\hz|\,dt\notag
        \\
        &\leq\left(\int_0^\tau\|z\|^2\,dt\right)^{1/2}\left(\int_0^\tau\|\hv_1\|^2\,dt\right)^{1/2}+\left(\int_0^\tau|Av_2|^2\,dt\right)^{1/2}\left(\int_0^\tau|\hz|^2\,dt\right)^{1/2}\notag.
    \end{align}
From Lemma \ref{lem:w:L2}, Lemma \ref{lem:hv:H1}, Lemma \ref{lem:stability:nudge}, and Lemma \ref{lem:stability:sensitivity}, we deduce
    \begin{align}\notag
        |\sfW(\hnu_1)-\sfW(\hnu_2)|&\leq \frac{2V_1\hV_1}{\unu}|\De\hnu|\tau+\frac{CV_2}{\unu}|\De\hnu|\tau,
    \end{align}
where $C:=\left(\frac{8c_L^2\hV_1\hV_0V_1^2}{\unu^2}+\frac{V_1^2}{\unu^2}+\hV_1^2\right)$. Setting $C_\sfW:= \frac{2V_1\hV_1}{\unu}\tau+\frac{CV_2}{\unu}\tau$ completes the proof.
\end{proof}

\section{Existence and Uniqueness of steady states to the sensitivity equation}

Define
    \begin{align}\label{def:L:K}
        L_\mu:=\hnu A+\mu P_N,\qquad K:=DB(v^*).
    \end{align}
and the associated bilinear form on $V$,
    \begin{align}\label{def:form:a}
        a(\psi,\ph):=\hnu\lpp\psi,\ph\rpp+\mu(P_N\psi,P_N\ph)+b(v^*,\psi,\ph)+b(\psi,v^*,\ph).
    \end{align}

\begin{Lem}\label{lem:Grashof:LM}
Under the assumptions of Proposition \ref{prop:Grashof}, Proposition \ref{prop:Grashof:nudge}, and Proposition \ref{prop:Grashof:sensitivity}, $a$ is bounded on $V\times V$ and coercive:
    \begin{align}\label{est:coercive}
        a(\psi,\psi)\geq\frac12(L_\mu\psi,\psi)=\frac{\hnu}2\|\psi\|^2+\frac{\mu}2|P_N\psi|^2,
    \end{align}
for all $\psi\in V$. In particular, there exists a unique $\hv^*=\hv^*(\hnu)\in V$, satisfying
    \begin{align}\label{eq:hvstar}
        a(\hv^*,\ph)=-(Av^*,\ph),
    \end{align}
for all $\ph\in V$.
\end{Lem}

\begin{proof}
Recall that $b(v^*,\psi,\psi)=0$ by \eqref{eq:energy:identity}. Then \eqref{est:trilinear:a}, our assumptions on $\mu, N,\hnu$, and Young's inequality imply
    \begin{align}\notag
        |b(\psi,v^*,\psi)|\leq c_L\|v^*\||\psi|\|\psi\|\leq\frac{\hnu}2\|\psi\|^2+\frac{\mu}4|P_N\psi|^2.\notag
    \end{align}
This proves boundedness. On the other hand, we also see that
    \begin{align}\notag
        (L_\mu\psi,\psi)=\hnu\|\psi\|^2+\mu|P_N\psi|^2.
    \end{align}
Thus, 
    \begin{align}\notag
             |b(\psi,v^*,\psi)|\leq \frac{1}2(L_\mu\psi,\psi),
    \end{align}
which implies \eqref{est:coercive}. The Lax-Milgram theorem then applies.
\end{proof}

\bibliographystyle{plain}

\vfill

\begin{minipage}[t]{0.8\textwidth}
\noindent Vincent R. Martinez$^\dagger$\\
{\footnotesize
Department of Mathematics \& Statistics\\
CUNY Hunter College \\
Department of Mathematics \\
CUNY Graduate Center \\
Department of Computing \& Mathematical Sciences\\
California Institute of Technology\\
Web: \url{http://math.hunter.cuny.edu/vmartine/}\\
Email: \url{vrmartinez@hunter.cuny.edu},\url{vrm@caltech.edu}\\
}
\end{minipage}

\begin{minipage}[t]{0.6\textwidth}
\noindent Sarah Strikwerda\\
{\footnotesize
Department of Mathematics\\
Colorado State University\\
Web: \url{https://sites.google.com/view/sarahstrikwerda}\\
Email: \url{s.l.strikwerda@colostate.edu}
}
\end{minipage}%
\begin{minipage}[t]{0.8\textwidth}
\noindent Xiang Wan \\
{\footnotesize
Department of Mathematics and Statistics\\
Loyola University Chicago\\
Web: \url{https://www.luc.edu/math/profiles/wanxiang.shtml/}\\
Email: \url{xwan1@luc.edu}
}
\end{minipage}

\vfill

$\dagger$ corresponding author

\end{document}